\documentclass{amsart}

\usepackage[top=1in, bottom=1in, left=1in, right=1in]{geometry}

\usepackage{mathrsfs}
\usepackage{enumerate}
\usepackage{enumitem}
\usepackage{amsthm}
\usepackage{thmtools}
\usepackage{amsmath}
\usepackage{stmaryrd}

\usepackage{amssymb,amsxtra,amsfonts}
\usepackage{graphicx}

\usepackage[normalem]{ulem}

\usepackage[unicode=true,colorlinks=true,linkcolor=blue,urlcolor=blue, citecolor=red, pagebackref=True]{hyperref}

\usepackage{amscd}
\usepackage{xcolor}
\usepackage{ifpdf} 
\usepackage[all]{xy}
\ifpdf
\else
\fi

\usepackage{xcolor}

\makeatletter
\@namedef{subjclassname@2020}{%
 \textup{2020} Mathematics Subject Classification}
\makeatother

\numberwithin{equation}{section}
\numberwithin{table}{section}

\declaretheorem[style=plain,parent=section]{theorem}
\declaretheorem[style=plain,sibling=theorem]{corollary}
\declaretheorem[style=plain,sibling=theorem]{lemma}
\declaretheorem[style=plain]{claim}
\declaretheorem[name={Claim}, style=plain, unnumbered]{claimNoNum}
\declaretheorem[name={Claim A}, style=plain, unnumbered]{claimA}
\declaretheorem[name={Claim B}, style=plain, unnumbered]{claimB}

\declaretheorem[style=plain,sibling=theorem]{proposition}

\declaretheorem[style=plain,sibling=theorem]{question}

\declaretheorem[style=definition,sibling=theorem]{definition}

\declaretheorem[style=definition, qed=\hfill $\diamond$, sibling=definition]{example}
\declaretheorem[style=remark,sibling=theorem]{remark}

\theoremstyle{definition}
\newtheorem*{exampleth}{Proof}
\newenvironment{subproof}{\begin{exampleth}}{\hfill $\diamond$ \end{exampleth}}

\DeclareMathOperator{\Hom}{Hom}

\DeclareMathOperator{\Supp}{Supp}
\newcommand{\Q}{\mathbb{Q}}
\newcommand{\R}{\mathbb{R}}

\newcommand{\CC}{\mathbb{C}}
\newcommand{\Z}{\mathbb{Z}}
\newcommand{\N}{\mathbb{N}}

\newcommand{\cF}{\mathcal{F}}

\newcommand{\cO}{\mathcal{O}}
\newcommand{\cP}{\mathcal{P}}

\newcommand{\cS}{\mathcal{S}}

\newcommand{\mfkm}{\mathfrak{m}}

\DeclareMathOperator{\lcm}{lcm}

\newcommand{\tv}{\mathcal{X}} 

\newcommand{\NP}{\cN\cP} 

\newcommand{\NF}{\cN\cF}

\providecommand {\from}{{\colon}}

\providecommand{\spec}{\operatorname{Spec}}

\providecommand{\Hom}{\operatorname{Hom}}

\providecommand{\codim}{\operatorname{codim}}

\providecommand{\rk}{\operatorname{rk}}

\DeclareMathOperator{\Trop}{Trop} 
\DeclareMathOperator{\ptrop}{\Trop_{>0}}

\newcommand{\relo}[1]{\ensuremath{{#1}^{\circ}}} 

\newcommand{\init}{\ensuremath{\operatorname{in}}}
\newcommand{\initwf}[2]{\ensuremath{\init_{#1}(#2)}} 

\newcommand{\rDG}[1]{\ensuremath{\partial_{#1}\Gamma}} 
\newcommand{\rDGi}[2]{\ensuremath{\partial_{#1}{#2}}} 

\newcommand \Rp {\ensuremath{\R_{\geq 0}}}

\newcommand \Ccurve {\mathcal{C}}

\newcommand{\cN}{\mathcal{N}}

\newcommand{\nodesT}[1]{\ensuremath{
    {V}^{\circ}({#1})}} 
\newcommand{\leavesT}[1]{\ensuremath{\partial\,#1}} 

\newcommand{\roottree}{r}

\newcommand{\nodesTevRoot}[1]{\ensuremath{\partial_{#1}\Gamma}}

\newcommand{\pag}[1]{\ensuremath{\partial_{#1}\Gamma}}

\newcommand{\starT}[2]{\ensuremath{\operatorname{Star}_{#1}(#2)}} 

\newcommand{\wtuv}[2]{\ensuremath{\ell_{#1,#2}}} 
\newcommand{\du}[1]{\ensuremath{d_{#1}}} 

\newcommand{\valv}[1]{\ensuremath{\delta_{#1}}} 

\newcommand{\Nw}[1]{\ensuremath{N(\partial #1)}} 
\newcommand{\Mw}[1]{\ensuremath{M(\partial #1)}} 
\newcommand{\Nwp}[1]{\ensuremath{N(#1)}} 
\newcommand{\Mwp}[1]{\ensuremath{M(#1)}} 

\newcommand{\ww}{\ensuremath{\omega}} 
\newcommand{\wu}[1]{\ensuremath{w_{#1}}} 

\newcommand{\wtuvG}[3]{\ensuremath{\wtuv{#1}{#2}^{#3}}}

\newcommand{\wtNve}[2]{\ensuremath{m_{#1,#2}}} 
\newcommand{\wtNveL}[3]{\ensuremath{m_{#1,#2,#3}}} 
\newcommand{\zu}{\ensuremath{\underline{z}}}

\newcommand{\zexp}[1]{\ensuremath{\zu^{#1}}} 

\newcommand{\fvi}[2]{\ensuremath{f_{#1,#2}}} 

\newcommand{\gvi}[2]{\ensuremath{g_{#1,#2}}} 
\newcommand{\fgvi}[2]{\ensuremath{F_{#1,#2}}} 

\newcommand{\hvi}[2]{\ensuremath{h_{#1,#2}}} 

\newcommand{\sG}[1]{\ensuremath{\mathcal{S}(#1)}}

\newcommand{\cvei}[3]{\ensuremath{c_{#1,#2,#3}}} 
\newcommand{\cvi}[2]{\ensuremath{c_{#1,#2}}}

\newcommand{\simplex}[1]{\operatorname{\Delta}_{#1}}

\newcommand{\conv}{\operatorname{conv}}

\newcommand{\baryc}[2]{\operatorname{Bar}(#1;#2)}

\newcommand{\emb}{\ensuremath{\operatorname{\iota}}} 

\newcommand{\distGuv}[3]{\ensuremath{\operatorname{dist}_{#1}(#2,#3)}} 

\newcommand{\projT}[1]{\ensuremath{p_{#1}}}

\newcommand{\boxedo}[1]{\ensuremath{\boxed{#1}}}

\begin{document}
\title[Local tropicalizations of splice type surface singularities]{Local tropicalizations of splice type surface singularities}
\author[M.A. Cueto, P. Popescu-Pampu, D. Stepanov]{Maria Angelica Cueto, Patrick Popescu-Pampu${}^{\S}$ and Dmitry Stepanov\\
{(with an appendix by Jonathan Wahl)}}
\thanks{${\S}$ \emph{Corresponding author}}

\date{}
\keywords{Surface singularities, complete intersection singularities, tropical geometry, Newton non-degeneracy}
\subjclass[2020]{\emph{Primary}:  14B05, 14T90, 32S05; \emph{Secondary}: 14M25, 57M15}

\begin{abstract} 
   Splice type surface singularities were introduced by Neumann and Wahl 
   as a generalization of the class of Pham-Brieskorn-Hamm complete intersections of   
   dimension two. Their construction depends on a weighted tree called a splice diagram. 
  In this paper, we study these singularities from the tropical viewpoint. We characterize their local 
  tropicalizations as the cones over the appropriately embedded associated splice diagrams. 
  As a corollary, we reprove some of Neumann and Wahl's earlier results on these singularities 
  by purely tropical methods, and show that splice type surface singularities are Newton 
  non-degenerate complete intersections in the sense of Khovanskii. We also confirm that under suitable coprimality conditions on its weights, the diagram can be uniquely recovered from the local tropicalization.
  
  As a corollary of the Newton non-degeneracy property, we obtain an alternative proof of a recent theorem of de Felipe, Gonz\'alez P\'erez and Mourtada, stating 
  that embedded resolutions of any plane curve singularity can be achieved by a single toric morphism, 
  after re-embedding the ambient smooth surface germ in a higher-dimensional 
  smooth space. 
  The paper ends with an appendix by Jonathan Wahl, providing a criterion of regularity of a sequence in a ring of convergent power series, given the regularity of an associated sequence of initial forms. 
 \end{abstract}

\begin{center}
    {\bf (Published online in Mathematische Annalen on 19 December 2023. https://doi.org/10.1007/s00208-023-02755-y)}
    \bigskip
\end{center}

\dedicatory{To Walter Neumann, on the occasion of his 75\textsuperscript{th}~birthday.}

\maketitle

\vspace{-2ex}
\section{Introduction}\label{sec:introduction}

\emph{Splice diagrams} are finite trees with half-edges weighted by integers and with nodes (internal vertices) decorated by $\pm$ signs.  If the half-edge weights around each node are pairwise coprime, we say that the splice diagram is \emph{coprime}. This class of weighted trees was first introduced by
Siebenmann~\cite{S 80} in 1980 to encode graph manifolds which are integral homology spheres. Coprime splice diagrams with only + node decorations and positive half-edge weights were used by Eisenbud and Neumann in~\cite{EN 85} to study special kinds of links in integral homology spheres, in particular those corresponding to curves on normal surface singularities with integral homology sphere links. One of the main theorems of~\cite{EN 85} states that such integral homology spheres are described by positively-weighted coprime splice diagrams satisfying the \emph{edge determinant condition}, namely, that the product of the two weights associated to any fixed internal edge must be greater than the product of the weights of the neighboring half-edges.

Interesting isolated surface singularities arise from splice diagrams. For example, complete intersections of \emph{Pham-Brieskorn-Hamm} hypersurface singularities are associated to star splice diagrams  (i.e., those with a single node). As recognized by Hamm in~\cite[\S 5]{H 69} and \cite{H 72}, in order to determine an isolated singularity in $\CC^n$, all maximal minors of the coefficient matrix $(c_{ij})_{i,j}$ of each  polynomial $f_i := \sum_{j} c_{ij} z_j^{a_j}$ in the Brieskorn system must be non-zero. In turn, work of Neumann~\cite{N 83} shows that universal abelian covers of quasi-homogeneous complex normal surface singularities with rational homology sphere links are complete intersections of Pham-Brieskorn-Hamm hypersurface singularities. 
       
In 2002, Neumann and Wahl~\cite{NW 02} extended this family of complete intersections by defining \emph{splice type surface singularities} associated to splice diagrams whose weights satisfy a special arithmetic property called the \emph{semigroup condition}. These singularities are defined by explicit  splice type systems of convergent power series near the origin, whose coefficients satisfy generalizations of Hamm's maximal minors conditions. Splice type surface singularities and the related class of \emph{splice quotients} (determined by diagrams subject to an additional congruence condition) have been further studied by both authors in~\cite{NW 05bis, NW 05, NW 10},  and by Lamberson, N\'emethi, Okuma and Pedersen in~\cite{L 09, N 12,  NO 08, NO 09, O 06, O 08, O 10, O 15, P 11,  P 18}. For more details, we refer the reader to the surveys~\cite{N 07, O 12, W 06, W 22}.

The present paper uses tropical geometry techniques  to study splice type systems with $n$ leaves associated to splice diagrams satisfying the edge determinant and semigroup conditions. Our first main result recovers and strengthens a  central theorem from~\cite{NW 05bis} (see~\autoref{thm:spliceicis}).  More precisely:

    \begin{theorem}\label{thm:main1}
         Splice type systems are Newton non-degenerate complete 
         intersection systems of equations.  
         The associated splice type singularities are isolated, irreducible and   
         not contained in any coordinate subspace of the corresponding ambient space $\CC^n$.
     \end{theorem}

\noindent
This statement plays a major role in the proof of the Neumann-Wahl Milnor fiber conjecture for splice type singularities with integral homology sphere links~\cite{NW 05} obtained by the present authors. For an overview of  this proof, we refer the reader to~\cite{CPPS 22}. Motivated by \autoref{thm:main1}, we formulate \autoref{ques:minNNDemb}.

The notion of Newton non-degeneracy (in the sense of Kouchnirenko~\cite{K 76} and Khovanskii~\cite{K 77}) is clossely related to the notion of the initial form of a series relative to a weight vector, which lies  at the core of tropical geometry (see~\autoref{sec:local-trop}, in particular~\autoref{def:initwf}). A regular sequence of convergent power series $(f_1,\ldots, f_s)$ in 
$\CC\{z_1, \dots, z_n\}$ defining a germ $(X,0) \hookrightarrow \CC^n$ is a 
\emph{Newton non-degenerate complete intersection system} if for each weight vector 
$\wu{}$ with positive entries, the associated initial forms $(\initwf{\wu{}}{f_i})_i$ determine 
hypersurfaces  of the algebraic torus $(\CC^*)^n$  
whose sum is a normal  crossings divisor in the neighborhood of their 
intersection.  This condition is automatically  satisfied 
whenever the intersection is empty. Surprisingly, not many examples of Newton 
non-degenerate complete intersection systems are known in codimension two or higher. 
\autoref{thm:main1} contributes a large class of examples of such systems.

       Newton non-degeneracy enables the resolution of the germ $(X,0) \hookrightarrow \CC^n$  by a single toric morphism. Indeed, works of Varchenko~\cite{V 76} (for hypersurfaces) and Oka~\cite[Chapter III, Theorem (3.4)]{O 97} (for complete intersections) show that such a morphism may be defined by  a regular subdivision of the positive orthant refining the dual fan of the Newton polyhedron of each function $f_i$. Furthermore, the complete dual fan is not needed to achieve a resolution. Indeed,~\cite[Theorem III.3.4]{O 97} allows us to restrict to the subfan corresponding to the orbits intersecting the strict transform of the given germ.

       The support of this subfan depends only on the ideal defining the germ $(X,0)$ but not
       on the particular generators $f_i$. 
       It is the finite part of the so-called \emph{local tropicalization} of the 
       embedding $(X,0) \hookrightarrow \CC^n$. This local version of the standard notion of 
       tropicalization of a subvariety of an algebraic torus was first introduced by the last two 
       authors in~\cite{PPS 13} as a tool to study arbitrary subgerms of $\CC^n$        or, more generally, arbitrary morphisms from analytic or formal germs 
       to germs of toric varieties. 
       
       As was mentioned earlier, splice diagrams record topological information about the link of 
       splice type singularities. Indeed, starting from a normal surface singularity with  
       a rational homology sphere link, Neumann and Wahl~\cite{NW 02} build a 
       splice diagram  which determines the dual tree of the minimal normal crossings resolution 
       (up to a finite ambiguity) whenever this graph is not a star tree. 
       This diagram is 
       homeomorphic  to the dual tree: it is obtained  by disregarding all bivalent vertices of the dual tree. 
       It satisfies the edge determinant condition, but not necessarily the semigroup 
       or the congruence conditions. When $(X,0)  \hookrightarrow \CC^n$ has an integral 
       homology sphere link, the ambiguity disappears and the dual tree is completely determined by the splice diagram.

       Given a splice diagram $\Gamma$ with $n$ leaves, the construction of Neumann and Wahl associates a weight vector $\wu{u}\in (\Z_{>0})^n$ to each vertex $u$ of $\Gamma$. These vectors induce a piecewise linear embedding of $\Gamma$ into the standard simplex in $\R^n$ after appropriate normalization. Our second main result shows the close connection between $\Gamma$, the local tropicalization of the associated splice type system in $\CC^n$ and its resolution diagrams:      
       
       \begin{theorem}\label{thm:main2}
          Let $\Gamma$ be a splice diagram satisfying the semigroup condition and 
           let $(X,0)  \hookrightarrow \CC^n$ be the germ defined by an associated 
           splice type system. Then, the finite local tropicalization of $X$ is the cone over an embedding of $\Gamma$ in $\R^n$. Furthermore, in the coprime case,  $\Gamma$ can be uniquely recovered from this fan.
       \end{theorem}

       \autoref{thm:main2} shows that the link at the origin of the local tropicalization of $(X,0)  \hookrightarrow \CC^n$ (obtained by intersecting the fan with the $(n-1)$-dimensional sphere) is homeomorphic to the splice diagram $\Gamma$. 
       To the best of our knowledge, \emph{this is the first tropical interpretation of Siebenmann's splice diagrams}. In this spirit, we view Siebenmann's paper~\cite{S 80} as a precursor to tropical geometry (for others, see~\cite[Chapter 1]{MS 15}).

              Our method to characterize the local tropicalization is different from the general one 
       discussed in Oka's book~\cite{O 97} and described briefly above. Namely, we do not use the Newton polyhedra of the collection of series defining the germ. 
       Instead, we use a ``mine-sweeping'' approach, using successive stellar subdivisions of the standard simplex in $\R^n$ dictated by the splice diagram, in order to remove relatively open cones in the positive orthant avoiding the local 
       tropicalization.

       Once the local tropicalization is determined (via~\autoref{thm:main2}), a simple computation confirms the Newton non-degeneracy of the system.  In turn, by analyzing the local tropicalizations of the intersections of the germ with the coordinate subspaces of $\CC^n$, we conclude that $(X,0)$ is an isolated complete intersection surface singularity, thus completing the proof of~\autoref{thm:main1}.
       
   As a consequence of ~\autoref{thm:main1}, we provide an alternative proof of the main theorem of de Felipe, Gonz\'alez P\'erez and Mourtada \cite{FGM 21}, stating that \emph{any germ of a reduced plane curve may be resolved by one toric modification after re-embedding its ambient smooth germ of surface into a higher-dimensional germ $(\CC^n,0)$} (see \autoref{cor:onetoricplanecurve}). The first theorem of this kind was proved by Goldin and Teissier \cite{GT 00} for \emph{irreducible} germs of plane curves.

\smallskip
       
Our paper is organized as follows. In~\autoref{sec:splicetypenotation}, 
we review the definitions and main properties of splice diagrams, splice type systems 
and end-curves associated to rooted splice diagrams, following~\cite{NW 05bis, NW 05}. 
\textcolor{blue}{Sections}~\ref{sec:local-trop} and~\ref{sec:nnondeg} include background 
material about local tropicalizations and Newton non-degeneracy. 
In~\autoref{sec:comb-embedd-splice} 
we show how to embed a given splice diagram $\Gamma$ with $n$ leaves into the standard 
$(n-1)$-simplex in $\R^n$, and we highlight various convexity properties of this embedding. 
The proof of the first part of~\autoref{thm:main2} is discussed in~\autoref{sec:local-trop-newm}, 
while~\autoref{thm:main1} is proven in~\autoref{sec:NewtonND}. 
\autoref{sec:recov-splice-diagr} characterizes local tropicalizations of splice type 
systems defined by a coprime splice diagram and shows how to recover the diagram from 
the tropical fan, thus yielding the second part of~\autoref{thm:main2}. Finally,~\autoref{sec:natur-splice-type} discusses the dependency of the construction of splice type systems on the choice of admissible monomials for arbitrary splice diagrams, in the spirit of~\cite[Section 10]{NW 05bis}. 

\autoref{sec:appendixL3.3}, written by Jonathan Wahl, includes a proof of~\cite[Lemma 3.3]{NW 05bis} that was absent from the literature. This result confirms that given a finite sequence $(f_1,\ldots, f_s)$ in $\CC\{z_1,\ldots, z_n\}$ and a fixed positive integer vector $\wu{}\in {(\Z_{>0})}^n$, the regularity of the sequence $(\initwf{\wu{}}{f_1}, \ldots, \initwf{\wu{}}{f_s})$ of initial forms ensures that the original sequence is regular, and furthermore, that the $\wu{}$-initial ideal must be generated by the sequence of initial forms. This statement can be used to determine if a given $\wu{}$ lies in the local tropicalization of the germ defined by the vanishing of the input sequence (see~\autoref{cor:initialForms}), providing an alternative proof to part of~\autoref{thm:main2}.

\section{Splice diagrams, splice type systems and end-curves from rooted splice diagrams}
\label{sec:splicetypenotation}

In this section, we recall the notions of splice diagram and splice type systems associated to them. The definitions follow closely the work of Neumann and Wahl~\cite{NW 05bis, NW 05}.
\medskip

We start with some basic terminology and notations about trees:

\begin{definition} \label{def:vocabtrees}
 A \emph{tree} is a finite connected graph 
   with no cycles and at least one vertex. 
The \emph{star} of a vertex {$v$} of the tree $\Gamma$ is the set $\boxedo{\starT{\Gamma}{v}}$  of edges adjacent to $v$.   The \emph{valency}  of $v$ is the cardinality of $\starT{\Gamma}{v}$, which we denote by $\boxedo{\valv{v}}$. 
A \emph{node} of a tree is a vertex $v$ whose valency  is greater than one, 
whereas a \emph{leaf} is a one-valent vertex. We denote the set of nodes 
of $\Gamma$ by $\boxedo{\nodesT{\Gamma}}$ and its set of 
leaves by $\boxedo{\leavesT{\Gamma}}$. 
\end{definition}

\noindent When the ambient tree is understood from context, we remove it from the notation and simply write $\starT{}{v}$.

\begin{remark}\label{rm:geodesics} 
Endowing a tree with a metric allows us to consider geodesics on it and distances between vertices. Whenever these notions are invoked, it is understood that each edge of the tree has length one.
\end{remark}

\begin{definition}\label{def:convexHull} 
    Given a subset $W= \{p_1,\ldots, p_k\}$ of vertices of the tree $\Gamma$, we 
denote by $\boxedo{[W]}$ or $\boxedo{[p_1,\ldots, p_k]}$  the subtree of $\Gamma$ spanned by these points. We  call it the \emph{convex hull} of the set $\{p_1,\ldots, p_k\}$ inside $\Gamma$. 
For example, $\Gamma = [\leavesT{\Gamma}]$. 
\end{definition}

Splice diagrams are special kinds of trees enriched with weights around all  nodes, as we now describe:

\begin{definition} 
   A \emph{splice diagram} is a pair $(\Gamma,\{\du{v,e}\}_{v,e})$, where
    $\Gamma$ is a tree without  valency-two vertices, 
   with at least one node,
   and decorated with a weight function  
   on the star of each node $v$ of $\Gamma$, denoted by
  \[
  \starT{\Gamma}{v} \to \Z_{>0}\colon \;  e\mapsto \du{v,e}.
  \]
 We call $\boxedo{\du{v,e}}$ the \emph{weight of $e$ at $v$}. 
 If $u$ is any other vertex of $\Gamma$ such that $e$ lies in the unique geodesic of $\Gamma$ joining $u$ and $v$, we write $\boxedo{\du{v,u}} := \du{v,e}$.
   We view this as the weight in the neighborhood of $v$ pointing towards $u$.
     The \emph{total weight} of a node $v$ of $\Gamma$ is  the product $\displaystyle{\boxedo{\du{v}}:=\!\!\!\!\prod_{e\in \starT{\Gamma}{v}}\!\!\!\! \du{v,e}}$.  
\end{definition}

   \begin{remark} Let $(\Gamma,\{\du{v,e}\}_{v,e})$ be a splice diagram. 
     For simplicity, we remove the collection of weights from the notation and simply use $\Gamma$ to refer to the splice diagram. By a similar abuse of notation,  
     we may view $\starT{\Gamma}{v}$ also as a splice diagram, whose weights around 
     its unique node $v$ are inherited from $\Gamma$.  Splice diagrams with one node 
     will be referred to as \emph{star splice diagrams}, and the underlying graphs as \emph{star trees}.
   \end{remark}

  \begin{definition}\label{def:linkingNumbers}
  Let $u$ and $v$ be  two distinct vertices of the splice diagram $\Gamma$. 
    The \emph{linking number} $\boxedo{\wtuv{u}{v}}$ between $u$ and $v$ is the product of all the  
    weights adjacent to, but not on,  the geodesic $[u,v]$ joining $u$ and $v$. 
    Thus, $\wtuv{v}{u} = \wtuv{u}{v}$. We set $\boxedo{\wtuv{v}{v}}:=\du{v}$ 
    for each node $v$ of $\Gamma$. 
    The \emph{reduced linking number} $\boxedo{\wtuv{u}{v}'}$ is defined via a similar product 
    where we exclude the weights around $u$ and $v$. In particular, $\wtuv{v}{v}' =1$ 
    for each node $v$ of $\Gamma$.
\end{definition}

  \begin{remark}\label{rem:linkvsRedLink}
Given a node $v$ and a leaf $\lambda$ of $\Gamma$, it is immediate to check that $\wtuv{v}{\lambda} \,{\du{v,\lambda}} = \wtuv{v}{\lambda}'\,\du{v}$.
  \end{remark}

Linking numbers satisfy the following useful identity, whose proof is immediate 
    from~\autoref{def:linkingNumbers} (see \cite[Proposition 69]{GGP 18}):
    
\begin{lemma}\label{lm:usefulIdentityLinkingEdge}
      If $u, v, w$ are vertices of $\Gamma$ with $u\in [v,w]$, then 
      $\wtuv{u}{v}\,\wtuv{u}{w} = \du{u} \, \wtuv{v}{w}$.
  \end{lemma}

In~\cite[Theorem 1]{N 77}, Neumann gave explicit descriptions of integral homology spheres associated to star splice diagrams as links of Pham-Brieskorn-Hamm surface singularities.
The following definition  was introduced by Neumann and Wahl in \cite[Section 1]{NW 05} to characterize which integral homology spheres may be realized as links of normal surface singularities. Its origins can be traced back to~\cite[page 82]{EN 85}.

\begin{definition}   \label{def:edgedet} 
  Let $\Gamma$ be a splice diagram. Given two adjacent nodes $u$ and $v$  of $\Gamma$,    the \emph{determinant} of the edge $[u,v]$ is the difference between the product of the 
  two decorations on $[u,v]$ and the product of the remaining decorations in the neighborhoods of $u$ and $v$, that is, 
 \begin{equation}\label{eq:edgeDet}
         \boxedo{\det([u,v])} :=  \du{u,v}\du{v,u} -\wtuv{u}{v}.
  \end{equation}
  The splice  diagram $\Gamma$ satisfies the \emph{edge determinant condition} if all edges have   positive  determinants.
\end{definition}

 The next result on integral homology sphere links is due to Eisenbud and Neumann (see~\cite[Theorem 9.4]{EN 85} for details and \autoref{def:coprimeSplice} for the meaning of the coprimality condition). It was strengthened by Pedersen in~\cite[Theorem 1]{P 11} to address the case of rational homology sphere links:
 
\begin{theorem}    \label{thm:charZHSlinks}
    The integral homology sphere links of normal surface singularities are precisely 
    the oriented $3$-manifolds $\Sigma(\Gamma)$ associated to coprime splice diagrams 
    $\Gamma$ which satisfy the edge determinant condition. 
\end{theorem}

\noindent
The construction of oriented $3$-manifolds from splice diagrams is due to Siebenmann~\cite{S 80}. They are obtained from splicing Seifert-fibered oriented $3$-manifolds $\Sigma (\starT{\Gamma}{v})$ associated to each node $v$ of $\Gamma$ along special fibers of their respective Seifert fibration corresponding to the edges of $\Gamma$.  For each $\Sigma (\starT{\Gamma}{v})$, these special fibers are in bijection with the $\valv{v}$-many edges adjacent to $v$. Each edge $[u,v]$ induces a splicing of both $\Sigma (\starT{\Gamma}{u})$ and $\Sigma (\starT{\Gamma}{v})$ along the oriented fibers corresponding to the edge. These fibers are knots in both  Seifert-fibered manifolds and their linking number is precisely $\wtuv{u}{v}$~(see \cite[Thm. 10.1]{EN 85}).

\smallskip

The following result shows that the edge determinant condition yields Cauchy-Schwarz' type inequalities:

\begin{lemma}\label{lm:CSh} Assume that the splice diagram $\Gamma$ satisfies the edge determinant condition. Then,
    \begin{equation}\label{eq:CSh}
      \du{u}\,\du{v} \geq \wtuv{u}{v}^2, \quad  \text{ for all nodes } u, v \in  \Gamma.
  \end{equation}
Furthermore,   equality holds if and only if $u = v$.
\end{lemma}

\begin{proof}
The result follows by induction on the distance $\distGuv{\Gamma}{u}{v}$ between $u$ and $v$ (see~\autoref{rm:geodesics}). The base case corresponds to adjacent nodes. \autoref{lm:usefulIdentityLinkingEdge} is used for the inductive step.
\end{proof}

If $\Gamma$ satisfies the edge determinant condition, then the linking numbers verify 
the following inequality, which generalizes 
\autoref{lm:usefulIdentityLinkingEdge}. This inequality is reminiscent of the ultrametric condition for dual graphs of arborescent singularities (see~\cite[Proposition 1.18]{GGPR 19}):

\begin{proposition}\label{pr:hypermetricineq}
         Assume that the splice diagram $\Gamma$ satisfies the edge determinant condition. 
         Then, for all  nodes $u,v$ and $w$ of $\Gamma$, we have 
       $\wtuv{u}{v}\,\wtuv{u}{w}\leq \du{u}\,\wtuv{v}{w}$.
       Furthermore, equality holds if and only if $u\in [v,w]$.
\end{proposition}

\begin{proof}   
  Consider the tree $T$ spanned by $u,v$ and $w$ and let $a$ be the unique node 
  in the intersections of the three geodesics $[u,v]$, $[u,w]$ and $[v,w]$. We prove 
  the inequality by a direct calculation. By~\autoref{lm:usefulIdentityLinkingEdge}  applied to the triples $\{a, u,v\}$, $\{a, u,w\}$ and $\{a, v,w\}$, we have:
  \[
    \wtuv{a}{u}\,\wtuv{a}{v} = \du{a} \, \wtuv{u}{v},
\quad     \wtuv{a}{u}\,\wtuv{a}{w} = \du{a} \, \wtuv{u}{w} \quad \text{ and } \quad
 \wtuv{a}{w}\,\wtuv{a}{v} = \du{a} \, \wtuv{w}{v}.   
  \]
  These expressions combined with the inequality~\eqref{eq:CSh} applied to the pair $\{a,u\}$ yield:
  \begin{equation*}\label{eq:simpleChain}
    \wtuv{u}{v}\,\wtuv{u}{w} =  \frac{\wtuv{a}{v}\, \wtuv{a}{u}}{\du{a}}\,\frac{\wtuv{a}{w}\, \wtuv{a}{u}}{\du{a}} = \frac{\wtuv{a}{v}\,\wtuv{a}{w}}{\du{a}} \frac{(\wtuv{a}{u})^2}{\du{a}} = \wtuv{v}{w}\frac{(\wtuv{a}{u})^2}{\du{a}} \leq \wtuv{v}{w} \,\du{u}.
    \end{equation*}
  Furthermore, \autoref{lm:CSh} confirms that equality is attained  if and only if $a=u$, that is, if and only if $u$ lies in the geodesic $[v,w]$. This concludes our proof.
\end{proof}

Before introducing splice type systems as defined by Neumann and Wahl in~\cite{NW 05bis, NW 05}, we set up notation arising from toric geometry. We write $\boxedo{n}:=|\leavesT{\Gamma}|$ for the number of leaves of the splice diagram $\Gamma$ (where $|\;|$ denotes the cardinality of a finite set) and  let $\boxedo{\Mwp{\leavesT{\Gamma}}}$ be the free abelian group generated by all leaves of $\Gamma$. We denote by $\boxedo{\Nwp{\leavesT{\Gamma}}}$ its dual lattice and write the associated pairing using dot product notation, i.e. $\wu{}\cdot m$ whenever $\wu{} \in \Nwp{\leavesT{\Gamma}}$ and $m\in \Mwp{\leavesT{\Gamma}}$.
Fixing a basis $\{\boxedo{\wu{\lambda}}: \lambda \in \leavesT{\Gamma}\}$ for $\Nwp{\leavesT{\Gamma}}$ and its dual basis $\{\boxedo{m_{\lambda}}: \lambda \in \leavesT{\Gamma}\}$ for  $\Mwp{\leavesT{\Gamma}}$ identifies both lattices with $\Z^{n}$.   To each $m_{\lambda}$, we associated a variable $\boxedo{z_{\lambda}}$. 
We view $\Mwp{\leavesT{\Gamma}}$ as the lattice of exponents 
of monomials in those variables and $\Nwp{\leavesT{\Gamma}}$ as the associated lattice of weight vectors.

In addition to defining weights for all leaves of $\Gamma$, each node $u$ in $\Gamma$ has an associated  weight vector:
\begin{equation}\label{eq:wu}
  \boxedo{\wu{u}} :=\sum_{\lambda \in \leavesT{\Gamma}} \wtuv{u}{\lambda} \, \wu{\lambda} \in \Nwp{\leavesT{\Gamma}}. 
  \end{equation}

As was mentioned in~\autoref{sec:introduction}, star splice diagrams $\Gamma$ with a unique node $v$ produce Pham-Brieskorn-Hamm singularities using the monomials $\{z_{\lambda}^{\du{v,\lambda}}: \lambda \in \leavesT{\Gamma}\}$. Neumann and Wahl's splice type systems~\cite{NW 05bis, NW 05} generalize this construct to diagrams with more than one node. In addition to satisfying the edge determinant condition, $\Gamma$ must have an extra arithmetic property that allows to replace each monomial $z_v^{\du{v,\lambda}}$ by a suitable monomial associated to the pair $(v,e)$ where $v$ is any node and $e$ is an edge adjacent to it (see~\eqref{eq:admissibleMon}).  This property, which ensures that all monomials associated to a vertex $v$ have the same  $\wu{v}$-degree, will automatically hold for star splice diagrams. 

\begin{definition}\label{def:semgpcond}  
  A splice diagram $\Gamma$ satisfies the \emph{semigroup condition} if for each node $v$ and each edge $e\in \starT{}{v}$,  the total weight  $\du{v}$ of $v$ 
 belongs to the subsemigroup of $(\N, +)$ generated by the set of linking numbers between $v$ and the leaves $\lambda$ seen from $v$ in the direction of $e$, that is, such that $e \subseteq [v,\lambda]$. 
 Therefore, we may write:
\begin{equation}\label{eq:PowersadmisibleMon}
         \du{v} = \sum_{\lambda \in \nodesTevRoot{v,e}} \wtNveL{v}{e}{\lambda} \, \wtuv{v}{\lambda}\,, \quad    
         \text{ or equivalently } \quad \du{v,e} = \sum_{\lambda \in \nodesTevRoot{v,e}} \wtNveL{v}{e}{\lambda} \, \wtuv{v}{\lambda}',
\end{equation}
where $\boxedo{\wtNveL{v}{e}{\lambda}}\in \N$ for all $\lambda$ and
$\boxedo{\nodesTevRoot{v,e}}$ is the set of leaves $\lambda$ of $\Gamma$ with $e\subseteq [v,\lambda]$.
  \end{definition}

Assume that $\Gamma$ satisfies the semigroup condition and pick coefficients $\wtNveL{v}{e}{\lambda}$ satisfying ~\eqref{eq:PowersadmisibleMon}. 
Using these integers we define an exponent vector (i.e., an element of $\Mwp{\leavesT{\Gamma}}$) for each pair $(v,e)$ as above:
\begin{equation}\label{eq:admissibleMonExp}
  \boxedo{\wtNve{v}{e}} := \sum_{\lambda \in \nodesTevRoot{v,e}}
  \wtNveL{v}{e}{\lambda} \,m_{\lambda} \in  \Mwp{\nodesTevRoot{v,e}}\subset \Mwp{\leavesT{\Gamma}}.
\end{equation}
Following~\cite{NW 05}, we refer to it as an \emph{admissible exponent} for $(v,e)$.
Note that the relation~\eqref{eq:PowersadmisibleMon} is equivalent to:
\begin{equation}\label{eq:admMonwuvalue} \wu{v}\cdot \wtNve{v}{e} = \du{v}.
\end{equation}

Each admissible exponent  $\wtNve{v}{e}$ defines an \emph{admissible monomial}, which was denoted by $M_{v,e}$ in~\cite{NW 05}:
\begin{equation}\label{eq:admissibleMon}
\boxedo{\zexp{\wtNve{v}{e}}} := \prod_{\lambda \in \nodesTevRoot{v,e}}
z_{\lambda}^{\wtNveL{v}{e}{\lambda}}.
\end{equation}

\begin{definition}   \label{def:splicesystem}
      Let $\Gamma$ be a splice diagram which satisfies both the edge determinant and the semigroup conditions of \textcolor{blue}{Definitions}~\ref{def:edgedet} and \ref{def:semgpcond}. We fix an order for its set  $\leavesT{\Gamma}$ of $n$ leaves.
        \begin{itemize}
             \item A \emph{strict splice type system} associated to $\Gamma$ 
                 is a finite family  of $(n-2)$ polynomials of the form:
                \begin{equation}\label{eq:surface}
                         \boxedo{ \fvi{v}{i}(\zu)}   :=\sum_{e\in {\starT{
                               }{v}}} \!\!\!\!    \boxedo{\cvei{v}{e}{i}}  \; \zexp{\wtNve{v}{e}} \quad 
                               \text{for all }   i\in \{1,\dots, \valv{v}-2\} \, \text{ and } v\text{  a node of }\Gamma,
                 \end{equation}
                     where  $\wtNve{v}{e} \in \Mw{\Gamma} $ are the admissible exponent vectors defined by~\eqref{eq:admissibleMonExp} for each node  $v\in \Gamma$ and each edge $e \in \starT{}{v}$.
                    We also require the coefficients $\cvei{v}{e}{i}$ to satisfy the \emph{Hamm determinant conditions}.     
                   Namely, for any node $v\in \Gamma$,  if we fix an ordering 
                   of the edges in $\starT{}{v}$,  then all the maximal minors of the matrix of coefficients $(\cvei{v}{e}{i})_{e,i} \in \CC^{\valv{v} \times (\valv{v}-2)}$ must be non-zero.

             \item A \emph{splice type system} $ \boxedo{\sG{\Gamma}} $  
             associated to $\Gamma$ is a finite family  of power series  of the form 
                \begin{equation}\label{eq:surfaceSeries}
                    \boxedo{ \fgvi{v}{i}(\zu)}   := \fvi{v}{i}(\zu)  + \gvi{v}{i}(\zu)
                    \quad \text{for all }   i\in \{1,\dots, \valv{v}-2\} \, \text{ and } v\text{  a node of }\Gamma,
                 \end{equation}
                where the collection $(\fvi{v}{i})_{v,i}$ is a strict splice type system associated to $\Gamma$ and each $\gvi{v}{i}$ is a convergent power series satisfying the following condition for each exponent $m$ in the support of $\gvi{v}{i}$:
                  \begin{equation}\label{eq:gviConditions}
                           \wu{v}\cdot m > \du{v} \quad \text{ and }\quad \wu{u} \cdot m > \wtuv{u}{v} 
                           \quad \text{ for each node } u \text{ of } \Gamma \text{ with } u\neq v.
                  \end{equation}
                                    \item A \emph{splice type singularity} associated to  $\Gamma$ is the subgerm of $(\CC^n, 0)$ defined  by $\sG{\Gamma}$.
     \end{itemize}
\end{definition} 

\begin{remark}\label{rm:gviConditions} 
The inequalities in \eqref{eq:gviConditions} should be compared with the equality imposed in~\eqref{eq:admMonwuvalue}.
As was shown by Neumann and Wahl in~\cite[Lemma 3.2]{NW 05bis}, the right-most inequality in~\eqref{eq:gviConditions} follows from the left-most one and the edge determinant condition. We choose to include both inequalities in~\eqref{eq:gviConditions} for mere convenience since we will need both of them for several arguments in~\autoref{sec:local-trop-newm}.
          \end{remark}

The issue of dependency of the set of germs defined by splice type systems on the choice of admissible monomials is a subtle one. We postpone this discussion to~\autoref{sec:natur-splice-type}.

As was mentioned in~\autoref{sec:introduction}, splice type singularities satisfy the following crucial property, proved by Neumann and Wahl in~\cite[Thm. 2.6]{NW 05bis}.     
An alternative proof of this statement, using local tropicalization, 
will be provided at the end of~\autoref{sec:NewtonND}.
    
    \begin{theorem}   \label{thm:spliceicis}
     Splice type singularities are isolated complete intersection surface singularities. 
\end{theorem}

        \begin{figure}[tb]
    \includegraphics[scale=0.7]{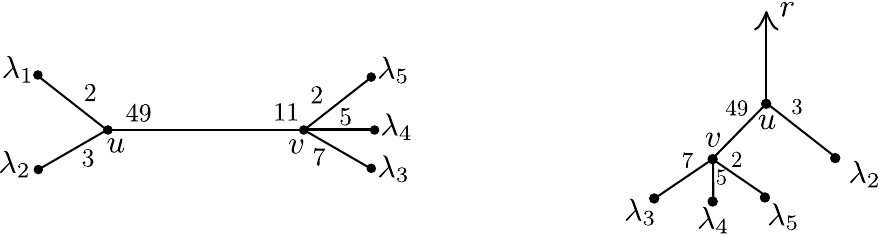}
        \caption{From left to right: a splice diagram and its associated rooted diagram obtained by fixing one of the leaves as its root $\roottree$, and removing one weight from the star of each node.\label{fig:numExampleSimpleEndCurve}}
        \end{figure}

\begin{example}\label{ex:2NodesNW}
  We let $\Gamma$ be the splice diagram to the left of \autoref{fig:numExampleSimpleEndCurve}. 
  Then, $\du{u} = 294$, $\du{v}= 770$ and $\wtuv{u}{v} = 420$ and so the edge determinant condition $\du{u}\,\du{v}>\wtuv{u}{v}^2$ holds for $[u,v]$.

 The semigroup condition is also satisfied, since 
      \begin{equation*}
               49 =   0 \cdot   (2 \cdot 5) +  1 \cdot  ( 2 \cdot 7)   +  1 \cdot (5 \cdot 7) \quad \text{ and } \quad
               11  =   1  \cdot (3) + 4 \cdot  (2) = 3  \cdot (3) + 1 \cdot  (2).
      \end{equation*}
 Thus, we may take as exponents $\wtNve{u}{[u,v]} = (0,0,0,1,1)$ and  $\wtNve{v}{[u,v]} = (1,4,0,0,0)$ or $(3,1,0,0,0)$ in $\Z^5$. 
  A possible  strict splice type system for $\Gamma$ is:
  \begin{equation}\label{eq:numExampleSurface}
    \begin{cases}  \fvi{u}{1}:= \;\;\;\, z_1^{2}\;\; \;-\;\;\;2\; z_2^3\; +\;\;\; z_4\,z_5 , \\
  \fvi{v}{1}:=\;\;\;\; z_1z_2^4 +  z_3^7 + \;\; z_4^5\, - \;2155\; z_5^2   , \\
   \fvi{v}{2}:=  33\, z_1z_2^4 +  z_3^7 + 2\,z_4^5 - \; 2123 \, z_5^2 \,  .
    \end{cases}
  \end{equation}
 An alternative system is obtained by replacing the admissible monomial $z_1z_2^4$ with $z_1^3z_2$.   The coefficients of the system were chosen to simplify the parameterization of the end-curve of the  corresponding splice type surface singularity associated to the leaf $\lambda_1$ (see~\autoref{ex:NumExample3} for details).
\end{example}

        A central role in this paper will be played by tropicalizations and weighted initial forms of series and ideals of $\CC\{z_{\lambda}: \lambda \in \leavesT{\Gamma}\}$, which we discuss in~\autoref{sec:local-trop}. In~\autoref{pr:InitWufvi}, we determine the initial forms of the series $\fgvi{v}{i}$ of a splice type system $\sG{\Gamma}$ with respect to each weight vector $\wu{u}$ from~\eqref{eq:wu}. Its proof is a consequence of the next two lemmas:

\begin{lemma}\label{rm:weightsAdjacent} Assume that $\Gamma$ is a splice diagram satisfying the edge determinant condition. Then,  for any pair  of adjacent nodes $u,v$ of $\Gamma$ we have:
  \begin{equation*}
    \wu{u} \in\; \frac{\wtuv{u}{v}}{\du{v}}\, \wu{v} + \R_{>0}\langle \wu{\lambda}\colon \lambda \in \nodesTevRoot{v,[u,v]}\rangle.
  \end{equation*}
\end{lemma}

\begin{proof} We write $e=[u,v]$. The definition of linking numbers gives  the following expressions for each  $\lambda \in \leavesT{\Gamma}$: 
  \[
  \wtuv{u}{\lambda} =\begin{cases}
  (\wtuv{v}{\lambda}\, \du{u})/(\du{v,u}\,\du{u,v}) & \text{ if }\lambda \in \nodesTevRoot{u,e}, \\
(\wtuv{v}{\lambda}\, \du{v,u}\,\du{u,v})/\du{v} & \text{ if } \lambda \in \nodesTevRoot{v,e}.
  \end{cases}
  \]
  
\noindent 
The statement follows by substituting these expressions in the definition of $\wu{u}$ from~\eqref{eq:wu} and by using the edge determinant condition, i.e.,
\begin{equation}\label{eq:wuwv}
  \wu{u} =  \frac{\du{u}}{\du{v,u}\,\du{u,v}}\! \sum_{\lambda \in \nodesTevRoot{u,e}}
  \wtuv{v}{\lambda}\, \wu{\lambda} + \frac{\du{v,u}\,\du{u,v}}{\du{v}} \sum_{\lambda \in \nodesTevRoot{v,e}}  \!\!\wtuv{v}{\lambda} \,\wu{\lambda} =
  \frac{\wtuv{u}{v}}{\du{v}}\,\wu{v} +  \sum_{\lambda \in \nodesTevRoot{v,e}}   \underbrace{\frac{\det(e)\,\wtuv{v}{\lambda}}{\du{v}}}_{>0}\, \wu{\lambda}.\qedhere
\end{equation}
  \end{proof}

\begin{lemma}\label{lem:keyid}
  Assume that $\Gamma$ satisfies the edge determinant and semigroup conditions. Then, the exponent vector $\wtNve{v}{e}$ from~\eqref{eq:admissibleMonExp} satisfies 
  $\displaystyle{ \wu{u} \cdot \wtNve{v}{e}  \geq \wtuv{u}{v}}$  for all nodes $u$ of $\Gamma$ and each edge $e\in \starT{\Gamma}{v}$.   
  Furthermore, equality holds if and only if $e\not \subseteq [u,v]$.
\end{lemma}

\begin{proof}  If $u=v$, then $\wu{v} \cdot \wtNve{v}{e}= \du{v} = \wtuv{v}{v}$.  
  If $u\neq v$, we argue by  induction on the distance $\distGuv{\Gamma}{u}{v}>0$ between $u$ and $v$ in the tree $\Gamma$. If $\distGuv{\Gamma}{u}{v}=1$, we let $e'=[u,v]$. Expression~\eqref{eq:wuwv} yields
  \begin{equation*}\label{eq:interm}
\wu{u} \cdot     \wtNve{v}{e}=    \frac{\wtuv{u}{v}}{\du{v}}  \,   \wu{v}\cdot \wtNve{v}{e} + \frac{\det(e')}{\du{v}} \big (\sum_{\lambda \in \nodesTevRoot{v,e'}} \wtNveL{v}{e}{\lambda}\,\wtuv{v}{\lambda} \big ) = \wtuv{u}{v} + \frac{\det(e')}{\du{v}} \big (\sum_{\lambda \in \nodesTevRoot{v,e'}} \wtNveL{v}{e}{\lambda}\,\wtuv{v}{\lambda} \big ).
  \end{equation*}
  The second summand is always non-negative and it equals zero if and only if $e\neq e'$.

  If $\distGuv{\Gamma}{u}{v}>1$, we let $u'$ be the unique node adjacent to $u$ in  $[u,v]$ and set $e':=[u',u]$. Note that $e\subseteq [v,u]$ if and only if $e\subseteq [v,u']$. Expression~\eqref{eq:wuwv} applied to $\{u,u'\}$, the non-negativity of each $\wtNveL{v}{e}{\lambda}$, the inductive hypotheses on $\{u',v\}$ and~\autoref{pr:hypermetricineq} yield
  \begin{equation*}\label{eq:d>1}\wu{u} \cdot \wtNve{v}{e} = \frac{\wtuv{u}{u'}}{\du{u'}} \underbrace{\wu{u'}\cdot \wtNve{v}{e}}_{\geq \wtuv{u'}{v}}  + \frac{\det(e')}{\du{u'}} \sum_{\lambda \in \nodesTevRoot{u',e'}}
    \underbrace{\wtNveL{v}{e}{\lambda}\wtuv{u'}{\lambda}}_{\geq 0} \; \geq  \frac{\wtuv{u}{u'}\,\wtuv{u'}{v}}{\du{u'}} = \wtuv{u}{v}.
      \end{equation*}
      By construction, equality is achieved if and only if $e\nsubseteq [u',v]$, which is equivalent to $e \nsubseteq [u,v]$.
\end{proof}

In \autoref{sec:NewtonND}, we will be interested in curves obtained from a given splice type system when we choose a  leaf  $\roottree$ of the corresponding splice diagram $\Gamma$ to be its root. We orient the resulting rooted tree $\boxedo{\Gamma_{\roottree}}$ towards the root and remove one weight in the neighborhood of each node, namely the one pointing towards the root,  as seen on the right of~\autoref{fig:numExampleSimpleEndCurve}.
We write $\boxedo{\rDG{\roottree}}$ for the set of $(n-1)$ non-root leaves of 
$\Gamma_{\roottree}$ and assume it is ordered.
 The following definition was introduced in~\cite[Section 3]{NW 05bis} by the name 
 of \emph{splice diagram curves}. 
 
\begin{definition}    \label{def:endcurverSR}
  Assume that the rooted splice diagram $\Gamma_{\roottree}$ satisfies the semigroup condition and consider a fixed strict splice type system $\sG{\Gamma}$ associated to the (unrooted) splice diagram $\Gamma$.
  For each node $v$ of $\Gamma$ and each index $i \in \{ 1, \dots, \valv{v}\}$,  we
     let $\boxedo{\hvi{v}{i}(\zu)} \in \CC[z_{\lambda}\colon \lambda \in \rDG{\roottree}]$ 
    be  the polynomial obtained from $ \fvi{v}{i}(\zu)$ by removing the term 
    corresponding to the unique edge adjacent to $v$  pointing towards $\roottree$.  
    The subvariety of $\CC^{\rDG{\roottree}}\simeq \CC^{n-1}$ defined by the vanishing of $ (\hvi{v}{i}(\zu))_{v,i}$ 
    is called the \emph{end-curve of} $\sG{\Gamma}$ relative to $\roottree$. 
    We denote it by $\boxedo{\Ccurve_\roottree}$.
\end{definition}

A planar embedding of $\Gamma_{\roottree}$ determines an ordering of the edges adjacent to a fixed node $v$ that point away from $\roottree$ (for example, by reading them from left to right). 
Once this order is fixed, by the Hamm determinant conditions,  each group of equations $ (\hvi{v}{i}(\zu)=0)_{i=1}^{\valv{v}-2}$ becomes equivalent to a collection of $\wu{v}$-homogeneous binomial equations  of the form
\[
\zexp{\wtNve{v}{e_j}} - a_{v,j} \zexp{\wtNve{v}{e_{\valv{v}-1}}}= 0 \quad \text{ for } j\in \{1,\ldots, \valv{v}-2\},
\]
with all $a_{v,j}\neq 0$.  The next statement summarizes the main properties of $\Ccurve_{\roottree}$ discussed in~\cite[Theorem 3.1]{NW 05bis}:

\begin{theorem}\label{thm:end-curvesNW} 
  The subvariety $\Ccurve_{\roottree}\subseteq \CC^{n-1}$ is
a reduced complete intersection curve, smooth away from the origin, and meets any coordinate subspace of $\CC^{n-1}$ only at the origin. It has $g$ many components, where $g:=\gcd\{\wtuv{\roottree}{\lambda}: \lambda \in \rDG{\roottree}\}$. All of them are isomorphic to torus-translates of the monomial curve in $\CC^{n-1}$ with parameterization $t\mapsto (t^{\wtuv{\roottree}{\lambda_1}/g}, \ldots, t^{\wtuv{\roottree}{\lambda_{n-1}}/g})$.
  \end{theorem}

\begin{example}\label{ex:NumExample3} 
  We fix the splice diagram from~\autoref{ex:2NodesNW}  and consider its rooted analog obtained by setting the first leaf as its root $\roottree$, as seen in the right of~\autoref{fig:numExampleSimpleEndCurve}. By construction, $\wtuv{\roottree}{2} = 49$, $\wtuv{\roottree}{3}=30$, $\wtuv{\roottree}{4}=42$, $\wtuv{\roottree}{5} = 105$, $\wtuv{r}{u} = 147$, $\wtuv{r}{v}=210$, $\wu{u}=(49,30,42,105)$ and $\wu{v}=(70,10,14,35)$. The equations defining this end-curve are obtained by removing the monomial indexed by the edge pointing towards $\roottree$ in each equation from~\eqref{eq:numExampleSurface}. Since $g=1$, the curve $\Ccurve_{\roottree}$ 
  is reduced and irreducible. It is defined as the solution set to
 \[-2\,z_2^3+ z_4\,z_5 = z_3^7 + z_4^5 - 2155\,z_5^2 = z_3^7+2\,z_4^5 - 2123\, z_5^2=0.\] Linear combinations of the last two  expressions yield the equivalent binomial system:
 \[-2\,z_2^3+ z_4\,z_5 = z_4^5 + 32 z_5^2 = z_3^7 - 2187 z_5^2 = 0.
\]
 An explicit parameterization is given by 
$(z_2,z_3,z_4,z_5)= (-t^{49}, 3\,t^{30}, -2\,t^{42}, t^{105})$. 
\end{example}

The collection  $\hvi{v}{i}(\zu)$ of polynomials defining the end-curve $\Ccurve_{\roottree}$ determines a map $G_{\roottree}\colon \CC^{n-1} \to \CC^{n-2}$. Our next result, which we state for comparison's sake with \autoref{cor:dom-map}, 
discusses the restriction of this map to each coordinate hyperplane of $\CC^{n-1}$:

\begin{corollary}\label{cor:dominant-curve-map}
     For every $ \lambda \in \rDG{\roottree}$, the restriction of $G_{\roottree}$ to the hyperplane $H_{\lambda}$ of  $\CC^{n-1}$ defined by the equation $z_\lambda=0$ is dominant.
\end{corollary}

\begin{proof}
   By \autoref{thm:end-curvesNW}, the fiber over the origin 
  of the restricted map $G_{\roottree|H_{\lambda}}$ is finite. Upper semicontinuity of fiber dimensions   
  implies that the generic fiber is also $0$-dimensional. Since $\dim H_{\lambda}=n-2$, the map   
  $G_{\roottree\,|H_{\lambda}}$ must be dominant.
\end{proof}

\section{Local tropicalization}
\label{sec:local-trop}

In~\cite{PPS 13}, the last two authors developed a theory of {\em local tropicalizations} 
of algebraic, analytic or formal germs endowed with 
maps to (not necessarily normal) toric varieties, adapting the original formulation 
of global tropicalization  
(see, e.g.~\cite{MS 15}) to the local setting. In this section, we recall the basics on local tropicalizations that will be needed in~\autoref{sec:local-trop-newm}. We focus our attention on germs 
$(Y,0)\hookrightarrow \CC^n$ defined  by ideals $I$ of the ring of convergent power series 
$\boxedo{\cO}:=\CC\{z_1,\ldots, z_n\}$ near the origin, rather than of its completion 
$\boxedo{\hat{\cO}}:=\CC\llbracket z_1,\ldots, z_n \rrbracket $. 
As~\autoref{rm:convergentToPowerSeries} confirms, both local tropicalizations yield the same set.
\medskip

The notion of local tropicalization of an embedded germ is rooted on the construction of initial ideals associated to  non-negative weight vectors, which we now describe.
Any weight vector ${\wu{}} := (\wu{1},\ldots, \wu{n}) \in (\Rp)^n$ induces a real-valued valuation  on $\cO$, known as the \emph{$\wu{}$-weight},  as follows. Given a monomial 
$\boxedo{\zu^{{\alpha}}}:=z_1^{\alpha_1}\cdots z_n^{\alpha_n}$, we set
  $\boxedo{\wu{}(\zu^{{\alpha}})} := \wu{}\cdot {\alpha} = \sum_{i=1}^n \wu{i}\,\alpha_i$. In turn, for each $f=\sum_{\alpha} c_{\alpha} \,\zu^{\alpha}$ with $f\neq 0$ we set
  \begin{equation}\label{eq:valwf}
       \boxedo{\wu{}(f)} := \min\{\wu{}(\zu^{\alpha}): c_\alpha\neq 0\} = 
            \min\{\wu{}\cdot \alpha: c_\alpha\neq 0\}.
  \end{equation}
  We define $\wu{}(0):=\infty$. The set $\{\alpha: c_\alpha\neq 0\}$ is called 
  the \emph{support of $f$}, and it is the basis of the construction of the \emph{Newton polyhedron} and  the
  \emph{Newton fan} of $f$ (see \autoref{def:NewtonFan}).

 \begin{definition}\label{def:initwf} 
    Given $f\in \cO$ and ${\wu{}} \in (\Rp)^n$, the \emph{$\wu{}$-initial form} $\boxedo{\initwf{\wu{}}{f}} \in \cO$ is the sum of the terms in the series $f$ with minimal $\wu{}$-weight $\wu{}(f)$.    
    In turn, given an ideal $I$ of $\cO$, the \emph{$\wu{}$-initial ideal $\boxedo{\initwf{\wu{}}{I}\cO}$ 
    of $I$ in $\cO$}  is  generated by the $\wu{}$-initial forms of all elements of $I$.    
     If ${\wu{}} \in (\R_{>0})^n$, the 
     \emph{$\wu{}$-initial ideal $\boxedo{\initwf{\wu{}}{I}}$ of $I$}  is the ideal of 
     $\CC[z_1, \ldots, z_n]$   
     generated by the $\wu{}$-initial forms of all elements of $I$.  
  \end{definition}

 \begin{example}\label{ex:2NodesNWcont} Given the minimal splice type system from~\autoref{ex:2NodesNW}, we have $\wu{u} = (147, 98, 60, 84, 210)$ and $\wu{v}=(210, 140, 110, 154, 385)$. The polynomial $\fvi{u}{1}$ is $\wu{u}$-homogeneous, whereas $\fvi{v}{1}$ and $\fvi{v}{2}$ are $\wu{v}$-homogeneous. Their initial forms relative to the weight vectors $\wu{u}$ and $\wu{v}$ are:
      \[
    \begin{cases}  
           \init_{\wu{v}}(\fvi{u}{1})=  \; z_1^{2} \;-\;2\; z_2^3, \\
          \init_{\wu{u}}(\fvi{v}{1})=   \;  z_3^7 \;+ \; z_4^5\, - \;2155\; z_5^2   , \\
           \init_{\wu{u}}(\fvi{v}{2})= \;  z_3^7 \;+ 2\,z_4^5 - \; 2123 \, z_5^2 \,  .
    \end{cases}
    \]  
    We will see in \autoref{lm:wu} that  $\initwf{\wu{}}{\sG{\Gamma}} =\langle \initwf{\wu{}}{\fvi{u}{1}}, \initwf{\wu{}}{\fvi{v}{1}}, \initwf{\wu{}}{\fvi{v}{2}}\rangle\cO$ when $\wu{}$ equals  $\wu{u}$ or $\wu{v}$.
 \end{example}

The initial forms of a series determine its Newton fan as follows:
  
  \begin{definition}\label{def:NewtonFan} 
     The  \emph{Newton polyhedron} $\boxedo{\NP(f)}$ of a non-zero element $f\in \cO$  
     is the convex hull of the Minkowski sum of $(\Rp)^n$ and the support of $f$. 
    Given a face $K$ of  $\NP(f)$, we let $\boxedo{\sigma_K}$ be the  closure of the 
    set of weight vectors $\wu{}$ in $(\Rp)^n$ supporting $K$ 
    (that is, such that the convex hull of the support of $\initwf{\wu{}}{f}$ is $K$). 
    The set $\{\sigma_{K}: K \text{ face of } \NP(f)\}$ is the \emph{Newton fan} 
    $\boxedo{\NF(f)}$ of $f$.
  \end{definition}
  
  The map $K \mapsto \sigma_K$ yields an inclusion-reversing bijection between the set of faces of $\NP(f)$ 
  and the Newton fan $\NF(f)$. Furthermore, every face $K$ of $\NP(f)$ satisfies
     \begin{equation} \label{eq:constsumdim}
           \dim K + \dim \sigma_K =n.
     \end{equation}

 \begin{definition}   \label{def:posloctrop} 
     Let $(Y,0) \subseteq \CC^n$ be a  germ defined by an ideal $I$ of $\cO$. The 
     \emph{local tropicalization of $I$} or \emph{of the germ $Y$}, is the set of all vectors 
     $\wu{}\in (\Rp)^n$ such that the $\wu{}$-initial ideal $\initwf{\wu{}}{I}\cO \subseteq \cO$ of $I$  
     is  monomial-free. We denote it by  $\boxedo{\Trop I}$ or $\boxedo{\Trop Y}$.
        In turn, the \emph{positive local tropicalization of  $I$} or \emph{of the germ $Y$},  
        is  the intersection of the local tropicalization with the positive orthant $(\R_{>0})^n$. 
        We denote it by  $\boxedo{\ptrop I}$ or $\boxedo{\ptrop Y}$.
\end{definition}

 \noindent
Even though \autoref{def:posloctrop} depends heavily on the fixed embedding $(Y,0) \subseteq \CC^n$, we omit it from the notation for the sake of simplicity. The next remarks clarify some differences between the present approach and that of~\cite{PPS 13}, which defines local tropicalizations using local valuation spaces~ (see \cite[Definitions 5.13 and 6.7]{PPS 13}). A recent extension of this construction to toric prevarieties by means of Berkovich analytification can be found in~\cite{KSU 21}. 

\begin{remark}\label{rm:convergentToPowerSeries} 
As shown in~\cite[Theorem 11.2]{PPS 13} and \cite[Corollary~4.3]{S 17},
local tropicalizations of ideals in either $\cO$ or $\hat{\cO}$ 
admit several equivalent characterizations  analogous to the 
\emph{Fundamental Theorem of Tropical Algebraic Geometry}~\cite[Theorem 3.2.3]{MS 15}. 
One of them is as Euclidean closures in $(\Rp)^n$ of images of local valuation spaces. 
By~\cite[Corollary 5.17]{PPS 13}, the canonical inclusion 
$(\cO,\mfkm)\hookrightarrow  ( \hat{\cO}, \hat{\mfkm})$  induces an isomorphism 
of local valuation spaces. This implies that extending an ideal in $\cO$ 
to the complete ring $\hat{\cO}$  will yield the same local tropicalization. 
Therefore, we can define local tropicalizations for ideals of $\cO$ 
rather than of $\hat{\cO}$, in agreement with the setting of splice type systems.
\end{remark}

As in the global case, local tropicalizations of hypersurface germs can be obtained from the corresponding Newton fans. Indeed, if the ideal $I$ is principal, generated by $f\in \cO$, 
and $\wu{}\in (\Rp)^n$, then each initial ideal $\initwf{\wu{}}{I}\cO$ is also principal, 
with generator $\initwf{\wu{}}{f}$. 
Equality (\ref{eq:constsumdim}) then yields the following statement (see~\cite[Proposition 11.8]{PPS 13}):

\begin{proposition} \label{prop:loctropprincid}
      The set $\Trop (f)$ is the union of all cones of the Newton fan  of $f$ dual to bounded edges of the Newton polyhedron  of $f$. 
\end{proposition}

\begin{figure}[tb]
  \includegraphics[scale=0.7]{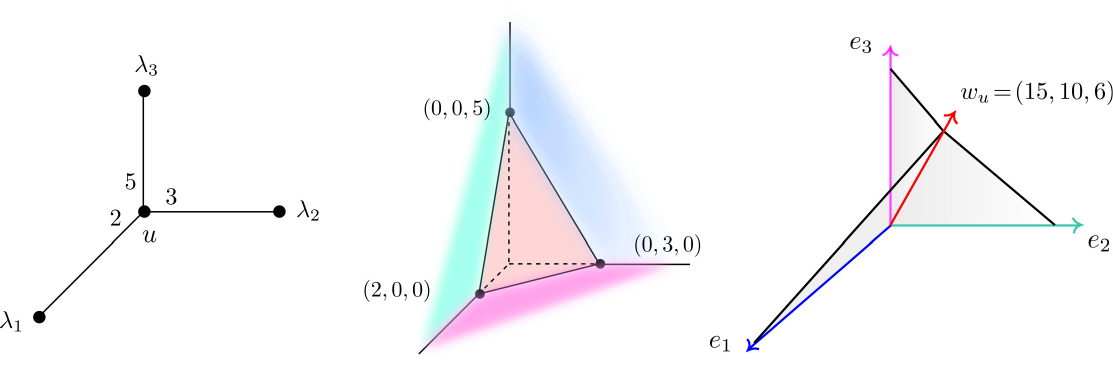}
    \caption{Splice diagram, Newton polyhedron and local tropicalization
      of the $E_8$ surface singularity.\label{fig:ExampleE8SingularityTropical}}
\end{figure}

 \begin{example}\label{ex:E8Tropical}
 The $E_8$ surface singularity is the splice type surface singularity defined by the polynomial $z_1^2+z_2^3+z_3^5$. Its associated splice diagram,  Newton polyhedron and  local tropicalization are depicted in~\autoref{fig:ExampleE8SingularityTropical}. Its local tropicalization is a 2-dimensional fan with four rays spanned by  $e_1$, $e_2$, $e_3$,  and $\wu{u}=(15,10,6)$. Its three maximal cones are spanned by the pairs $\{e_i,\wu{u}\}$ for $i\in \{1,2,3\}$.   These three cones are dual to the three bounded edges of the Newton polyhedron. 
 \end{example}

By contrast, if $I$ has two or more generators, their $\wu{}$-initial forms need not generate  $\initwf{\wu{}}{I}\cO$. 
  However, for the purpose of characterizing $\ptrop I$, it is enough to have a    \emph{tropical basis} for $I$ in the sense of \cite[Definition 10.1]{PPS 13}, 
  i.e.,  a finite set of generators $\{f_1,\ldots, f_s\}$ of 
  $I$ which is a universal standard basis of $I$ in the sense of \cite[Definition 9.8]{PPS 13} and 
  such that for any $\wu{}\in (\R_{>0})^n$, the $\wu{}$-initial ideal 
  $\initwf{\wu{}}{I}\cO$ contains a monomial if and only if one of the initial forms 
  $\initwf{\wu{}}{f_i}$ is a scalar multiple of a monomial.

\begin{remark}   \label{rem:morethantwo}
  Such tropical bases exist by~\cite[Theorem 10.3]{PPS 13} and can be used 
  to determine $\Trop I$ by intersecting the local tropicalizations of the corresponding 
  hypersurface germs. Furthermore, their existence ensures that local tropicalizations 
  are supports of rational fans in $(\Rp)^n$. Indeed, the corresponding fan is obtained 
  by considering the common refinement of the intersection of the local tropicalization 
  of each member of a tropical basis for $I$ combined with~\autoref{prop:loctropprincid}. 
  Furthermore, under the hypothesis that no irreducible 
    component of $(Y,0)$ is included in a coordinate subspace of $\CC^n$, 
    such fans and their refinements are \emph{standard tropicalizing fans} 
    of $Y$ in the sense of \autoref{def:tropicalizingFan} below.
 \end{remark}

\begin{figure}[tb]  
         \includegraphics[scale=0.65]{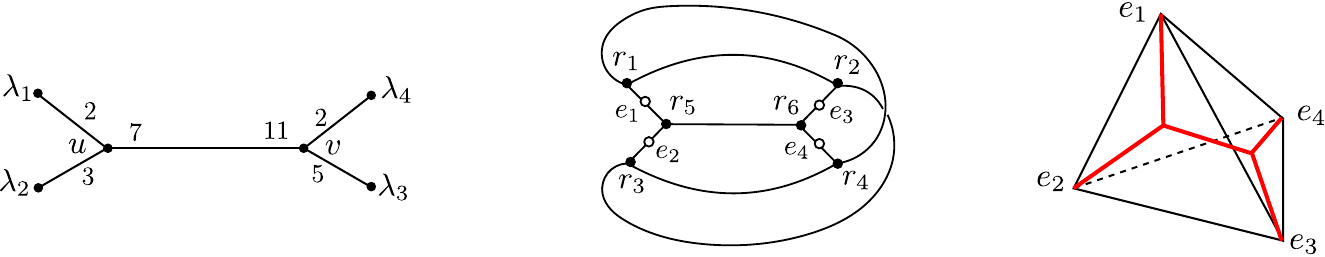}
         \caption{From left to right: Splice diagram, and representations of the global and local tropicalizations of~\autoref{ex:simpleExampleNW}.
             \label{fig:simpleExampleNW}
}               
  \end{figure}

\begin{example} \label{ex:simpleExampleNW} 
   We consider the germ of splice type surface singularity from~\cite[Example 2]{NW 05}, given by
  \begin{equation}\label{eq:NWSimple}
    \begin{cases}
 f_u:= z_1^2 + z_2^3 + z_3z_4 = 0,\\
 f_v:= z_3^5 + z_4^2 + z_1z_2^4 = 0.
  \end{cases}
  \end{equation}
associated to the splice diagram on the left of~\autoref{fig:simpleExampleNW}. A computation with the package \texttt{Tropical.m2}~\cite{TropicalPackage}, available in \texttt{Macaulay2}~\cite{M2}, determines the global tropicalization of the system (see~\cite[Definition 3.2.1]{MS 15}). It is a 2-dimensional fan in $\R^4$ with $f$-vector $(1,6,9)$. Its six rays are generated by the primitive vectors 
  \[
  \begin{aligned}
    r_1&:=(-28,-13,-16,-40), \qquad r_3:=(-2,-7,-6,-15), \qquad r_5:=(21,14,12,30),\\
    r_2&:=(-15, -10, -11, -19), \qquad  r_4:=(-6,-4,-1,-11), \qquad r_6:=(30,20,22,55).
  \end{aligned}
  \]
Notice that $r_5=\wu{u}$ and $r_6=\wu{v}$. Its nine top-dimensional cones are encoded by the graph depicted at the center of the figure.

  The local tropicalization of the germ defined by~\eqref{eq:NWSimple} is obtained by intersecting the global tropicalization with the positive orthant~\cite[Theorem 12.10]{PPS 13}.  We indicate 
  the positions of the four canonical basis elements $\{e_1,\ldots, e_4\}$ of $\R^4$ with unfilled dots inside the edges of the central graph of~\autoref{fig:simpleExampleNW}. As we will see in~\autoref{thm:tropsG}, the local tropicalization of this germ can be obtained as the cone over the red graph depicted in the standard tetrahedron seen in the right of the figure. Note that this graph is homeomorphic to the splice diagram. This fact is general, as confirmed by~\autoref{thm:injectivityrho}.
  \end{example}

\begin{remark}\label{rem:newTerminologyPPS}
     Our choice of terminology for local tropicalizations differs slightly from~\cite{PPS 13}, 
     as we now explain.  As was shown in~\cite[Section 6]{PPS 13}, the local
     tropicalizations of the intersection of a germ $(Y,0)\hookrightarrow \CC^n$ with each 
     coordinate subspace of $\CC^n$ can be glued together to form an \emph{extended fan} 
     in $(\Rp\cup \{\infty\})^n$ called the \emph{local nonnegative tropicalization} 
     of $Y$ in~\cite{PPS 13}. In the present paper, we refer to this structure as an \emph{extended tropicalization} of the germ $Y$, in agreement with Kashiwara and Payne's constructions for global tropicalizations (see~\cite[\S 6.2]{MS 15}). The finite 
     part of this extended tropicalization (i.e., its intersection with $(\Rp)^n$) is the local tropicalization from~\autoref{def:posloctrop}. A precise description  of the boundary 
     strata {of the extended local tropicalization of splice type singularities} 
     is given in~\autoref{ssec:infinite-tropicalization}.
\end{remark}

\begin{remark}\label{rem:loctroprernonred}
    The local tropicalization of a germ $(Y,0) \subseteq \CC^n$ coincides  with the local tropicalization of the associated reduced germ $(Y_{red},0) \subseteq \CC^n$ since the defining ideal of $(Y_{red},0)$ is the radical $\sqrt{I} \subseteq \cO$ and $\wu{}$-initial forms respect products. 
    As a consequence, $\initwf{\wu{}}{\sqrt{I}}$ is monomial-free if and only if the same is true for $\initwf{\wu{}}{I}$. Alternatively, the same statement can be obtained from  the definition of local tropicalization as the image of the local valuation space of $Y$ and the fact that the embedding $Y_{red} \hookrightarrow Y$ induces a homeomorphism of local valuation spaces (see~\cite[Lemma 5.18]{PPS 13}).
\end{remark}

\begin{remark}\label{rem:loctropunion}
    The local tropicalization of a reduced germ $(Y,0) \subseteq \CC^n$ is equal to the 
    union of the local tropicalizations of its irreducible components, and the same is true 
    for their local positive tropicalizations. This is a direct consequence 
    of the fact that the local valuation space of $(Y,0)$ is the union of the local valuation spaces of its irreducible components (see \cite[Lemma 5.18]{PPS 13}).
\end{remark}

The next result determines the local tropicalization of a germ $(Y,0)\hookrightarrow \CC^n$ 
from the positive one:
\begin{proposition}    \label{prop:compartrop} 
  The local tropicalization   
  $\Trop Y$ is the closure of the positive local tropicalization $\ptrop Y$ inside the cone $(\Rp)^n$.
\end{proposition}

\begin{proof}
      By~\textcolor{blue}{Remarks}~\ref{rem:loctroprernonred} and~\ref{rem:loctropunion},~it suffices to consider the case where 
     $(Y,0)$ is irreducible. In this situation, \cite[Theorem 11.9]{PPS 13} shows that 
     the extended local tropicalization of $(Y,0)$ is the closure of the extended positive 
     local tropicalization in the extended non-negative orthant $(\Rp\cup \{\infty\})^n$. 
Applying \autoref{lem:gentopclose} below to the case when $E$ is the extended local 
     tropicalization of $(\CC^n,0)$,  $U=(\Rp)^n$ and $A$ is the extended 
     local tropicalization of $(Y,0)$  confirms the claim about the closure of $\ptrop Y$ in $(\Rp)^n$. 
\end{proof}

The following lemma is a standard statement in general set topology, which may for instance be obtained as an immediate consequence of \cite[Theorem 17.4]{M 00}: 
\begin{lemma}  \label{lem:gentopclose}
   Let $E$ be a topological space and $U$ be an open subset. 
   Then, for any subset $A$ of $E$ we have:
   \[  \mathrm{cl}_U(A \cap U) =  \mathrm{cl}_E(A) \cap U, \]
where $\mathrm{cl}_E(A)$ is the closure of $A$ in $E$ and $\mathrm{cl}_U(A \cap U)$ denotes the closure 
   of $A \cap U$ in $U$.
\end{lemma}

In what follows, we restrict our attention to subgerms of $(\CC^n,0)$ with no components 
included in coordinate subspaces. 
Their local tropicalizations verify the following key property (see~\cite[Proposition 9.21, Theorems 10.3 and 11.9]{PPS 13}):

\begin{proposition}\label{prop:puredimltrop}
Let $(Y,0)$ be a subgerm of $(\CC^n, 0)$ with no irreducible components contained in coordinate subspaces of  $\CC^n$, and let $I \subseteq \cO$ be its defining ideal.   
  Then, $\Trop Y$ is the support of a rational polyhedral fan $\cF$ which satisfies the following conditions: 
        \begin{enumerate} 
           \item   \label{condsamedim}
               the dimension of all the maximal cones of $\cF$  agrees with the complex dimension  of $Y$; 
           \item \label{condnonempty}
                the maximal cones of $\cF$ have non-empty intersections with  $(\R_{>0})^n$;   
            \item \label{condinit}
               given any cone $\tau$ of $\cF$, the  $\wu{}$-initial ideal of $I$ is independent of the choice of $\wu{} \in \relo{\tau}$. 
        \end{enumerate}
\end{proposition}

\begin{proof}  
    By~\textcolor{blue}{Remarks}~\ref{rem:loctroprernonred} and~\ref{rem:loctropunion}, 
    we may restrict to the case when $(Y,0)$ is reduced and irreducible. The existence of $\cF$ 
    satisfying the first two conditions is a direct consequence 
   of  \cite[Theorem 11.9]{PPS 13}, which ensures the analogous properties hold for the extended 
   local tropicalization of $(Y,0)$. Note that $\Trop Y$ is non-empty, since $(Y,0)$ meets the 
   dense torus of $\CC^n$ (see \cite[Lemma 7.4]{PPS 13}). Furthermore, the fan structure on 
   $\Trop Y$ induced from a tropical basis $\{f_1,\ldots, f_r\}$ for $I$ discussed 
   in~\autoref{rem:morethantwo} satisfies condition (\ref{condinit}). In turn, any refinement 
   will satisfy conditions (\ref{condsamedim}) and (\ref{condnonempty}) as well. 
\end{proof}

\begin{remark}
   \autoref{prop:puredimltrop} allows us to recover the complex dimension of an irreducible germ 
  $(Y,0)\hookrightarrow \CC^n$ meeting the dense torus from its positive tropicalization 
  (see, for instance, the proofs of~\textcolor{blue}{Corollaries}~\ref{cor:expDimensionsG} 
  and~\ref{cor:codim1TropBoundary}). Indeed, its dimension agrees with the dimension   of any of the top-dimensional cones in any fixed fan satisfying 
  conditions (\ref{condsamedim})--(\ref{condinit}) in~\autoref{prop:puredimltrop}. 
If $(Y,0)$ is not irreducible, a similar procedure  determines the maximal   dimension of a component of the germ meeting the dense torus. 
\end{remark}

\begin{definition}\label{def:tropicalizingFan} 
     Let $(Y,0)$ be a subgerm of $\CC^n$ with no irreducible component 
     contained in a coordinate subspace of  $\CC^n$. Let $I$ be the ideal of $\cO$ defining $(Y,0)$.
    Any fan $\cF$ satisfying all three conditions in~\autoref{prop:puredimltrop} 
    is called a \emph{standard tropicalizing fan} for $(Y,0)$ or for $I$. 
     For every cone $\tau$ of $\cF$ meeting $(\R_{>0})^n$ and  any $\wu{}\in \relo{\tau}$, we write $\boxedo{\initwf{\tau}{I}}:=\initwf{\wu{}}{I}$ for the associated initial ideal in the polynomial ring $\CC[z_1,\ldots, z_n]$ and $\boxedo{\initwf{\tau}{Y}}:=\initwf{\wu{}}{Y}$ for the associated subscheme of $\CC^n$ (see~\autoref{def:initwf}). 
  \end{definition}
  
\begin{remark}\label{rm:fanStructureTrop}
  When $Y$ is a hypersurface germ defined by a series $f \in \cO$, the set of 
  cones of codimension one of the Newton fan of $f$ dual to bounded edges of the Newton polyhedron of $f$ satisfies all three conditions listed in~\autoref{prop:puredimltrop}. Furthermore, it is the coarsest fan with these properties. However, for germs of higher codimension such canonical choice need not always exist. For an example in the global (i.e., polynomial) setting, we refer the reader to~\cite[Example 3.5.4]{MS 15}. 
\end{remark}

The next proposition emphasizes the relevance of tropicalizing fans for producing birational models of irreducible germs  with desirable geometric properties, in the spirit of Tevelev's construction of tropical compactifications of subvarieties of tori~\cite{T 07}. 

\begin{proposition}
   \label{prop:purecodimlift}
   Fix a rational polyhedral fan $\cF$ contained in $(\Rp)^n$ and let  
     $\pi_{\cF}\from \tv_{\cF}\to \CC^n$ be the associated toric morphism. 
     Given an irreducible germ $(Y,0)\hookrightarrow \CC^n$ meeting the dense torus, 
     let $Y_{\cF}$ be the strict transform  of $Y$ under $\pi_{\cF}$  and write   
     $\pi \from Y_{\cF} \to Y$  for the restriction 
     of the map $\pi_{\cF}$ to $Y_{\cF}$. Then, the following properties hold:
       \begin{enumerate}
           \item\label{properness} The restriction $\pi$ is proper 
               if and only if the support $|\cF|$ contains the local tropicalization $\Trop Y$. 
            \item\label{torusOrbitsIntersections} Assume that $\pi$ is proper. Then, the strict transform $Y_{\cF}$ 
               intersects every orbit $S$ of $\tv_{\cF}$ along a non-empty 
               pure-dimensional subvariety with 
               $\codim_{Y_{\cF}}(Y_{\cF}\cap S)= \codim_{\tv_{\cF}} (S)$ 
               if and only if $|\cF| = \Trop Y$. 
       \end{enumerate}
\end{proposition}

\begin{proof} In what follows, we use standard terminology and notation from toric geometry, which can be found in Fulton's book~\cite{F 93}. The proof of~(\ref{torusOrbitsIntersections}) is similar to the global analog~\cite[Proposition 6.4.7 (2)]{MS 15}, so we leave it to the reader.

  It remains to prove assertion~(\ref{properness}). To this end, we consider a fan $\Sigma$ subdividing the 
non-negative cone
$(\Rp)^n$ and containing $\cF$ as a subfan. 
Such a fan exists by~\cite[Theorem 2.8 (III.2)]{E 96}.

We consider the toric varieties $\tv_{\cF}$ and $\tv_{\Sigma}$ associated to the fans $\cF$ and $\Sigma$ and the natural toric morphisms $\pi_{\cF}\colon \tv_{\cF}\to \CC^{n}$ and $\pi_{\Sigma}\colon \tv_{\Sigma}\to \CC^n$. We let $Y_{\Sigma}$ and  $Y_{\cF}$ be the strict transforms of $Y$ under these two maps. The aforementioned varieties and maps fit naturally into the  commutative diagram
\[\xymatrix{
Y_\Sigma \ar@{^{(}->}[r] & \tv_{\Sigma} \ar[rd]_{\pi_\Sigma} & 
& \tv_{\cF} \ar@{_{(}->}[ll] \ar[ld]^{\pi_{\cF}} & Y_{\cF} \ar@{_{(}->}[l] \ar[ld]^{\pi} \\
 & & \CC^n & Y \ar@{_{(}->}[l]
}\]
where the central triangle involves toric morphisms and the horizontal arrows are embeddings.  The vertical map $\pi$ on the right  is the restriction of $\pi_{\cF}$ to $Y_{\cF}$.
In what follows we view
$\tv_{\cF}$ as an open subvariety of $\tv_{\Sigma}$. 
 Note that the toric birational morphism $\pi_\Sigma$ is proper by \cite[Section~2.4]{F 93}, as the defining fans of its source and target have the same support, namely $(\Rp)^n$.

By construction, $\pi$ is proper if and only if $Y_{\Sigma} $  
  is contained in $\tv_{\cF}$. Thus, claim~(\ref{properness}) will follow if we  show that for every cone $\tau$ of $\Sigma$ we have the following equivalence:
  \begin{equation}   \label{eq:equivorbtor}
         \cO_\tau \cap Y_\Sigma\neq \emptyset \Longleftrightarrow
            \tau^{\circ}\cap \Trop Y\neq \emptyset,
     \end{equation}
where $\cO_\tau:=\Hom_{\operatorname{monoid}}(\tau^\perp\cap M,\CC)$ is the corresponding toric orbit.

It remains to prove~\eqref{eq:equivorbtor}. We start with the forward implication and fix  $y_0\in \cO_\tau\cap Y_\Sigma$. Then, there
exists a holomorphic arc in $Y\cap (\CC^*)^n$ parameterized as $t \mapsto y(t)$
such that the limit as $t \to 0$ of its strict transform in $Y_\Sigma$ equals $y_0$, i.e.,
     $$\lim_{t\to 0} y(t)=y_0\in \cO_\tau \cap Y_\Sigma.$$ 
   Such an arc can be built by choosing an irreducible subgerm of a curve $C_{\Sigma}$ of  
    the germ $(Y_\Sigma, y_0)$, not contained in the toric boundary, then by projecting it 
    to a subgerm $C$ of $Y$ via $\pi_{\Sigma}$ and, finally,  by choosing a normalization of $C$, which we identify with $(\mathbb{C},0)$. 

    We consider the weight vector $\wu{}:=\mathrm{ord}\,y(t)$ in the dual lattice $N:=M^{\vee}$ recording 
    the orders of vanishing of the components of $y(t)$. We claim that $\wu{}$ belongs to $\tau^{\circ}\cap \Trop Y$, so  $\tau^{\circ}\cap \Trop Y\ne \emptyset$, as desired.
        To prove this claim, notice that the arc $t \mapsto y_{\mathrm{in}}(t)$ of $Y_\Sigma$  obtained by keeping the $\wu{}$-initial terms $y_{\mathrm{in}}(t)$ of the components of 
$y(t)$ has the same limit when $t \to 0$ as $y(t)$ does. This fact can be checked by working in the affine toric variety associated to the cone $\Rp \langle \wu{}\rangle$.
Properties of limits in toric varieties from~\cite[Section~2.3]{F 93} ensure that $\lim_{t\to 0}y_{\mathrm{in}}(t) \in \cO_\tau$ is  equivalent to the fact that the 
weight vector $\wu{}$  lies in $\tau^{\circ}\cap N$. 
Therefore, we conclude that $\mathrm{ord}\,y(t) \in \tau^{\circ}$. 

It remains to verify that  $\wu{}\in \Trop Y$. To do so, it suffices to notice that $\Trop C \subseteq \Trop Y$ (since $C\subseteq Y$) and that $\Trop C = \Rp \langle \wu{}\rangle$. The latter is a direct consequence of \cite[Corollary 4.3 3')]{S 17} (see also Maurer's paper~\cite{M 80} which includes a  precursor of local tropicalization for germs of space curves).

Finally, to prove the reverse implication of~\eqref{eq:equivorbtor}, we assume that $\tau^{\circ}\cap \Trop Y \neq \emptyset$ and let  $\wu{} \in \tau^{\circ}\cap \Trop Y$ be a primitive lattice vector in $N\simeq \Z^n$. We consider a refinement $\Sigma_{\wu{}}$ of $\Sigma$ such that the ray $\tau_{\wu{}}=\Rp\langle \wu{}\rangle \in \Sigma_{\wu{}}$. By construction,  the orbit $\cO_{\tau_{\wu{}}}$ is mapped via the toric morphism
$\pi_{\wu{}}\colon \tv_{\Sigma_{\wu{}}}\to \tv_{\Sigma}$ to the orbit $\cO_\tau$.

The intersection
of $Y_{\Sigma_{\wu{}}}$ with the orbit $\cO_{\tau_{\wu{}}}$ is determined by the $\wu{}$-initial ideal $\initwf{\wu{}}{I(Y)}$ of the ideal $I(Y)$ defining $Y$, viewed in the Laurent polynomial ring. Since this initial ideal is monomial free because $\wu{}\in \Trop Y$, this intersection is non-empty.
 Since $\cO_{\tau_{\wu{}}} \subseteq Y_{\Sigma_{\wu{}}}$, the map $\pi_{\wu{}}$ ensures that $Y_\Sigma\cap \cO_\tau\neq \emptyset$ as well. This concludes our proof.
\end{proof}

One well-known feature of global tropicalizations of equidimensional subvarieties of toric varieties  
is the so-called \emph{balancing condition}~\cite[Section 3.3]{MS 15}. To this end, tropical varieties 
must be endowed with positive integer weights (called \emph{tropical multiplicities}) along their top-dimensional 
cones. Such multiplicities may be also defined in the local situation. We restrict the exposition to  equidimensional germs, since this is sufficient for the purposes of this paper.
 
  \begin{definition}\label{def:tropMult}
      Let $(Y,0) \hookrightarrow \CC^n$ be an equidimensional germ meeting $(\CC^*)^n$, defined 
      by an ideal $I$ of $\cO$ and let $\cF$ be a standard tropicalizing fan for it.  
      Given a top-dimensional 
      cone $\tau$ of $\cF$, we define the \emph{tropical multiplicity} of $\cF$ at $\tau$ to be the 
      number of irreducible components of $\initwf{\tau}{Y} \cap (\CC^*)^n$, counted with multiplicity.
  \end{definition}

In order to state the balancing condition for local tropicalization, we first define the notion of a pure rational weighted balanced fan. Recall that a fan is pure if all its top-dimensional cells have the same dimension.

  \begin{definition}\label{def:balancingCondition}
Consider a pure rational polyhedral fan $\cF$ in $\R^n$ relative to the lattice $\Z^n$, with positive integer weights on its maximal cones. Fix a cone $\tau$ of $\cF$ of codimension one in $\cF$, and let $\sigma_1,\ldots, \sigma_s$ be the maximal cones of $\cF$ containing $\tau$ as a face. Denote by $m_1,\ldots, m_s$ the corresponding weights. For each $i\in \{1,\ldots, s\}$ consider a vector $\ww_{\sigma_i|\tau} \in \Z^n \cap \sigma_i$ generating the lattice
    \begin{equation}\label{eq:balancing}
    \Lambda_i:=\frac{\Z^n \cap \R\langle \sigma_i\rangle}{\Z^n\cap \R\langle \tau \rangle}.
    \end{equation}
We say that $\cF$ is \emph{balanced at $\tau$} if $\sum_{i=1}^s m_{i} \ww_{\sigma_i|\tau} \in \Z^n \cap \R\langle \tau\rangle$. The fan $\cF$ is \emph{balanced} if it is balanced at each of its codimension one cones. 
  \end{definition}

  \begin{remark} \label{rm:uniquenessVectorBalancing} 
      By construction, the lattice $\Lambda_i$ in~\eqref{eq:balancing} is free of rank one. It has one generator, up to sign. Even though there are many choices for  $\ww_{\sigma_i|\tau}$, their projections onto $\Lambda_i$ give the same generator of the lattice.
  \end{remark}
  
Our last statement in this section confirms that the balancing condition holds for local tropicalizations of equidimensional germs in $\CC^n$ meeting the dense torus $(\CC^*)^n$. Its validity follows from~\cite[Remark 11.3, Theorem 12.10]{PPS 13}:

  \begin{theorem}     \label{thm:balancingCondition}
       Let $(Y,0) \hookrightarrow \CC^n$ be an equidimensional germ meeting $(\CC^*)^n$ and let $\cF$ be any refinement of a standard tropicalizing fan  for $Y$. Then, $\cF$ is a balanced fan when endowed with tropical multiplicities as in \autoref{def:tropMult}. 
  \end{theorem}
  
  \begin{remark}\label{rem:balancedLoc} 
        Balancing for local tropicalizations of equidimensional  germs of surfaces features in the proof of~\autoref{pr:StarIsIn}. This property will help us prove that the embedded splice diagrams in $\R^n$ are included 
    in the local tropicalizations of the corresponding splice type systems 
    (see~\autoref{ssec:proof-supseteq}). Tropical multiplicities will be also used 
    in~\autoref{sec:recov-splice-diagr} to recover the edge weights on any coprime splice 
    diagram $\Gamma$ from the local tropicalization of any splice type surface singularity associated to it.
  \end{remark}


\section{Newton non-degeneracy}
\label{sec:nnondeg}

In this section we discuss the notion of Newton non-degenerate complete intersection in the sense of Khovanskii \cite{K 77}, starting with 
the case of formal power series in $\hat{\cO}$, as introduced by Kouchnirenko in \cite[Sect. 8]{K 75} and \cite[Def. 1.19]{K 76}. Kouchnirenko's definition was later extended by Steenbrink~\cite[Def. 5]{S 14} to $\CC$-algebras of formal power series $\CC\llbracket P \rrbracket$ with exponents on an arbitrary saturated sharp toric monoid $P$. For precursors to this notion we refer the reader to Teissier's work~\cite[Section 5]{T 04}.

\begin{definition}  \label{def:Nnondegfunct} Given a series $f\in \hat{\cO}$, we say that $f$ is \emph{Newton non-degenerate} if for any positive weight vector $\wu{}\in (\R_{>0})^n$, the subvariety of the dense torus $(\CC^*)^n$ defined by  $\initwf{\wu{}}{f}$ is smooth.
\end{definition}

\begin{example}\label{ex:NNDHypersurface} 
      We consider the $E_8$ singularity   from~\autoref{ex:E8Tropical}, 
      defined by $f=z_1^2+z_2^3+z_3^5$. Its local tropicalization is depicted 
      in~\autoref{fig:ExampleE8SingularityTropical}. Since $\initwf{\wu{}}{f}$ is a monomial 
      whenever $\wu{}\notin \ptrop Z(f)$, the Newton non-degeneracy condition only 
      needs to be checked for $\wu{}\in \ptrop Z(f)$.

    The calculations are simplified since $\ptrop Z(f)$ is a subfan of the normal fan to the Newton polyhedron of $f$. If $\wu{}\in \R_{>0} \langle \wu{u}  \rangle$, we have $\initwf{\wu{}}{f}=f$. In turn, weight vectors in the relative interiors of maximal cones of $\ptrop Z(f)$ satisfy
  \[
  \initwf{\wu{}}{f} = \begin{cases}
    z_2^3 + z_3^5 & \text{ if }\wu{}\in \R_{>0} \langle e_1, \wu{u}\rangle,\\
     z_1^2 + z_3^5 & \text{ if }\wu{}\in \R_{>0} \langle e_2, \wu{u}\rangle,\\
     z_1^2 + z_2^3 & \text{ if }\wu{}\in \R_{>0} \langle e_3, \wu{u}\rangle.
  \end{cases}
  \]
All four initial forms describe surfaces that are smooth when restricted to the torus $(\CC^*)^3$. Thus, $Z(f)$ is Newton non-degenerate.
\end{example}

The notion of Newton non-degeneracy extends naturally  to finite sequences of functions. For our purposes, it suffices to restrict ourselves to \emph{regular sequences}, i.e.,  collections $(f_1,\ldots, f_s)$  in $\hat{\cO}$ where

    \begin{enumerate}
        \item $f_1$ is not a zero divisor of $\hat{\cO}$, and
        \item   for each $i\in \{1,\ldots, s-1\}$, the element ${f_{i+1}}$ is not  a zero divisor 
           in the quotient ring $\hat{\cO}/\langle f_1,\,\ldots, f_{i}\rangle \hat{\cO}$.
    \end{enumerate}
As $\hat{\cO}$ is a regular local ring, the germs defined by regular sequences in $\hat{\cO}$ are exactly formal complete intersections at the origin of $\CC^n$.

\begin{definition}  \label{def:Nnondegci}
  Fix a positive integer $s$ and a regular sequence $(f_1, \dots, f_s)$ in $\hat{\cO}$. 
  Consider the germ  $(Y, 0) \hookrightarrow \CC^n$ defined by the ideal $\langle f_1, \dots, f_s\rangle\hat{\cO}$. 
  The sequence $(f_1, \dots, f_s)$ is 
  \emph{a Newton non-degenerate complete intersection system} for $(Y,0)$ if for any  
  positive weight vector $\wu{}\in (\R_{>0})^n$, the hypersurfaces of $(\CC^*)^n$ 
   defined by each $\initwf{\wu{}}{f_i}$ form a normal crossings divisor in a neighborhood of their   
   intersection. Equivalently, the differentials  of the initial forms 
   $\initwf{\wu{}}{f_1}, \dots, \initwf{\wu{}}{f_s}$ 
   must be linearly independent at each point of this intersection. 
\end{definition} 

\begin{remark}\label{rem:emptyCI}
  Notice that \autoref{def:Nnondegci} allows for the initial forms to be monomials. This would determine an empty intersection with the dense torus $(\CC^*)^n$.
\end{remark}

Our main result in this section is a useful tool for proving that a regular sequence in $\hat{\cO}$ defines a Newton non-degenerate system. We make use of this statement in \autoref{thm:NewtonNonDeg}:

\begin{proposition}\label{pr:regSmoothIsNND} 
      Let $s$ be a positive integer, $(f_1, \dots, f_s)$ be a sequence in $\hat{\cO}$ and $\wu{}$ be a weight vector in $(\R_{>0})^n$. Assume that $(\initwf{\wu{}}{f_1}, \ldots, \initwf{\wu{}}{f_s})$ is a regular sequence in $\CC[z_1,\ldots, z_n]$ defining a  subscheme $Z_{\wu{}}$ of $\CC^n$. If $p$ is a smooth point of $Z_{\wu{}}\cap (\CC^*)^n$, then the  differentials at $p$ of the initial forms $\initwf{\wu{}}{f_1}, \dots, \initwf{\wu{}}{f_s}$ are linearly independent.
\end{proposition}

\begin{proof} Fix any point $p\in Z_{\wu{}}\subseteq (\CC^*)^n$. Our regularity hypothesis implies that $\dim(Z_{\wu{}})=n-s$. Since $p$ is a smooth point of $Z_{\wu{}}$, by~\cite[Lemma 4.3.5]{dJP 00} we know that $\cO_{Z_{\wu{}},p}$ is a regular local ring of dimension $\dim(Z_{\wu{}})$. Furthermore, this quantity equals  the embedding dimension $\mathrm{edim}(\cO_{Z_{\wu{}},p})$ of the local ring $\cO_{Z_{\wu{}},p}$.

   In turn, the Jacobian criterion for smoothness (see, e.g., \cite[Theorem 4.3.6]{dJP 00}) enables us to compute the rank of the $n\times s$ Jacobian matrix of the initial forms $\{\initwf{\wu{}}{f_i}\}_{i=1}^s$ evaluated at $p$, in terms of $n$ and this embedding dimension. More precisely, we have
    \[\rk_p(\operatorname{Jac}(\initwf{\wu{}}{f_1}, \ldots, \initwf{\wu{}}{f_s})) = n - \mathrm{edim}(\cO_{Z_{\wu{}},p}) = n-(n-s)=s.
    \]
   This implies that the differentials of the initial forms $\initwf{\wu{}}{f_1}, \ldots, \initwf{\wu{}}{f_s}$ are linearly independent at $p$.
\end{proof}
\begin{example}\label{ex:NW_NND} 
   We consider the splice type system from~\autoref{ex:simpleExampleNW}. As claimed by~\autoref{thm:main1}, the germ defined by it is Newton non-degenerate. By~\autoref{rem:emptyCI} we need only focus on weights $\wu{}\in \ptrop \langle f_u,f_v\rangle\cO$. A direct computation using~\autoref{lem:keyid} shows that the initial forms of $f_u$ and $f_v$ are constant along the relative interiors of maximal cones of $\ptrop \langle f_u,f_v \rangle\cO$. We check that the Newton non-degeneracy condition is satisfied for $\wu{}$ in two of the five maximal cones. The remaining three cases are similar.

Pick $\wu{}\in \R_{>0} \langle \wu{u},\wu{v}\rangle$. Then, $\initwf{\wu{}}{f_u} = z_1^2 + z_2^3$ 
and $\initwf{\wu{}}{f_v} = z_3^5 + z_4^2$. Their differentials are 
$ 2\,z_1 \, dz_1 +  3\,z_2^2 \, d z_2$ and $5\,z_3^4 \, dz_3 +  2\,z_4 \, dz_4$, 
so they are linearly independent everywhere in the torus
$(\CC^*)^4$.
  Similarly, if $\wu{}\in \R_{>0} \langle e_1,\wu{u}\rangle$, 
  we have $\initwf{\wu{}}{f_u} =  z_2^3 + z_3\,z_4$ and $\initwf{\wu{}}{f_v} = z_3^5 + z_4^2$. 
  Their differentials are $3\,z_2^2 \, dz_2 + z_4 \, dz_3 +  z_3 \, dz_4$ 
  and $5\,z_3^4 \, dz_3 +  2z_4 \, dz_4$. They are also linearly 
  independent everywhere in $(\CC^*)^4$.
\end{example}

\begin{remark}   \label{rem:nnci}
\autoref{def:Nnondegci} modifies slightly Khovanskii's original definition from 
\cite[Rem.~4 of Sect.~2.4]{K 77}, by imposing the regularity of the sequence $(f_1,\ldots, f_s)$.   
The generalization to formal power series $\CC\llbracket P \rrbracket$ associated to an arbitrary 
sharp toric monoid $P$ is straightforward, and parallels that done by Steenbrink~\cite{S 14} 
for hypersurfaces. A slightly different notion of \emph{Newton non-degenerate ideals} 
was introduced by Saia in~\cite{S 95} for ideals of finite codimension in  $\hat{\cO}$. 
For a comparison with Khovanskii's approach we refer the reader to Bivi\`a-Ausina's 
work~\cite[Lemma 6.8]{B 07}.  For a general perspective on Newton non-degenerate 
complete intersections, the reader can consult Oka's book~\cite{O 97}. 
A definition of \emph{Newton non-degenerate  algebraic subgerms} of $(\CC^n, 0)$ 
for not necessarily complete intersections was given by Aroca, G{\'o}mez-Morales 
and Shabbir~\cite{AGS 13}. The results of their paper (and the proofs involving Gr\"obner bases) can be extended 
to the analytic and formal contexts by working with standard bases, in the spirit of~\cite{PPS 13}.
Recent work of Aroca, G{\'o}mez-Morales  and Mourtada~\cite{AGM 22} generalize the constructions from~\cite{AGS 13} to subgerms of arbitrary 
  normal affine toric varieties.
\end{remark}

\section{Embeddings and convexity properties of splice diagrams}
  \label{sec:comb-embedd-splice}

  In this section, we describe simplicial fans in the real weight space $\Nwp{\leavesT{\Gamma}}\otimes_{\Z} \R \simeq \R^n$ that arise from splice diagrams and appropriate subdiagrams. These constructions will play a central role in~\autoref{sec:local-trop-newm} when characterizing local tropicalizations of splice type surface singularities. Throughout this section we assume that the splice diagram $\Gamma$ satisfies the edge determinant condition of~\autoref{def:edgedet}. The semigroup condition from~\autoref{def:semgpcond} plays no role.
  \medskip

We let  $\boxedo{\simplex{n-1}}$ be the standard $(n-1)$-simplex 
of the real weight space $\R^n$,  with  
vertices  $\{\wu{\lambda}: \lambda \in \leavesT{\Gamma}\}$. 
We start by defining a piecewise linear map from $\Gamma$ to 
 $\simplex{n-1}$:
\begin{equation}\label{eq:embedGamma}
  \boxedo{\emb} \colon \Gamma \to \simplex{n-1} \qquad\text{ where }\quad \emb(v) = \frac{\wu{v}}{|\wu{v}|}\quad \text{ for each vertex } v \text{ of } \Gamma.   \end{equation}
Here, $|\mathord{\cdot}|$ denotes the $L^1$-norm in $\R^n$. In particular,  $|\wu{v}| = \sum_{\lambda \in \leavesT{\Gamma}} {\wtuv{v}{\lambda}}$ for each node $v$ of $\Gamma$.
After identifying each edge $e$ of $\Gamma$ with the interval $[0,1]$, the map $\emb$ on $e$ is defined by convex combinations of the assignment at its endpoints. The injectivity of $\emb$ will be discussed in~\autoref{thm:injectivityrho}.

 The following combinatorial constructions, in particular Definitions~\ref{def:fullSubtree}, \ref{def:branches} and~\ref{def:barycenter},  play a prominent role in proving~\autoref{thm:tropsG}. Stars of vertices (see~\autoref{def:vocabtrees}) and convex hulls of vertices (see~\autoref{def:convexHull}) are central to many arguments below.

\begin{definition}\label{def:fullSubtree}
A subtree $T$ of the splice diagram $\Gamma$ is \emph{star-full} if $\starT{\Gamma}{v} \subseteq T$
for every node $v$ of $T$. A node of $T$ is called an \emph{end-node} if it is adjacent to exactly one other node of $T$. 
\end{definition}

Every tree with two or more nodes contains at least two end-nodes. The following statement, 
illustrated in ~\autoref{fig:CherryPruning}, describes a method
to produce new star-full subtrees from old ones by \emph{pruning from an end-node}.  Its proof is
straightforward, so we omit it. 

\begin{lemma}\label{lm:cherryPruning} 
Let $T$ be a star-full subtree of $\Gamma$ with $d$ leaves $\{u_1,\ldots, u_d\}$. 
Fix an end-node $v$ of 
$T$ and assume that $u_1,\ldots, u_{\valv{v}-1}$ are the only leaves of $T$ adjacent to $v$. 
Then, the convex hull $T' := [v,u_{\valv{v}},\ldots, u_{d}]$  is  also star-full.
\end{lemma}

\begin{figure}[tb]
    \includegraphics[scale=0.5]{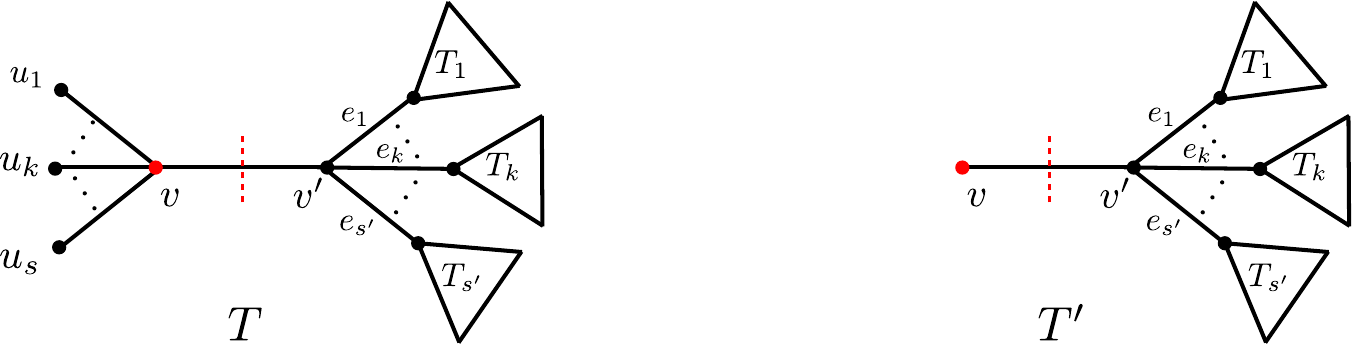}
    \caption{Pruning a star-full subtree $T$ of $\Gamma$ from an end-node $v$ of $T$  produces a new star-full subtree $T'$ with one fewer node, as in~\autoref{lm:cherryPruning}. Here, $s=\valv{v}-1$ and $s'=\valv{v'}-1$.
    \label{fig:CherryPruning}}
\end{figure}

\begin{definition}\label{def:branches}
A \emph{branch} of a tree $T$  adjacent to a node $v$  is a connected component of
  \[
  T\smallsetminus \big (\{v\} \cup \!\!\!\!\bigcup_{e\in \starT{\Gamma}{v}}\relo{e}\big ),
  \]
  where $\boxedo{\relo{e}}$ denotes the interior of the edge $e$.
\end{definition}

\noindent For example, $\{u_1\},\ldots, \{u_{s}\}$ and the convex hull $[v',T_1,\ldots T_{s'}]$ constructed inside the tree $T$ on the left of~\autoref{fig:CherryPruning} are the $\valv{v}$ branches of $T$ adjacent to the node $v$. 
Similarly, the branches of $T'$ adjacent to $v'$ are  $T_1,\ldots, T_{s'}$ and $\{v\}$. 
\smallskip

Star-full subtrees of splice diagrams have a key convexity property rooted in barycentric calculus. This is the
content of~\autoref{pr:convexitySimplices}. The following 
definition plays a crucial role in its proof.

\begin{definition}\label{def:barycenter}
Let $v$ be a node of $\Gamma$ and let $L$ be a set of leaves of $\Gamma$. We define $\baryc{L}{v}$ 
as the \emph{barycenter} of the leaves in $L$ with weights determined by $\wu{v}$, that is:
\begin{equation}
    \boxedo{\baryc{L}{v}} := \sum_{\lambda\in L}  \frac{\wtuv{v}{\lambda}}{\ell}\, \wu{\lambda} \qquad 
    \text{ where } \quad \ell:= \sum_{\lambda\in L}{\wtuv{v}{\lambda}}.
\end{equation}
In particular,  $\baryc{\{\lambda\}}{v} = \wu{\lambda}$ for any leaf $\lambda$.\end{definition}

\begin{remark}\label{rem:convexityBaryc} 
Fix a node  $v$ of $\Gamma$ with adjacent branches $\Gamma_1,\ldots, \Gamma_{\valv{v}}$. Then, the set $\{\baryc{\leavesT{\Gamma_i}}{v}: i=1,\ldots, \valv{v}\}$ is linearly independent, and  a direct computation gives
\begin{equation}\label{eq:lcombbarycenters}
\emb(v) = \sum_{j=1}^{\valv{v}} \big(\frac{1}{|\wu{v}|}\sum_{\lambda \in \leavesT{\Gamma_j}} \wtuv{v}{\lambda}\big) \,\baryc{\leavesT{\Gamma_j}}{v}.
\end{equation}
In particular, $\emb({v})$ lies in the relative interior of the simplex 
$\conv(  \{\baryc{\leavesT{\Gamma_j}}{v} \colon j=1,\ldots, \valv{v}\})\big)$, 
where $\conv$ denotes the affine convex hull inside $\simplex{n-1}$.
\end{remark}

Following the notation from~\autoref{def:semgpcond}, we write $\boxedo{\pag{a}}:=\nodesTevRoot{b,[a,b]}$ and $\boxedo{\pag{b}}:=\nodesTevRoot{a,[a,b]}$ for each pair of adjacent nodes $a,b$ of $\Gamma$.
Barycenters determined by splitting $\Gamma$ along the edge $[a,b]$  are closely related, as the following lemma shows:

\begin{lemma}\label{lm:adjacentBarycenters} 
  Let $a, b$ be two adjacent nodes of $\Gamma$,
  with associated sets of leaves $\pag{a}$ and $\pag{b}$ 
on each side of $\Gamma$. Then:
\begin{enumerate}
  \item \label{bary1} $\baryc{\pag{a}}{a} = \baryc{\pag{a}}{b}$ and 
       $\baryc{\pag{b}}{a} = \baryc{\pag{b}}{b}$ 
       (which   we denote by 
        $\boxedo{\baryc{\pag{a}}{[a,b]}}$ and $\boxedo{\baryc{\pag{b}}{[a,b]}}$, respectively).
  \item \label{bary2} The points $\emb({a})$ and $\emb({b})$ lie in the line segment 
        $[\baryc{\pag{a}}{[a,b]}, \baryc{\pag{b}}{[a,b]}]$.
    \item \label{bary3} $\baryc{\pag{a}}{[a,b]} < \emb({a}) < \emb({b}) < \baryc{\pag{b}}{[a,b]}$, 
          where $<$ is the order given by identifying the  segment 
          $[\baryc{\pag{a}}{[a,b]}, \baryc{\pag{b}}{[a,b]}]$ with $[0,1]$.   
\end{enumerate}
\end{lemma}

\begin{proof} 
We start by showing (\ref{bary1}). \autoref{def:linkingNumbers} induces the following identities:
\begin{equation}\label{eq:linksavsb}
    \wtuv{b}{\lambda} = \wtuv{a}{\lambda} \frac{\wtuv{a}{b}}{\du{a}} \quad \text{ for all } \lambda \in \pag{a}, \quad \text{ and } \quad     
    \wtuv{a}{\mu} = \wtuv{b}{\mu} \frac{\wtuv{a}{b}}{\du{b}} \quad 
    \text{ for all }  \mu \in \pag{b}.
\end{equation}
Thus, $a$ and $b$ contribute proportional weights to $\pag{a}$ and  $\pag{b}$, respectively. This implies that the corresponding barycenters agree, proving (\ref{bary1}).

Next, we discuss (\ref{bary2}). To simplify notation, we write $\wu{} = \baryc{\pag{a}}{[a,b]}$ and 
$\wu{}' = \baryc{\pag{b}}{[a,b]}$. Then,~\eqref{eq:linksavsb} yields:
\[\wu{a} = \big (\sum_{\lambda \in \pag{a}} \wtuv{a}{\lambda}\big ) \,\wu{} + 
\big (\sum_{\mu \in \pag{b}} \wtuv{a}{\mu}\big )\, \wu{}' \quad \text{ and }\quad 
\wu{b} = \big (\sum_{\lambda \in \pag{a}} \wtuv{b}{\lambda}\big ) \,\wu{} + 
\big (\sum_{\mu \in \pag{b}} \wtuv{b}{\mu}\big )\, \wu{}'.
\]
The definition of $\emb$ then confirms that  $\emb(a)$ and $\emb(b)$ are convex combinations of $\wu{}$ and $\wu{}'$.

It remains to prove (\ref{bary3}).
The condition $\wu{} < \emb({a}) <\emb({b}) <\wu{}'$ claimed in (\ref{bary3}) is equivalent to:
\begin{equation*}
\big({\sum\limits_{\lambda \in \pag{a}} \wtuv{a}{\lambda}}\big)\,\big({\sum\limits_{\mu \in \pag{b}} \wtuv{a}{\mu}}\big)^{-1} > \big({\sum\limits_{\lambda \in \pag{a}} \wtuv{b}{\lambda}}\big)\,\big({\sum\limits_{\mu \in \pag{b}} \wtuv{b}{\mu}}\big)^{-1}.
\end{equation*}
Rearranging these expressions by sums with common indexing sets and simplifying further using~\eqref{eq:linksavsb} yields the equivalent identity:
\begin{equation}\label{eq:equivIneq}
    1 >  \big({\sum_{\lambda \in \pag{a}} \wtuv{b}{\lambda}}\big)\,
    \big({\sum_{\lambda \in \pag{a}} \wtuv{a}{\lambda}}\big)^{-1} \,
    \big(\sum_{\mu \in \pag{b}} \wtuv{a}{\mu}\big)
    \big({\sum_{\mu \in \pag{b}} \wtuv{b}{\mu}}\big)^{-1} = 
    \frac{\wtuv{a}{b}}{\du{a}}\, \frac{\wtuv{a}{b}}{\du{b}}.
\end{equation}
The edge determinant condition for $[a,b]$ confirms the validity of~\eqref{eq:equivIneq}, and so (\ref{bary3}) holds.
\end{proof}

Our next result shows that the image under $\emb$ of the vertices adjacent to a fixed node $v$ satisfy a convexity property analogous to that of~\autoref{rem:convexityBaryc}:

\begin{proposition}\label{pr:convexityBar}
Let $v$ be a node of $\Gamma$, with adjacent vertices $u_1,\ldots, u_{\valv{v}}$.
Then:
\begin{enumerate}
\item \label{linIndep} $\{\emb(u_i):\; i=1, \ldots,\valv{v}\}$ is linearly independent;
  \item \label{inRelInt} $\displaystyle{\emb({v}) \in
\relo{\big(\conv(\{\emb(u_j) \colon j=1,\ldots, \valv{v}\})\big)}}$.
\end{enumerate}

\end{proposition}

\begin{proof} Both statements are clear if $v$ is only adjacent to leaves of $\Gamma$ (i.e., when $\Gamma$ is a star splice diagram), so we may assume $v$ is adjacent to some node of $\Gamma$. Thus, up to relabeling if necessary, we suppose $\{u_1,\dots, u_s\}$ are leaves of $\Gamma$ and $\{u_{s+1},\ldots, u_{\valv{v}}\}$ are nodes of $\Gamma$ (we set $s=0$  if $v$ is only adjacent to nodes).

  For each $j\in\{1,\ldots, \valv{v}\}$, we let $\Gamma_j$ be the branch of $\Gamma$ adjacent to $v$ containing $u_j$, and   set $\wu{j}:=\baryc{\leavesT{\Gamma_j}}{v}$. \autoref{lm:adjacentBarycenters} (\ref{bary2}) and the definition of barycenters   ensures the existence of $\alpha_j \in (0,1]$ satisfying
\begin{equation}\label{eq:proportional}
    \emb(u_j) = \alpha_j \wu{j} + (1 - \alpha_j)\emb(v)  \qquad \text{ for each }\; j\in \{1, \ldots, \valv{v}\},
\end{equation}
with $\alpha_j = 1$ if and only if $j\in \{1,\ldots, s\}$.

We start by proving~(\ref{linIndep}).  We fix a linear relation $\sum_{j=1}^{\valv{v}} \beta_j \, \emb(u_j) = 0$. Substituting~\eqref{eq:proportional} into this dependency relation  yields
  \begin{equation}\label{eq:depEquiv}\sum_{j=1}^{\valv{v}} (\beta_j\, \alpha_j) \, \wu{j} + B  \emb(v) = 0 \qquad \text{ where } \quad B:=\sum_{j=1}^{\valv{v}} \beta_j\,(1-\alpha_j) .
  \end{equation}
  We claim that $B=0$. Indeed, assuming this is not the case, we  use~\eqref{eq:depEquiv}  to write $\emb(v)$ in terms of $\wu{1},\ldots, \wu{\valv{v}}$. Comparing this expression with~\eqref{eq:lcombbarycenters} and using the linear independence of $\{\wu{1},\ldots, \wu{\valv{v}}\}$ gives:
  \begin{equation}\label{eq:valuesForBetai}\beta_j\,\alpha_j = - \frac{B}{|\wu{v}|} \sum_{\lambda \in \leavesT{T_j}} \wtuv{v}{\lambda} \quad \text{ for each }\; j\in\{1,\ldots, \valv{v}\}.
  \end{equation}
  In particular, all $\beta_i$ are non-zero and have the same sign, namely the opposite sign to $B$.
  Summing up the expressions~\eqref{eq:valuesForBetai} over all $j$ yields $\sum_{j=1}^{\valv{v}}\beta_j = 0$, which cannot happen due to the sign constraint on the $\beta_j$'s. From this it follows that $B=0$.

Since $B=0$, the linear independence of 
  $\{\wu{1}, \ldots, \wu{\valv{v}}\}$ forces $\beta_j\alpha_j=0$ for all $j$. Combining this with our assumption that $\alpha_j>0$ for all $j$, gives $\beta_j=0$ for all $j=1,\dots, \valv{v}$, thus confirms~(\ref{linIndep}).

To finish, we discuss~(\ref{inRelInt}). We let $q_1,\ldots, q_{\valv{v}}$ be the coefficients used in~\eqref{eq:lcombbarycenters} to write $\emb(v)$ as a convex combination of $\wu{1},\ldots, \wu{\valv{v}}$. Substituting the value of each $\wu{j}$ obtained from~\eqref{eq:proportional} in expression~\eqref{eq:lcombbarycenters} yields:
  \begin{equation}\label{eq:inRelInte0}
  \emb(v) = \sum_{j=1}^s  \frac{q_j}{A} \,  \emb(u_j)  + \sum_{j=s+1}^{\valv{v}}  \frac{q_j}{A\,\alpha_j} \, \emb(u_j) \qquad \text{ where } \quad A := 1+ \sum_{j=s+1}^{\valv{v}} \frac{q_j\,(1-\alpha_j)}{\alpha_j}.
  \end{equation}
  The conditions $0<\alpha_j<1$ for $j\in \{s+1,\ldots, \valv{v}\}$, and the definition of $q_1,\ldots, q_{\valv{v}}$ ensure that the right-hand side of~\eqref{eq:inRelInte0} is a positive convex combination of $\emb(u_1), \ldots, \emb(u_{\valv{v}})$, as we wanted to show.
  \end{proof}

Each subtree $T$ of $\Gamma$ determines a polytope via the map $\emb$: 
\begin{equation}\label{eq:simplexT}
   \boxedo{\simplex{T}}:=\conv(\{\emb(u) \colon  u\in \leavesT{T}\})\subseteq \simplex{n-1}. 
\end{equation}
For example, $\simplex{\Gamma}$ is the standard simplex $\simplex{n-1}$.
 We view the next result as a key convexity property of star-full subtrees of splice diagrams.

\begin{lemma}\label{lm:coeffspos}
Fix a star-full subtree $T$ of $\Gamma$. For every node $u$ of $T$, $\emb(u)$ admits an expression of the form
\begin{equation}\label{eq:ucoeffspos}
\displaystyle{\emb(u) = \sum_{\mu\in \leavesT{T}} \alpha_\mu \,\emb(\mu)}\quad 
\text{ with }\quad \displaystyle{\sum_{\mu\in \leavesT{T}} \alpha_\mu = 1}\quad
\text{ and }\quad \alpha_\mu>0\;\text{ for all } \mu \in \leavesT{T}.
\end{equation}
In particular, $\emb(T)\subset \simplex{T}$.
\end{lemma}
\begin{proof} 
We let $p$ be the number of nodes of $T$ and proceed by induction on  $p$. The statement is vacuous for $p=0$. If $p=1$,
then $T$ is a star tree and the result follows by~\autoref{pr:convexityBar}~(\ref{inRelInt}). For the inductive step, we let 
 $p\geq 2$  and suppose that the result holds for star-full subtrees with 
$(p-1)$ nodes.

 We fix $\leavesT{T} = \{u_1,\ldots, u_d\}$ and we let $v$ be an end-node of $T$. Following~\autoref{fig:CherryPruning}, we let $v'$ be the unique node of $T$ adjacent to 
$v$ and assume  that $u_1,\ldots, u_{\valv{v}-1}$ are adjacent to $v$.
\autoref{pr:convexityBar}~(\ref{inRelInt}) applied
to $\starT{\Gamma}{v}$ gives
\begin{equation}\label{eq:u}
\emb(v) =   \beta_{0} \,\emb(v') + \sum_{j=1}^{\valv{v}-1}\beta_j\,\emb(u_j)    \quad 
\text{ with }\quad \sum_{j=0}^{\valv{v}-1} \beta_j = 1 \quad \text{ and } \quad \beta_j>0 \;\text{ for all } j.
\end{equation}

Using~\autoref{lm:cherryPruning}, we let $T'$ be the star-full
subtree of $\Gamma$ obtained by pruning $T$ from $v$. By construction, $T'$ has $(p-1)$ nodes 
and its leaves are 
$\{v\} \cup \{u_{\valv{v}},\ldots, u_d\}$.
The inductive hypothesis yields the following expressions  for $v'$ and all other (potential) nodes $u\neq v'$ of $T'$:
\begin{equation}\label{eq:vu'}
\emb(v') =  \gamma_0\,\emb(v) + 
\sum_{j=\valv{v}}^d \gamma_j\, \emb(u_j)   \quad \text{ and } \quad     
\emb(u) =  \alpha_0\,\emb(v) + \sum_{j=\valv{v}}^d \alpha_j\, \emb(u_j),
\end{equation}
with $\displaystyle{\alpha_0 +  \sum_{j=\valv{v}}^d \alpha_i
  = \gamma_0+ \sum_{j=\valv{v}}^d \gamma_j = 1}$ and 
  $\alpha_i, \gamma_i > 0$ for all $i\in \{0,\valv{v},\ldots, d\}$.   
Since $0<\gamma_0\beta_0 <1$, substituting the expression for $\emb(v')$ obtained from~\eqref{eq:vu'}
into~\eqref{eq:u} produces the desired positive convex combination for $\emb(v)$:
\begin{equation*}
\emb(v) = \sum_{j=1}^{\valv{v}-1} \frac{\beta_j}{1-\gamma_0\beta_0}\,\emb(u_j) +  \sum_{j=\valv{v}}^d \frac{\beta_0\gamma_j}{1-\gamma_0\beta_0}\,\emb(u_j) .
\end{equation*}
In turn, substituting this identity in both expressions from~\eqref{eq:vu'} gives the positive convex combination statement for all remaining  nodes  of $T$. The inclusion $\emb(T)\subset \simplex{T}$ follows by the convexity of $\simplex{T}$.
\end{proof}

\autoref{pr:convexityBar} and~\autoref{lm:coeffspos} combined have the following natural consequence:
\begin{corollary}\label{cor:inclusionOfsimplices}
  For each pair of star-full subtrees $T, T'$ of $\Gamma$ with $T'\subseteq T$,  we have $\simplex{T'} \subseteq \simplex{T}$.
\end{corollary}

Our next result is a generalization of~\autoref{pr:convexityBar} and it highlights a key combinatorial property shared by  $\Gamma$ and all its star-full subtrees.
\begin{proposition}\label{pr:convexitySimplices} 
 Let $T$ be a star-full subtree of $\Gamma$. Then:
  \begin{enumerate}
  \item \label{item:LIStar-Full} the weights $\{\emb(u)\colon u \in \leavesT{T}\}$ are linearly independent;
  \item \label{item:simplexStar-Full}$\simplex{T}$ is a simplex of dimension $|\leavesT{T}|-1$;
  \item \label{item:vInStar-Full} for each node $v$ of $T$ we have $\emb(v) \in \relo{\simplex{T}}$.
  \end{enumerate}
\end{proposition}

\begin{proof} Item~(\ref{item:simplexStar-Full}) is a direct consequence of~(\ref{item:LIStar-Full}). In turn, item~(\ref{item:vInStar-Full}) follows from~(\ref{item:simplexStar-Full}) and~\autoref{lm:coeffspos}. Thus, it remains to prove~(\ref{item:LIStar-Full}).  
  We distinguish two cases, depending on the number of nodes of $T$, denoted by $p$.

\textbf{Case 1:}  If $p< 2$, then  $T$ is either a vertex, an edge of $\Gamma$, or a star tree. If $T$ is a vertex of $\Gamma$, then the claim holds because $\emb(u)\neq 0$ for any vertex $u$ of $\Gamma$. If $T$ is a star tree, the statement agrees with~\autoref{pr:convexityBar}~(\ref{linIndep}).

  Next, assume $T$ is an edge of $\Gamma$. We consider two scenarios. First, if $T$ joins a leaf $\lambda$ and a node $u$ of $\Gamma$, then the result follows immediately since $\emb(u)$ and $\wu{\lambda}$ are linearly independent.
On the contrary, assume $T$ joins two adjacent nodes of $\Gamma$, say $u$ and
$v$. Pick two leaves $\lambda, \mu$ of $\Gamma$ with $ u\in [\lambda, v]$ and $v \in [u,\mu]$ (i.e., $\lambda$ is on the 
$u$-side and $\mu$ is on the $v$-side of $\Gamma$, as seen from the edge $[u,v]$). 
\autoref{lm:usefulIdentityLinkingEdge}
yields the following formula for the $(\lambda, \mu)$-minor of the
matrix $(\wu{u} | \wu{v})$, involving the determinant  $\det([u,v])$ of the edge $[u,v]$, as defined in~\eqref{eq:edgeDet}:
\[  \begin{pmatrix}
    \wtuv{\lambda}{u} &     \wtuv{\lambda}{v}\\
\wtuv{\mu}{u} & \wtuv{\mu}{v}\\
    \end{pmatrix} = \wtuv{\lambda}{u}\,\wtuv{\mu}{v}\,\frac{\du{u,v} \du{v,u} - \wtuv{u}{v} }{\du{u,v}\du{v,u}} = \frac{\wtuv{\lambda}{u}\,\wtuv{\mu}{v} \det([u,v])}{\du{u,v}\du{v,u}}.
\]
The last expression is positive by the edge determinant condition, so  $\{\emb(u),\emb(v)\}$ is linearly independent.

\textbf{Case 2:}  If $p\geq  2$, we know that $d:=|\leavesT{T}|\geq 4$. We prove the result by reverse induction on $d$.  When $d=n$, we have $T=\Gamma$ and there is nothing to prove since $\emb(\lambda)=\wu{\lambda}$ for all $\lambda\in \leavesT{\Gamma}$. 
For the inductive step, we take $d<n$ and label the leaves of $T$ by
$u_1$, $u_2,\ldots, u_d$. We assume that the linear independence holds for any star-full subtree with $k\geq d+1$ leaves. Without loss of generality we assume $u_1$ is a node of $\Gamma$ (one must exist since $T\neq \Gamma$ is star-full).  Set $s=\valv{u_1}-1$.

  Next, we define $T' := T\cup \starT{\Gamma}{u_1}$. By construction, $T'$ is a star-full subtree of $\Gamma$ with $(d+s-1)$ leaves.  
Since $u_1$ is a node of $T'$,~\autoref{lm:coeffspos} applied to $T'$ yields a positive convex combination:
\begin{equation}\label{eq:newRhoU}
  \emb(u_1) = \sum_{v\in \leavesT{T'}} \beta_v \,\emb(v)
\quad 
\text{ with }\quad  \sum_{v\in \leavesT{T'}} \beta_v = 1 \quad \text{ and }\quad \beta_v>0 \quad \text{for all } v\in \leavesT{T'}. 
\end{equation}

To prove the linear independence for the $d$ points $\emb(u_1),\ldots, \emb(u_d)$, we fix a potential  dependency relation   $\sum_{j=1}^d\alpha_j \, \emb(u_j) = 0 $.
Substituting~\eqref{eq:newRhoU} into it gives a linear dependency relation for the leaves of $T'$:
\[
\sum_{j=2}^d(\alpha_j + \alpha_1\, \beta_{u_j})\, \emb(u_j) + \sum_{j=d+1}^{d+s} (\alpha_1\,\beta_{u_j}) \emb(u_j)=0,
\]
where $\{u_{d+1},\ldots, u_{d+s}\}$ are the leaves of $T'$ adjacent to $u_1$.
The inductive hypothesis applied to $T'$ and the positivity of each $\beta_{v}$ with $v\in \leavesT{T'}$ forces  $\alpha_j=0$ for all $j=1,\ldots, d$. Thus, (\ref{item:LIStar-Full}) holds.
\end{proof}

Next, we state the main result in this section, which is a natural consequence of~\autoref{lm:coeffspos}:

\begin{theorem}\label{thm:injectivityrho} 
The map $\emb$ from~\eqref{eq:embedGamma} is injective.
\end{theorem}

\begin{proof} 
We prove that the statement holds when restricted to any star-full subtree $T$ of $\Gamma$. As in the proof of~\autoref{lm:coeffspos}, we argue by induction on the number $p$ of nodes of $T$.
If $p=0$, then $T$ is either a vertex $u$ or an edge $[u,v]$. The statement in the first case is  tautological. The result for the second one holds by construction because $\{\emb(u), \emb(v)\}$ is linearly independent by~\autoref{pr:convexityBar}~(\ref{linIndep}).

If $p=1$, then $T$ is a star tree. Let $v$ be its unique node and $\{u_1,\ldots, u_{\valv{v}}\}$ be its leaves. Injectivity over 
$T$ is a direct consequence of the following identity:
\[\emb([u_i,v]) \cap \emb([u_j,v])=\{\emb(v)\}\quad \text{ for all }i\neq j,
\]
which we prove by a direct computation. Indeed, pick $0\leq a \leq b\leq 1$ with
\begin{equation}\label{eq:injStar}
  a\,\emb(u_i) + (1-a)   \,\emb(v) =    b\,\emb(u_j) + (1-b)   \,\emb(v).
\end{equation}

\noindent By~\autoref{pr:convexityBar}~(\ref{inRelInt}), $\emb(v)$ admits a unique
expression:
\[\emb(v) = \sum_{k=1}^{\valv{v}}\alpha_k\,\emb(u_k) \quad \text{ with } \quad 
\sum_{k=1}^{\valv{v}} \alpha_k = 1 \quad\text{ and }\quad \alpha_k>0 \;\text{ for all } k.
\]
Substituting this identity in~\eqref{eq:injStar} yields the following affine dependency relation for $\emb(\leavesT{T})$:
\[
  (a+(b-a) \alpha_i) \emb(u_i) + ((b-a)\alpha_j - b) \emb(u_j) + 
  \sum_{k\neq i,j} \alpha_k (b-a) \emb(u_k)=0
\]
By~\autoref{pr:convexityBar}~(\ref{linIndep}), we conclude that 
$a+(b-a)\alpha_i = (b-a)\alpha_j - b = 0$ and $\alpha_k(b-a) = 0$ for all $k\neq i,j$. Since $\alpha_k>0$ for all $k\neq i,j$ and $\valv{v}\geq 3$, 
it follows that $b-a=0$ and $a=-b=0$. Therefore, expression~\eqref{eq:injStar} represents 
$\emb(v)$.

Finally, pick $p\geq 2$ and assume the result holds for star-full subtrees with $p-1$ nodes. Let $T$ be a 
star-full subtree with $p$ nodes and pick  an end-node $v$ of $T$. As in~\autoref{fig:CherryPruning}, write $v'$ for the unique node of $T$ 
adjacent to it and $\{u_1,\ldots, u_{\valv{v}-1}\}$ for the leaves of $T$ adjacent to $v$.

As in~\autoref{lm:cherryPruning}, let $T'$ be the star-full subtree obtained
by pruning $T$ from $v$. Our inductive hypothesis ensures that $\emb$ is injective when
restricted to $T'$. By the $p=1$ case we know that $\emb([v,u_i])\cap \emb([v,u_j]) = \{\emb(v)\}$ if $i\neq j$. Thus, the injectivity of $\emb$ when restricted to $T$ will be proven if we show:
\begin{equation}\label{eq:injectIndStep}
\emb([v,u_i]) \cap \emb(T') = \{\emb(v)\} \quad \text{ for all }i=1,\ldots, \valv{v}-1.
\end{equation}
The identity follows from~\autoref{pr:convexitySimplices}. 
Indeed, we write any $\wu{}$ on the left-hand side of~\eqref{eq:injectIndStep} as
\begin{equation}\label{eq:wuAsConvex}
\wu{} := a\emb(u_i) + (1-a) \emb(v) \in  \emb([v,u_i]) \cap \emb(T')\quad \text{ with }\quad 0\leq a \leq 1.
\end{equation}

Recall that $\wu{}\in \emb(T')\subset \simplex{T'}$ and $\simplex{T'}\subset \simplex{T}$ by~\autoref{cor:inclusionOfsimplices}. 
 Since $\emb(v) \in \relo{(\simplex{T})}$ as in~\autoref{pr:convexitySimplices}~(\ref{item:vInStar-Full}), substituting this expression into~\eqref{eq:wuAsConvex} and comparing it with the known expression for $\wu{}$ as an element of $\simplex{T}$ yields an affine dependency equation for $\{\emb(u)\colon u \in \leavesT{T}\}$. The positivity constraint on the coefficients used to write $\emb(v)$ as an element of $\relo{(\simplex{T})}$ forces $a=0$, and so~\eqref{eq:injectIndStep} holds.
\end{proof}

\section{Local tropicalization of splice type systems}\label{sec:local-trop-newm}

Let $\Gamma$ be a splice diagram with $n$ leaves and let $\cS(\Gamma)$ be a
splice type system associated to it, as in~\autoref{def:splicesystem}. Fixing a
total order on $\leavesT{\Gamma}$ yields an embedding of the corresponding splice type singularity $(X, 0)$ into $\CC^n$. In this section we describe the local tropicalization of this embedded germ, following the characterization from 
\autoref{def:posloctrop}. As a byproduct, we confirm the first half of~\autoref{thm:main1}, namely that $(X,0)$ is a complete intersection in $\CC^n$ with no irreducible components contained in any coordinate subspace. 
\medskip

The injectivity of the map $\emb$ from~\eqref{eq:embedGamma}, discussed in~\autoref{thm:injectivityrho}, fixes a  natural simplicial fan structure on the cone over $\emb(\Gamma)$ in $\R^n$:

\begin{definition}  \label{def:spliceFan}
  Let $\Gamma$ be a splice diagram. Then, the set $\R_{\geq 0} \emb(\Gamma)$ has a natural fan structure, with top-dimensional cones 
  \[
  \{\Rp\emb([u,v]): [u,v] \text{ is an edge of }\Gamma\}.
  \]
We call it the \emph{splice fan of $\Gamma$}. 
\end{definition}

Here is the main result of this section:

 \begin{theorem}\label{thm:tropsG}
   The local tropicalization of $(X,0)\hookrightarrow \CC^n$ is supported on the splice fan of $\Gamma$.
 \end{theorem}

 We prove~\autoref{thm:tropsG} by a double inclusion argument, first restricting our attention  to the positive local tropicalization.  In~\autoref{ssec:proof-subseteq} we show that  the positive local tropicalization of $\sG{\Gamma}$ is contained in the  support of the splice fan of $\Gamma$.  We prove this fact by working with simplices  associated to star-full subtrees of $\Gamma$, which were introduced in~\autoref{def:fullSubtree}.  For clarity of exposition, we break the arguments into a series of combinatorial  lemmas and propositions. These results allow us to certify that the ideal generated by the $\wu{}$-initial forms of all the series $\fgvi{v}{i}$ in $\sG{\Gamma}$  always contains a monomial when $\wu{}$ lies in the complement of the splice fan of $\Gamma$ in $(\R_{>0})^n$.

In turn, showing that the support of the splice fan of $\Gamma$ lies in the Euclidean closure of the local tropicalization of $(X,0)$ involves the so-called balancing condition for pure-dimensional local tropicalizations. This is the subject of~\autoref{ssec:proof-supseteq}. An alternative proof will be given in~\autoref{sec:NewtonND} after proving the Newton non-degeneracy of the germ $(X,0)$.
 
 The fact that the positive local tropicalization of $(X,0)$ is pure-dimensional is verified in an indirect way.  Our proof technique relies on the explicit computation of the boundary components of the extended tropicalization, which is done in~\autoref{ssec:infinite-tropicalization}. This establishes the first half of~\autoref{thm:main1} discussed above (see~\autoref{cor:expDimensionsG}).  As a consequence,  we confirm by~\autoref{cor:closureOfPositiveTrop} that the local tropicalization is the Euclidean closure of the positive one. This result together with the findings in~\textcolor{blue}{Sections}~\ref{ssec:proof-subseteq} and~\ref{ssec:proof-supseteq} complete the proof of~\autoref{thm:tropsG}.

 \begin{remark}\label{rm:notationConvention}
   Throughout the next subsections, we adopt the following convenient notation for the admissible exponent vectors  $\wtNve{v}{e}$ from~\eqref{eq:admissibleMonExp}. 
   Given a node $v$ and a vertex $u$ of $\Gamma$ with $u\neq v$, we define  
   $\boxedo{\wtNve{v}{u}}:=\wtNve{v}{e}$ where $e$ is the unique edge adjacent to $v$ 
   and lying  in the geodesic $[v,u]$. Similarly, given a star-full subtree $T$ of $\Gamma$ 
   not containing $v$, we write $\boxedo{\wtNve{v}{T}}:=\wtNve{v}{u}$, where $u$ is any  vertex of $T$.
 \end{remark}

 \smallskip
 \subsection{The positive local tropicalization is contained in the support of the splice fan.}\label{ssec:proof-subseteq}
$\:$ 
\smallskip

In this subsection we show that the only points in $\simplex{n-1}$ contained in the positive local tropicalization $\ptrop \langle\sG{\Gamma}\rangle$ are included in $\emb(\Gamma)$. We exploit the terminology and convexity results stated in~\autoref{sec:comb-embedd-splice}.

Our first technical result will  be used extensively throughout this section to determine $\ptrop \langle\sG{\Gamma}\rangle$. As the proof shows, the Hamm determinant conditions imposed on the system $\sG{\Gamma}$ play a crucial role.

\begin{lemma}\label{lm:threeMonomialsAreEnough} 
    Fix a node $v$ of $\Gamma$ and let  $e,e',e''$  be three distinct edges of
      $\starT{\Gamma}{v}$. Fix $\wu{} \in (\Rp)^n$ and suppose that the admissible exponent vectors $\wtNve{v}{e}, \wtNve{v}{e'}, \wtNve{v}{e''}$ from~\eqref{eq:admissibleMonExp} satisfy:
    \begin{equation}\label{eq:eepepp} \wu{}\cdot \wtNve{v}{e} < \wu{}\cdot \wtNve{v}{e'} \quad \text{ and } 
       \quad \wu{}\cdot \wtNve{v}{e} < \wu{}\cdot \wtNve{v}{e''}.
    \end{equation}
    Then, $\zexp{\wtNve{v}{e}}=\initwf{\wu{}}{f}$ for some  $f$ in the linear span of $\{\fvi{v}{i}: i=1\ldots, \valv{v}-2\}$. 
    If, in addition, $\wu{}$ satisfies 
    \begin{equation}\label{eq:hotDontCount}
\wu{}\cdot \wtNve{v}{e} <     \wu{}\cdot m \quad \text{ for each }\; {m} \in \bigcup_{i=1}^{\valv{v}-2} \Supp(\gvi{v}{i}),
    \end{equation}
then $\zexp{\wtNve{v}{e}}=\initwf{\wu{}}{F}$ for some series $F$  in the linear span of $\{\fgvi{v}{i}: i=1\ldots, \valv{v}-2\}$. In particular,     $\wu{}\notin \Trop \langle\sG{\Gamma}\rangle $.
\end{lemma}

\begin{proof}  We let  $e_1,\ldots, e_{\valv{v}}$ be the edges adjacent to $v$ and assume that $e_{\valv{v}-2}:=e$, $e_{\valv{v}-1}:=e'$ and $e_{\valv{v}}:=e''$. 
  Using the Hamm determinant conditions, we build a basis $\{\fvi{v}{i}'\}_{i=1}^{\valv{v}-2}$ for the linear span of $\{\fvi{v}{i}\}_{i=1}^{\valv{v}-2}$ where
  \begin{equation*}\label{eq:newBasis}\fvi{v}{i}':= \zexp{\wtNve{v}{e_i}} + a_{i} \, \zexp{\wtNve{v}{e'}} + b_i\, \zexp{\wtNve{v}{e''}} \quad \text{ for each }\; i\in \{1,\ldots, \valv{v}-2\}.
  \end{equation*}
 From~\eqref{eq:eepepp} we conclude that $\initwf{\wu{}}{\fvi{v}{\valv{v}-2}'} = \zexp{\wtNve{v}{e}}$. Taking $f:=\fvi{v}{\valv{v}-2}'$ proves the first part of the statement.

 For the second part, the technique yields a new basis $\{\fgvi{v}{i}'\}_{i=1}^{\valv{v}-2}$ for the linear span of $\{\fgvi{v}{i}'\}_{i=1}^{\valv{v}-2}$ with
 \[\fgvi{v}{i}' = \fvi{v}{i}' + \gvi{v}{i}'\qquad \text{ for each }\;i\in \{1,\ldots, \valv{v}-2\},\]
 where each $\gvi{v}{i}'$ is a linear combination of $\{\gvi{v}{j}\}_{j=1}^{\valv{v}-2}$. Condition~\eqref{eq:hotDontCount} then ensures that
\[\initwf{\wu{}}{\fgvi{v}{\valv{v}-2}'} = \initwf{\wu{}}{\fvi{v}{\valv{v}-2}'} = \zexp{\wtNve{v}{e}}.
\]
Thus, the series $F= \fgvi{v}{\valv{v}-2}'$ satisfies the required properties. In particular, the ideal $\initwf{\wu{}}{\langle\sG{\Gamma}\rangle}$ contains the monomial $\zexp{\wtNve{v}{e}}$ and so $\wu{}\notin \Trop \langle\sG{\Gamma}\rangle$ by definition.
\end{proof}

Next, we state the main theorem in this subsection, which yields one of the required inclusions in~\autoref{thm:tropsG} when choosing $T=\Gamma$. More precisely:

\begin{theorem}\label{thm:hardInclusion2D} For every star-full subtree $T$ of $\Gamma$, we have $ \simplex{T}\cap \ptrop \langle\sG{\Gamma}\rangle  \subseteq \emb(T)$.
\end{theorem}

\begin{proof} Recall from~\eqref{eq:simplexT} that $\simplex{T}$ is the convex hull of the set of leaves $\leavesT{T}$ of $T$, viewed in $\simplex{n-1}$ via the map $\emb$.   We proceed by induction on the number of nodes of $T$, which we denote by $p$. If $p=0$, then $T$ is either a vertex or an edge of $\Gamma$, and $\simplex{T} = \emb(T)$. For the inductive step, assume $p\geq 1$ and pick a node $v$ of $T$. Let $T_1,\ldots, T_{\valv{v}}$ be the branches of $T$ adjacent to $v$, as in \autoref{def:branches}. We use the point $\emb(v)$ to perform a stellar subdivision of $\simplex{T}$, giving a decomposition  $\simplex{T} = \bigcup_{\lambda\in \leavesT{T}} \tau_{v,\lambda}$, where
\begin{equation}\label{eq:tau_lambda}
 \boxedo{\tau_{v,\lambda}}:=\conv(\{\emb(v)\}\cup \emb(\leavesT{T}\smallsetminus \{\lambda\})) \qquad \text{ for all }\;\lambda \in \leavesT{T}.
\end{equation}

By~\autoref{lm:coeffspos}, $\tau_{v,\lambda}$ is a simplex of dimension $|\leavesT{T}|-1$. \autoref{pr:interiorAvoidsTrop} below shows that $\tau_{v,\lambda} \cap \ptrop\sG{\Gamma}$ lies in the boundary of  $\tau_{v,\lambda}$. In turn,~\autoref{pr:noMixing} below ensures that 
\[\partial\tau_{v,\lambda} \cap \ptrop \langle\sG{\Gamma}\rangle \subseteq \bigcup_{\substack{1\leq i \leq \valv{v}\\i\neq j}}\simplex{[T_i,v]},
     \]
     where $T_j$ is the unique branch of $T$ adjacent to $v$ that contains the leaf $\lambda$, 
     and $\boxedo{[T_i,v]}$ denotes the convex hull in $\Gamma$ of $T_i \cup \{v\}$. 
     Combining this fact with the inductive hypothesis applied to all  
     star-full subtrees $[T_i,v]$ of $\Gamma$ with $i \in \{1, \ldots, \valv{v}\}$ gives
\begin{equation}\label{eq:onlyBoundaryMatters}
  \simplex{T} \cap \ptrop \langle\sG{\Gamma}\rangle \subseteq \bigcup_{i=1}^{\valv{v}} (\simplex{[T_i,v]} \cap \ptrop \langle\sG{\Gamma}\rangle)\subseteq \bigcup_{i=1}^{\valv{v}}\emb([T_i,v]) \subseteq \emb(T).\qedhere
\end{equation}
\end{proof}

As a natural consequence of this result, we deduce one of the two inclusions required to confirm~\autoref{thm:tropsG}:

\begin{corollary}\label{cor:hardInclusion2D} 
      The positive local tropicalization of $\sG{\Gamma}$ is contained in the support 
      of the splice fan of $\Gamma$.
\end{corollary}

In the remainder of this subsection, we discuss the two key propositions used in the proof of~\autoref{thm:hardInclusion2D}. We start by showing that the relative interior of any of the top-dimensional simplices $\tau_{v,\lambda}$ from~\eqref{eq:tau_lambda} obtained from the stellar subdivision of $\simplex{T}$ induced by a node $v$ of $T$ does not meet $\ptrop \langle\sG{\Gamma}\rangle$.~\autoref{lm:threeMonomialsAreEnough} plays a central role. The task is purely combinatorial and the difficulty lies in how to select the triple of admissible exponent vectors required by the lemma. Throughout, we make use of branches of subtrees, which were introduced in~\autoref{def:branches}. 

\begin{proposition}\label{pr:interiorAvoidsTrop} 
Let $T$ be a  star-full subtree of $\Gamma$ and $\lambda \in \leavesT{T}$. 
     Consider the simplex $\tau_{v,\lambda}$ defined by \eqref{eq:tau_lambda}. Then, its relative interior $\relo{\tau_{v,\lambda}} $ is disjoint from $\ptrop \langle\sG{\Gamma}\rangle$.  
\end{proposition}

\begin{proof}
  Let $u$ be the unique node of $T$ adjacent to $\lambda$, and denote by $T_1,\ldots, T_{\valv{u}}$ the branches of $T$ adjacent to $u$. We assume that $T_{\valv{u}}=\{\lambda\}$ and fix any $\wu{}\in \relo{\tau_{v,\lambda}}$.
  Since $\tau_{v,\lambda}$ is a simplex, we can  write $\wu{}$ uniquely as
  \begin{equation}\label{eq:wjExpression}
         \wu{} = \alpha_v\emb(v) + \sum_{j=1}^{\valv{u}-1} \wu{j} \quad \text{ where } \quad     
          \boxedo{\wu{j}}  :=  \sum_{\mu \in \leavesT{T_j}} \alpha_{\mu,j} \emb(\mu)
           \quad    \text{ for all }\; j\in \{1,\ldots, \valv{u}-1\}, 
  \end{equation}
and $\boxedo{\alpha_v},  \boxedo{\alpha_{\mu,j}} >0$ for all $\mu \in \leavesT{T_j}$.

In what follows we analyze the $\wu{}$-initial forms of the series $\fgvi{u}{i}$ from  $\sG{\Gamma}$ for $i\in \{1,\ldots, \valv{u}-2\}$  and use~\autoref{lm:threeMonomialsAreEnough} to confirm that $\wu{} \notin \Trop \langle\sG{\Gamma}\rangle$.  To this end, we compare the $\wu{}$-weights of 
$\zexp{\wtNve{u}{\lambda}}$ and of the remaining monomials in $\fgvi{u}{i}$, for each $i\in \{1,\ldots, \valv{u}-2\}$. We treat the monomials appearing in $\fvi{u}{i}$ and $\gvi{u}{i}$ separately.

\autoref{lm:directCalculationsIForOldProp6_7} discusses the monomials in $\gvi{u}{i}$ and confirms
that the required condition~(\ref{eq:hotDontCount}) of~\autoref{lm:threeMonomialsAreEnough} holds for $\wtNve{u}{\lambda}$. In turn, \autoref{lm:directCalculationsIIForOldProp6_7} verifies that we can find two edges  $e',e''$ of $T$ adjacent to $u$ satisfying the inequalities~\eqref{eq:eepepp} for $e=[u,\lambda]$. Therefore,~\autoref{lm:threeMonomialsAreEnough} applied to the node $u$ confirms that $\zexp{\wtNve{u}{\lambda}}$ is the $\wu{}$-initial form of a series in the linear span of $\{\initwf{\wu{}}{\fgvi{u}{i}}: i=1,\ldots, \valv{u}-2\}$. Thus, $\wu{}\notin \Trop \langle\sG{\Gamma}\rangle$ as we wanted to show.
\end{proof}

In the next two lemmas we place ourselves in the setting of \autoref{pr:interiorAvoidsTrop}. In particular, we use the notations introduced in its proof.

\begin{lemma}\label{lm:directCalculationsIForOldProp6_7} 
       Let $\wu{}$  be a weight vector satisfying condition~\eqref{eq:wjExpression}. 
       Then, for each $i\in \{1,\ldots, \valv{u}-2\}$ and each monomial $\zexp{m}$ appearing in $\gvi{u}{i}$ we have 
  \begin{equation}\label{eq:ineqForGui_muLambda}
       \wu{}\cdot \wtNve{u}{\lambda} < \wu{} \cdot m\,.
  \end{equation}
\end{lemma}

\begin{proof} 
    First, we compare the $\wu{}$-weights of $\zexp{\wtNve{u}{\lambda}}$ and all the monomials $\zexp{m}$ appearing in $\gvi{u}{i}$. We do so by looking at the weight 
contributed by each summand of $\wu{}$ in the decomposition \eqref{eq:wjExpression}. \autoref{lem:keyid} and conditions~\eqref{eq:gviConditions} ensure that
\begin{equation}\label{eq:weightOfEmbv}
      \emb(v)\cdot \wtNve{u}{\lambda} = \frac{\wtuv{u}{v}}{|\wu{v}|} < \emb(v) \cdot m\,.
\end{equation}

     In turn, to compare the $\wu{j}$-weights of the monomials $\zexp{\wtNve{u}{\lambda}}$ and $\zexp{m}$, we notice that the only summands $\alpha_{\mu,j}\emb(\mu)$ of $\wu{j}$ contributing a positive weight to $\zexp{\wtNve{u}{\lambda}}$ are the ones coming from those $\mu \in \leavesT{T_j}$ that are nodes in $\Gamma$. Again, \autoref{lem:keyid} and the conditions on $\gvi{u}{i}$ listed in~\eqref{eq:gviConditions} confirm that 
\[
      \emb(\mu) \cdot \wtNve{u}{\lambda} = \frac{\wtuv{u}{\mu}}{|\wu{\mu}|} < \emb(\mu) \cdot m \qquad \text{ for all }\mu\in  \leavesT{T_j}\smallsetminus \leavesT{\Gamma}\,.\]
Combining this inequality with the positivity of the coefficients $\alpha_{\mu,j}$ yields the inequality
\begin{equation}\label{eq:weightOfWj}
  \wu{j}\cdot \wtNve{u}{\lambda} =  \! \sum_{\mu \in \leavesT{T_j}\smallsetminus \leavesT{\Gamma}} \!  \alpha_{\mu,j} \frac{\wtuv{u}{\mu}}{|\wu{\mu}|}\;\; \leq \sum_{\mu \in \leavesT{T_j}\smallsetminus \leavesT{\Gamma}} \!  \alpha_{\mu,j}\, (\emb(\mu) \cdot m)
  \;\;\leq \;\;\wu{j} \cdot m \;\quad \text{ for } 1\leq j< \valv{u}\,. 
\end{equation}
      Furthermore, the leftmost inequality is strict if the set $\leavesT{T_j}\smallsetminus \leavesT{\Gamma}$ is non-empty.

    Finally, the positivity of $\alpha_v$ together with the inequalities~\eqref{eq:weightOfEmbv} and \eqref{eq:weightOfWj} ensure that each $m$ in the support of $\gvi{u}{i}$ satisfies the  inequality~\eqref{eq:ineqForGui_muLambda}. This concludes our proof.
\end{proof}

\begin{lemma}\label{lm:directCalculationsIIForOldProp6_7} 
       Let $\wu{}$  be a weight vector verifying condition~\eqref{eq:wjExpression}. 
      Then, there exist two different edges $e', e''$ adjacent to the node 
      $u$ of $\Gamma$ satisfying
         \[
                \wu{}\cdot \wtNve{u}{\lambda} < \wu{}\cdot \wtNve{u}{e'} \quad \text{ and } 
                  \quad \wu{}\cdot \wtNve{u}{\lambda} < \wu{}\cdot \wtNve{u}{e''}.
         \]
\end{lemma}

\begin{proof}
We use the  notation $\wtNve{u}{T}$ introduced in~\autoref{rm:notationConvention}. It is enough to find  
     two different branches $T_{j_1}$ and $T_{j_2}$ of $T$ adjacent to $u$ 
     with $j_1,j_2<\valv{v}$, and verifying
\begin{equation}\label{eq:ineqToCheckFirstProp}
     \wu{}\cdot \wtNve{u}{\lambda} < \wu{}\cdot \wtNve{u}{T_{j_1}}\quad  \text{ and } \quad \wu{}\cdot \wtNve{u}{\lambda} < \wu{}\cdot \wtNve{u}{T_{j_2}}.
  \end{equation}
  
 Our choice will depend on the nature of $u$. If $u=v$ we pick any pair of distinct indices $j_1,j_2<\valv{v}$. On the contrary, if $u\neq v$ we take $T_{j_1}$ to be the unique branch of $T$ adjacent to $u$ containing $v$, and let $T_{j_2}$ be any other branch of $T$ adjacent to $u$  not containing $\lambda$ with $j_2<\valv{v}$. After relabeling, we may assume that $j_1 = 1$ and $j_2=2$.

  In the remainder, we confirm that these two branches satisfy the inequalities in~\eqref{eq:ineqToCheckFirstProp} by analyzing the contributions of each summand of $\wu{}$ in the decomposition \eqref{eq:wjExpression} to the total weight of each of the three relevant monomials. We follow the same reasoning as in the proof of~\autoref{lm:directCalculationsIForOldProp6_7}. 
  \autoref{lem:keyid} is again central to our computations. 
  
 We start with the contribution of $\emb(v)$. The lemma confirms that 
  \begin{equation}\label{eq:directCalIIForv}
         \emb(v)\cdot \wtNve{u}{\lambda} =   \frac{\wtuv{u}{v}}{|\wu{v}|} =  \min\{\emb(v)\cdot \wtNve{u}{T_1} ,    \emb(v)\cdot \wtNve{u}{T_2}\}. 
  \end{equation}

     To analyze the $\wu{j}$-weight of the three monomials, we recall from the proof of \autoref{lm:directCalculationsIForOldProp6_7} that we need only consider the contributions of those vertices $\mu\in \leavesT{T_j}$ that are nodes of $\Gamma$. Indeed, we have 
\begin{equation}\label{eq:ineqj}
        \wu{j} \cdot  \wtNve{u}{\lambda} =  \! \sum_{\mu \in \leavesT{T_j}\smallsetminus \leavesT{\Gamma}} \!  \alpha_{\mu,j} \frac{\wtuv{u}{\mu}}{|\wu{\mu}|} \qquad \text{ for } j\in \{1,\ldots, \valv{u}-1\}\,. 
\end{equation}
     This expression agrees with the value of $\wu{j}\cdot \wtNve{u}{T_1}=\wu{j}\cdot \wtNve{u}{T_2}$ when $j>2$ even if $\leavesT{T_j}\smallsetminus \leavesT{\Gamma}$ is the empty-set.

        If $j=1$,~\autoref{lem:keyid} confirms that
\begin{equation}\label{eq:ineq1}
       \wu{1} \cdot  \wtNve{u}{\lambda} = \wu{1} \cdot  \wtNve{u}{T_2} =  \! \sum_{\mu \in \leavesT{T_1}\smallsetminus \leavesT{\Gamma}} \!  \alpha_{\mu,1} \frac{\wtuv{u}{\mu}}{|\wu{\mu}|} < \wu{1} \cdot  \wtNve{u}{T_1}\,. 
\end{equation}
       The inequality is strict even if the set $\leavesT{T_1}\smallsetminus \leavesT{\Gamma}$ is empty. Indeed, the fact that the tree $T$ is star-full ensures that when $\leavesT{T_1}\subset \leavesT{\Gamma}$, any variable appearing in $\zexp{\wtNve{u}{T_1}}$ will be indexed by an element in $\leavesT{T_1}$. Therefore,  the total $\wu{1}$-weight of $\wtNve{u}{T_1}$ is positive no matter the choice of admissible exponent $\wtNve{u}{T_1}$.

     By symmetry, we deduce that $\wu{2} \cdot  \wtNve{u}{\lambda} = \wu{2} \cdot  \wtNve{u}{T_1}  < \wu{2} \cdot  \wtNve{u}{T_2}$. Combining this fact, the positivity of $\alpha_{v}$, and  the expressions~\eqref{eq:directCalIIForv},~\eqref{eq:ineqj} and~\eqref{eq:ineq1} yields the inequalities in~\eqref{eq:ineqToCheckFirstProp}. This concludes our proof.
\end{proof}

Our next result is central to the inductive step in the proof of~\autoref{thm:hardInclusion2D}. As with the previous two lemmas, the proof is combinatorial and the difficulty lies in how to select the triple of admissible exponent vectors required by~\autoref{lm:threeMonomialsAreEnough} in a way that is compatible with a given input proper face of $\tau_{v,\lambda}$.

\begin{proposition}\label{pr:noMixing}
  Let $T$ be a star-full subtree of $\Gamma$, and let  $v$ be a node of $T$. 
  Denote by $T_1,\ldots, T_{\valv{v}}$ the branches of $T$ adjacent to  $v$.  Let $L$ be a proper subset of $\leavesT{T}$ 
  which is not included in any $\leavesT{T_i}$, and set
     \begin{equation}\label{eq:cP}
          \boxedo{\cP} := \conv(\{\emb(v)\} \cup \emb(L)).
     \end{equation}
  Then, $\cP$ is a simplex of dimension $|L|$ and its relative interior 
  $\relo{\cP}$ is disjoint from $\ptrop\sG{\Gamma}$.
\end{proposition}

\begin{proof} By \autoref{lm:coeffspos}, we have $\emb(v)\in \relo{(\simplex{T})}$. In addition,  \autoref{pr:convexitySimplices}~(\ref{item:simplexStar-Full}) and (\ref{item:vInStar-Full}) imply that $\cP$ is a simplex of the expected dimension. It remains to show that $\relo{\cP} \cap \ptrop \langle\sG{\Gamma}\rangle = \emptyset$. 
  For each $j=1,\ldots, \valv{v}$, we set
  \begin{equation}\label{eq:Li}
          \boxedo{L_j}  :=L\cap \leavesT{T_j} \qquad \text{ and } \qquad  
            \boxedo{\tau_j} :=\Rp \langle \emb(L_j)\rangle.
  \end{equation}
  By definition, we have $\tau_j=\{\mathbf{0}\}$ if $L_j=\emptyset$. 
  Moreover, each $\tau_j$ is a simplicial cone of dimension $|L_j|$.
  Our assumptions on $L$ and a suitable relabeling of the  branches $T_j$ (if necessary) ensure the existence of some $q \in \{2,\ldots, \valv{v}\}$ with $L_j\neq \emptyset$ for all $j\in \{1,\ldots, q\}$ and $L_{j}=\emptyset$ for all $j\in \{q+1, \ldots, \valv{v}\}$.

We argue by contradiction and pick $\wu{} \in \relo{\cP}\cap  \ptrop \langle\sG{\Gamma}\rangle$. Since $\cP$ is a simplex and $w$ belongs to its relative interior, we may write $\wu{}$ as 
  \begin{equation}\label{eq:wwInMixedBatch}
             \wu{} = \alpha_v\emb(v) + \sum_{j=1}^{q} \wu{j} \qquad \text{ with } \alpha_v>0 \; 
                \text{ and } \; \wu{j} \in \relo{\tau_j}  \text{ for all }j.
  \end{equation}
  \autoref{lm:properIsEmpty} below implies that $L_j = \leavesT{T_j}$ for all $j\leq q$.   Since $L\subsetneq \leavesT{T}$, it follows that  $q<\valv{v}$.

  Next, we use~\autoref{lm:threeMonomialsAreEnough} to confirm that $\zexp{\wtNve{v}{T_{\valv{v}}}}$ is the $\wu{}$-initial form of a series in the linear span of $\{\initwf{\wu{}}{\fgvi{v}{i}}\}_{i=1}^{\valv{v}-2}$, which contradicts our  assumption $\wu{} \in \ptrop \langle\sG{\Gamma}\rangle$. We use the admissible exponents $\wtNve{v}{T_{1}}$, $\wtNve{v}{T_{2}}$ and $\wtNve{v}{T_{\valv{v}}}$.

First, the star-full property of $T$ combined with the condition that $L_j=\leavesT{T_j}$ for all $j\leq q$ ensures that
\begin{equation}\label{eq:vvalueSecondrop}
  \begin{aligned}
\emb(v) \cdot  \wtNve{v}{T_{\valv{v}}}&    =   \emb(v) \cdot  \wtNve{v}{T_{2}} =  \emb(v) \cdot \wtNve{v}{T_{1}} = \frac{\du{v}}{|\wu{v}|} \qquad 
 \text{and } \\
    \wu{j} \cdot \wtNve{v}{T_{\valv{v}}}&   = \wu{j} \cdot \wtNve{v}{T_1} =  \wu{j} \cdot \wtNve{v}{T_2} \quad \text{ for all }\;j \in \{3, \ldots, q\}.
  \end{aligned}
\end{equation}
These identities follow directly from~\autoref{lem:keyid}. In turn, the same arguments employed in the proof of~\autoref{lm:directCalculationsIIForOldProp6_7}, together with  the equalities $L_1=\leavesT{T_1}$ and $L_2=\leavesT{T_2}$ give 
\begin{equation}\label{eq:weightsw1w2OldProp6_8}
    \wu{1} \cdot \wtNve{v}{T_{\valv{v}}}   = \wu{1} \cdot \wtNve{v}{T_2} <  \wu{1}  \cdot \wtNve{v}{T_1} \quad \text{ and }\quad     \wu{2} \cdot \wtNve{v}{T_{\valv{v}}}   = \wu{2} \cdot \wtNve{v}{T_1} <  \wu{2}  \cdot \wtNve{v}{T_2}.
\end{equation}

 Expressions~\eqref{eq:vvalueSecondrop} and~\eqref{eq:weightsw1w2OldProp6_8} combined yield $\wu{} \cdot \wtNve{v}{T_{\valv{v}}} < \wu{} \cdot \wtNve{v}{T_{1}}$ and $\wu{} \cdot \wtNve{v}{T_{\valv{v}}} < \wu{} \cdot \wtNve{v}{T_{2}}$. Thus, the first condition required by~\autoref{lm:threeMonomialsAreEnough} is satisfied.
        
To finish, we must compare the $\wu{}$-weight of $\zexp{\wtNve{v}{T_{\valv{v}}}}$ with that of any exponent  $m$ in the support of a fixed series $\gvi{v}{i}$. For each $j\in \{1,\ldots, q\}$ we write $\wu{j} :=  \sum_{\mu \in \leavesT{T_j}} \alpha_{\mu,j} \emb(\mu)$ with $\alpha_{\mu,j}>0$ for all $\mu$. Then, the defining  properties \eqref{eq:gviConditions} of $\gvi{v}{i}$ and the reasoning followed in the proof of~\autoref{lm:directCalculationsIForOldProp6_7} imply that
        \begin{equation}\label{eq:Prop76SeriesComparison}
     \emb(v) \cdot  \wtNve{v}{T_{\valv{v}}}  <  \emb(v) \cdot m \quad \text{ and }\quad \wu{j} \cdot \wtNve{v}{T_{\valv{v}}} =\!\!\!\! \sum_{\mu \in L_j\smallsetminus \leavesT{\Gamma}} \!\!\!\!  \alpha_{\mu,j} \frac{\wtuv{v}{\mu}}{|\wu{\mu}|} \leq \wu{j} \cdot m \quad \text{ for } 1\leq j \leq q\,.
        \end{equation}
    Note that the right-most inequality for $j \in \{1\ldots, q\}$ is strict whenever $L_j \nsubseteq \leavesT{\Gamma}$.
    
        Combining both parts of~\eqref{eq:Prop76SeriesComparison} and the positivity of $\alpha_v$ yields $\wu{}\cdot \wtNve{v}{T_{\valv{v}}} < \wu{} \cdot m$ whenever  $\zexp{m}$ appears in $\gvi{v}{i}$. This verifies the second hypothesis required for~\autoref{lm:threeMonomialsAreEnough}, contradicting our choice of $\wu{}\in \ptrop \langle\sG{\Gamma}\rangle$.
\end{proof}

\begin{lemma}\label{lm:properIsEmpty} 
Let $L$ be a set as in \autoref{pr:noMixing} and $\cP$ be the simplex defined by formula \eqref{eq:cP}. 
    For each $j\in\{1,\ldots, \valv{v}\}$, let $L_j$ be the set of elements of $L$ which are leaves of $T_j$,  as in~\eqref{eq:Li}. 
    If $\relo{\cP}\cap \ptrop \langle\sG{\Gamma}\rangle\neq \emptyset$, then $L_j$ is either empty or equal to  $\leavesT{T_j}$. 
\end{lemma}

\begin{proof} We argue by contradiction and assume $\emptyset \subsetneq L_j \subsetneq \leavesT{T_j}$, so in particular $|\leavesT{T_j}|>1$. We break the argument into four combinatorial claims, guided by~\autoref{fig:keyLemma2D}. The left-most picture informs the discussion for \textcolor{blue}{Claims}~\ref{cl:vpEndNode} and~\ref{cl:lambdaInProperBranch}. The central picture refers to \autoref{cl:NoLeafAdjToUInLi}, and the right-most picture illustrates~\autoref{cl:NoBranchAtUMeetsLi}. Throughout, we fix $\wu{}\in \relo{\cP} \cap \ptrop \langle\sG{\Gamma}\rangle$ and write $\wu{}$ as in~\eqref{eq:wwInMixedBatch},
where we write each  $\wu{k} \in \relo{\tau_k}$ as $\wu{k}:=\sum_{\mu \in L_k} \alpha_{\mu,k} \emb(\mu)$ for each $k\in \{1,\ldots, q\}$, with $\alpha_{\mu,k}>0$ 
for all $\mu,k$.

   \begin{figure}[tb]
    \includegraphics[scale=0.5]{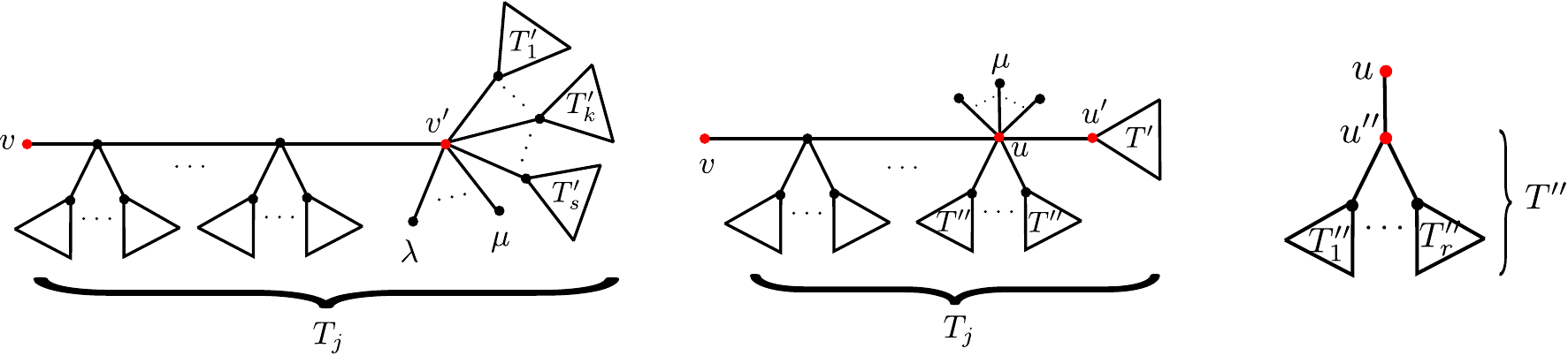}
      \caption{Building a maximal sub-branch of $T_j$ avoiding $L_j$ starting from a suitable leaf $\lambda\in \leavesT{T_j}\smallsetminus L_j$ and moving inwards towards $v$, as in the proof of~\autoref{lm:properIsEmpty}. Here, $s\leq \valv{v'}-2$ and $r=\valv{u''}-1$.\label{fig:keyLemma2D}}
    \end{figure}

   First, we pick $\lambda \in \leavesT{T_j}\smallsetminus L_j$ furthest away from $v$ in the geodesic metric on $T$ (see~\autoref{rm:geodesics}). Let $v'$ be the unique node of $T_j$ adjacent to $\lambda$. The condition $|\leavesT{T_j}|>1$ ensures that $v'\neq v$. The maximality condition satisfied by $\lambda$ restricts the nature of the node $v'$. More precisely:
  \begin{claim}\label{cl:vpEndNode} The node $v'$ is an end-node of $T_j$.
  \end{claim}
  \begin{subproof} We consider all branches $T_k'$ of $T_j$ adjacent to $v'$ and  containing neither $v$ nor $\lambda$. Our goal is to show that  $|\leavesT{T_k'}| = 1$ for all $k$. We argue by contradiction.

Assume that $k$ is such that $|\leavesT{T_k'}|>1$. Then, by the maximality of the distance between $v$ and $\lambda$, we have $\leavesT{T_k'} \subseteq L_j$. We consider the series $\fgvi{v'}{i}$ of $\sG{\Gamma}$ for $i\in \{1,\ldots, \valv{v'}-2\}$, and the admissible exponent vectors at $v'$ associated to $\lambda$, $v$ and $T_k'$. We claim that the weight $\wu{}$ satisfies the inequalities
      \begin{equation}\label{eq:wweightCase1}
           \wu{} \cdot \wtNve{v'}{\lambda} < \wu{} \cdot \wtNve{v'}{T_k'}, 
           \qquad \wu{} \cdot \wtNve{v'}{\lambda} < \wu{} \cdot \wtNve{v'}{v} \quad \text{ and }   
           \quad    \wu{} \cdot \wtNve{v'}{\lambda} < \wu{} \cdot m
      \end{equation}
      for each $m$ in the support of some $\gvi{v'}{i}$. This cannot happen 
      by~\autoref{lm:threeMonomialsAreEnough}, since $\wu{}\in \ptrop \langle\sG{\Gamma}\rangle$.

      To prove the inequalities \eqref{eq:wweightCase1}, we analyze the contributions of each summand of $\wu{}$ in the decomposition \eqref{eq:wwInMixedBatch} to the total $\wu{}$-weight of each admissible monomial, as we did in the proofs of~\textcolor{blue}{Lemmas}~\ref{lm:directCalculationsIForOldProp6_7} and~\ref{lm:directCalculationsIIForOldProp6_7}. Here, the node $u$ appearing in those lemmas is replaced by the node $v'$. As usual, \autoref{lem:keyid} is central to our arguments. The contribution of each weight vector $\wu{p}$ for $p\neq j$ to the total weight of $\zexp{\wtNve{v'}{\lambda}}$ comes from elements in  $L_p\smallsetminus \leavesT{\Gamma}$.

      We start by verifying the left-most inequality in~\eqref{eq:wweightCase1}. Since $v$ is a node of $\Gamma$, \autoref{lem:keyid}  ensures that:
    \begin{equation}\label{eq:case1vweight}
         \emb(v) \cdot \wtNve{v'}{\lambda}  =  \frac{\wtuv{v'}{v}}{|\wu{v}|}  = 
            \emb(v) \cdot \wtNve{v'}{T'_k} <  \emb(v) \cdot \wtNve{v'}{v}\,.
    \end{equation}
  In turn, for each $p\in \{1,\ldots, q\}$ with $p\neq j$, the weight $\wu{p}$ satisfies
    \begin{equation}\label{eq:case1weightLambdaTkp}
        \wu{p} \cdot \wtNve{v'}{\lambda} = \sum_{\mu \in L_p \smallsetminus 
         \leavesT{\Gamma}} \alpha_{\mu,p} \frac{\wtuv{v'}{\mu}}{|\wu{\mu}|} = \wu{p}\cdot \wtNve{v'}{T_k'}\,. 
     \end{equation}
   On the contrary, if $p=j$, the condition $\leavesT{T_k'}\subseteq L_j$ ensures that 
   $\wu{j} \cdot \wtNve{v'}{\lambda} < \wu{j} \cdot \wtNve{v'}{T_k'}$. Indeed, separating the 
   $\wu{j}$-weight of $\wtNve{v'}{T_k'}$ into the contributions of three groups of elements 
   from $L_j$ we have
      \begin{equation}\label{eq:totalwjweightTKprime}
          \wu{j} \cdot \wtNve{v'}{T_k'} = \sum_{\mu \in L_j \smallsetminus  
          (\leavesT{T_k'}\cup\leavesT{\Gamma})} \alpha_{\mu,j} \frac{\wtuv{v'}{\mu}}{|\wu{\mu}|}  +
         \sum_{\mu \in \leavesT{T_k'}\smallsetminus \leavesT{\Gamma}} \alpha_{\mu,j} 
         \frac{\wu{\mu}\cdot \wtNve{v'}{T_k'}}{|\wu{\mu}|} + \sum_{\mu \in \leavesT{T_k'} 
         \cap \leavesT{\Gamma}} \alpha_{\mu,j} \frac{\wu{\mu}\cdot \wtNve{v'}{T_k'}}{|\wu{\mu}|}.
      \end{equation}
   Notice that $\wu{\mu}\cdot \wtNve{v'}{T_k'}>\wtuv{v'}{\mu}$ whenever 
   $\mu\in \leavesT{T_k'}\smallsetminus \leavesT{\Gamma}$ and 
   $\wu{\mu}\cdot \wtNve{v'}{T_k'}\geq 0$ for $\mu \in\leavesT{T_k'}\cap \leavesT{\Gamma}$.

     Similarly, the $\wu{j}$-weight of $\wtNve{v'}{\lambda}$ can be determined by separating 
     the elements of $L_j$ into two types:
      \begin{equation}\label{eq:totalwjweightalpha}
               \wu{j} \cdot \wtNve{v'}{\lambda} = \sum_{\mu \in L_j \smallsetminus  (\leavesT{T_k'}    
               \cup\leavesT{\Gamma})} \alpha_{\mu,j} \frac{\wtuv{v'}{\mu}}{|\wu{\mu}|}  + 
               \sum_{\mu \in \leavesT{T_k'}\smallsetminus \leavesT{\Gamma}} \alpha_{\mu,j} 
               \frac{\wtuv{v'}{\mu}}{|\wu{\mu}|}\,.
      \end{equation} 
      Comparing expressions~\eqref{eq:totalwjweightTKprime} and \eqref{eq:totalwjweightalpha}, the positivity of all the coefficients $\alpha_{\mu,j}$ implies the inequality $\wu{j} \cdot \wtNve{v'}{\lambda} < \wu{j} \cdot \wtNve{v'}{T_k'}$ if the set $\leavesT{T_k'}\smallsetminus \leavesT{\Gamma}$ is non-empty. In turn, if this last set is empty, the  fact that $T$ is star-full ensures that all variables featuring in the admissible monomial  $\zexp{\wtNve{v'}{T_k'}}$ are indexed by elements of $\leavesT{T_k'}\subseteq \leavesT{\Gamma}$. Therefore, the last summand in~\eqref{eq:totalwjweightTKprime} is strictly positive. Thus, the comparison of the same two expressions gives $\wu{j} \cdot \wtNve{v'}{\lambda} < \wu{j} \cdot \wtNve{v'}{T_k'}$ in this situation as well.

By adding up the strict inequality $\wu{j} \cdot \wtNve{v'}{\lambda} < \wu{j} \cdot \wtNve{v'}{T_k'}$ and the equalities  $\wu{p} \cdot \wtNve{v'}{\lambda} = \wu{p}\cdot \wtNve{v'}{T_k'}$ from  \eqref{eq:case1weightLambdaTkp} for each $p \in \{1, \dots, q\} \setminus \{j\}$, we obtain:
    \[ (\sum_{k=1}^{q} \wu{k} ) \cdot \wtNve{v'}{\lambda} <   
        (\sum_{k=1}^{q} \wu{k} ) \cdot  \wtNve{v'}{T_k'}.\] 
Combining this inequality with  the decomposition
\eqref{eq:wwInMixedBatch}, the equality 
$\emb(v) \cdot \wtNve{v'}{\lambda}  =  \emb(v) \cdot \wtNve{v'}{T'_k}$ from 
\eqref{eq:case1vweight} yields the left-most strict inequality in~\eqref{eq:wweightCase1}.

       Next, we confirm the central inequality in~\eqref{eq:wweightCase1}. 
       By ~\eqref{eq:case1vweight} and the positivity of $\alpha_v$, it is enough to check 
       that $\wu{p}\cdot \wtNve{v'}{\lambda}\leq  \wu{p} \cdot \wtNve{v'}{v}$ for all $p\in \{1,\ldots, q\}$. Indeed, \autoref{lem:keyid} yields:
    \begin{equation}\label{eq:case1weightLambdavp}
        \wu{p} \cdot \wtNve{v'}{\lambda} = \sum_{\mu \in L_p \smallsetminus \leavesT{\Gamma}} 
        \alpha_{\mu,p} \frac{\wtuv{v'}{\mu}}{|\wu{\mu}|} \leq \sum_{\mu \in L_p} \alpha_{\mu,p} 
        \frac{\wu{\mu}\cdot \wtNve{v'}{v}}{|\wu{\mu}|} = \wu{p}\cdot \wtNve{v'}{v}\,,
    \end{equation}
  since all coefficients $\alpha_{v,\mu}$ are positive,  $\wtuv{v'}{\mu}\leq \wu{\mu}\cdot \wtNve{v'}{v}$ 
  for each $\mu\in L_p \smallsetminus \leavesT{\Gamma}$ and $0\leq \wu{\mu}\cdot \wtNve{v'}{v}$ 
  if  $\mu \in L_p\cap \leavesT{\Gamma}$.

   To finish, we must certify the right-most inequality in~\eqref{eq:wweightCase1}. 
   We use the same reasoning as in~\autoref{lm:directCalculationsIForOldProp6_7}. 
   The properties defining the series $\gvi{v'}{i}$ combined with the star-full condition of $T$, 
   the inclusion $\leavesT{T_k'}\subseteq L_j$ and the 
   expressions~\eqref{eq:case1vweight},~\eqref{eq:case1weightLambdaTkp}   
   and~(\ref{eq:totalwjweightalpha}) imply that for each monomial $\zexp{m}$ 
   appearing in a fixed $\gvi{v'}{i}$ we have the inequalities
 \[ \emb(v) \cdot \wtNve{v'}{\lambda}  
          <  \emb(v) \cdot m \quad \quad \text{ and } \quad \wu{p} \cdot \wtNve{v'}{\lambda} 
          \leq \wu{p} \cdot m  \quad \text{ for all }p=1,\ldots, q\,.
 \]
 The inequality  $\alpha_v>0$ then confirms the validity of the right-most inequality in~\eqref{eq:wweightCase1}.
  \end{subproof}

    \begin{claim}
      \label{cl:lambdaInProperBranch} We have $\starT{T_j}{v'}\cap L_j =\emptyset$. In particular, none of the leaves of $T_j$ adjacent to $v'$ can lie in $L_j$.
    \end{claim}

    \begin{subproof} 
     Since $v$ is an end-node by~\autoref{cl:vpEndNode}, any branch of $T$ emanating from $v'$ and not containing $v$ is a singleton. The statement follows by the same line of reasoning as~\autoref{cl:vpEndNode}, working with the exponents $\wtNve{v'}{v}$, $\wtNve{v'}{\lambda}$ and $\wtNve{v'}{T'_k}$, where $T_k'$ is any of the remaining branches of $T_j$ adjacent to $v'$.  Indeed, if $T_k' = \leavesT{T_k' \subseteq L_j$}, the inequalities in~\eqref{eq:wweightCase1} will remain valid, and this will contradict $\wu{}\in \ptrop \langle\sG{\Gamma}\rangle$.
    \end{subproof}

    As a consequence of~\autoref{cl:lambdaInProperBranch} we conclude that $\lambda$ is part of a branch of $T_j$ which avoids $L_j$ and has at least one node. Let $T'$ be a maximal branch of $T_j$ with this property and furthest away from $v$. We claim that $T' = T_j$.
    To prove this, we argue by contradicting the maximality of $T'$. We let $u$ be the node of $T_j$ adjacent to $T'$ and $u'$ be the node of $T'$ adjacent to $u$, as seen in the center of~\autoref{fig:keyLemma2D}. 
    Note that $u\neq v$ since $T'\neq T_j$ because $L_j$ is non-empty.

    Next, we analyze the intersections between $L_j$ and the leaves of all relevant branches of $T_j$ adjacent to $u$. We treat two cases, depending on the size of each such branch, starting with singleton branches:

        \begin{claim}\label{cl:NoLeafAdjToUInLi} 
            None of the leaves of $T_j$ adjacent to $u$ belongs to $L_j$.
    \end{claim}
        \begin{subproof} 
           We argue by contradiction and pick a leaf $\mu$ of $T_j$ adjacent to $u$ satisfying 
           $\mu \in L_j$. Then, a reasoning similar to that of~\autoref{cl:vpEndNode}  for the series $\fgvi{u}{i}$ replacing $\lambda$ by $T'$ and $T'_k$ by $\mu$ confirms that
          \begin{equation}\label{eq:cl3inequalities}
                \wu{}\cdot\wtNve{u}{T'} < \wu{}\cdot \wtNve{u}{\mu}, \quad
               \wu{}\cdot\wtNve{u}{T'} <  \wu{}\cdot\wtNve{u}{v}, \quad \text{ and } \quad
               \wu{}\cdot\wtNve{u}{T'} < \wu{} \cdot m
          \end{equation}
     whenever ${m}\in \Supp(\gvi{u}{i})$. Note that the inequality 
     $\wu{j}\cdot \wtNve{u}{T'} < \wu{j}\cdot \wtNve{u}{\mu}$ follows because $T_j$ is star-full. However,~\autoref{lm:threeMonomialsAreEnough} and the assumption $\wu{}\in \ptrop \langle\sG{\Gamma}\rangle$ refute  the validity of~\eqref{eq:cl3inequalities}, so $\mu \notin L_j$.
    \end{subproof}

        \begin{claim}\label{cl:NoBranchAtUMeetsLi} Given a non-singleton branch $T''\neq T'$ of $T_j$ adjacent to $u$ and with  $v\notin T''$, we have $\leavesT{T''}\cap L_j=\emptyset$.
    \end{claim}
        \begin{subproof}
First, we show that $\leavesT{T''}\nsubseteq L_j$. We argue by contradiction and consider the series of $\sG{\Gamma}$ determined by the node $u$. Replacing the roles of $v'$, ${\lambda}$ and ${T'_k}$ in~\eqref{eq:wweightCase1} by
 $u$, ${T'}$ and ${T''}$, respectively, the same proof technique from~\autoref{cl:vpEndNode} yields
          \[
\wu{}\cdot \wtNve{u}{T'}  < \wu{}\cdot \wtNve{u}{T''},  \qquad \wu{}\cdot \wtNve{u}{T'}  < \wu{}\cdot \wtNve{u}{v} \quad \text{ and }\quad \wu{}\cdot \wtNve{u}{T'} <\wu{} \cdot m
\]
for each $m$ in the support of any fixed $\gvi{u}{i}$.
\autoref{lm:threeMonomialsAreEnough} then shows that $\wu{}\notin \ptrop\sG{\Gamma}$, contradicting our original assumption on $\wu{}$.  From here it follows that $\leavesT{T''}\nsubseteq L_j$.

Next, we pick some leaf $\lambda\in \leavesT{T''}\smallsetminus L_j$. Following the reasoning of \textcolor{blue}{Claims}~\ref{cl:vpEndNode} and~\ref{cl:lambdaInProperBranch}, we build a maximal branch of $T''$ not meeting $L_j$. If $\leavesT{T''}\cap L_j\neq \emptyset$ this branch would be a maximal branch of $T$ avoiding $L_j$ and properly contained in $T''$. Furthermore, following the notation of the right-most picture in~\autoref{fig:keyLemma2D}, it would be contained in one of the branches $T_k''$ for $k\in \{1,\ldots, r\}$. As a result, the distance between $v$ and this branch would be strictly larger than $\distGuv{T_j}{v}{T'}$, contradicting the maximality choice of $T'$. Thus, we conclude that $\leavesT{T''}\cap L_j= \emptyset$, as we wanted to show.
        \end{subproof}

To finish, we observe that \textcolor{blue}{Claims}~\ref{cl:NoLeafAdjToUInLi} and~\ref{cl:NoBranchAtUMeetsLi} combined contradict the maximality of $T'$, since the convex hull of all branches adjacent to $u$ and not containing $v$  will be a branch of $T_j$ strictly containing $T'$ and not meeting $L_j$. From here it follows that $T'=T_j$, which cannot happen since  $L_j \neq \emptyset$.
\end{proof}

\smallskip
\subsection{Boundary components of the extended tropicalization}
\label{ssec:infinite-tropicalization}
$\:$
\smallskip 

In this subsection, we characterize the boundary strata of the extended tropicalization of the germ $(X,0)$ defined by $\sG{\Gamma}$ (see~\autoref{rem:newTerminologyPPS}). These strata are determined by  the positive local tropicalization of the intersection of $X$ with each coordinate subspace of  $\CC^n$.  This will serve two purposes. First, it will show by combinatorial methods that $(X,0)$ is a two-dimensional complete intersection with no boundary components 
(that is, without irreducible components contained in some coordinate hyperplane). 
Second, it will help us prove the  reverse inclusion to the one in ~\autoref{thm:hardInclusion2D}. The latter is the subject of~\autoref{ssec:proof-supseteq} 
(see \autoref{thm:EasyInclusionBalancing}).
\smallskip

We start by setting up notation. Throughout, we write $\boxedo{\sigma} := (\Rp)^n$ and fix  a positive-dimensional proper face  $\boxedo{\tau}$ of  $\sigma$. Since $\Nw{\Gamma}  \simeq \Z^{n}$ by  choosing the basis 
$\{\wu{\lambda}: \lambda \in \leavesT{\Gamma}\}$ from~\autoref{sec:splicetypenotation}, 
we define
\begin{equation}\label{eq:Ltau}
     \boxedo{L_{\tau}}:=\{\lambda \in \leavesT{\Gamma}: \wu{\lambda} \text{ is a ray of } \tau\}.
\end{equation}
We let $\boxedo{k}$ be the dimension of $\tau$ and consider the natural projection of vector spaces
\begin{equation}\label{eq:projtau}
  \boxedo{\projT{\tau}}\colon \R^n\to \R^n/\langle \tau\rangle \simeq \R^{n-k}.
\end{equation}
 By abuse of notation, whenever $\wu{\lambda}\notin \tau$, we identify $\wu{\lambda}$ with its image in $\R^{n-k}$ under $\projT{\tau}$.

\begin{definition} 
  Given a series $f\in \CC\{ z_{\lambda}: \lambda \in \leavesT{\Gamma}\}$, we let $\boxedo{f^{\tau}}$ be the series obtained from $f$ by setting all $z_{\lambda}$ with $\lambda \in L_{\tau}$ to be zero. We view $f^{\tau} \in \CC \{z_{\lambda}: \lambda \in  \tau^{\perp}\cap \Z^n\}$ as a series in the  $n-k$ variables in $\{z_{\lambda}: \lambda \in \leavesT{\Gamma}\smallsetminus L_{\tau}\}$. We call it the \emph{$\tau$-truncation} of $f$.
  \end{definition}

  \begin{definition}\label{def:torusOrbitTropicalization}
We let $\boxedo{X_{\tau}}$ be the intersection of the germ $(X,0)$ defined by $\sG{\Gamma}$ with the dense torus in the $(n-k)$-dimensional coordinate subspace of $\CC^n$ associated to $\tau$. This new germ is defined by the vanishing of the $\tau$-truncations of all series in  $\sG{\Gamma}$. The  
\emph{positive local tropicalization of $\sG{\Gamma}$ with respect to $\tau$} is  
defined as the positive local tropicalization of $X_{\tau}$ in $\R^{n-k}$. 
Following~\cite[Section 12]{PPS 13}, we denote it by $\boxedo{\ptrop(\sG{\Gamma},\tau)}$. 
  \end{definition}
  
  Our first result generalizes~\autoref{lm:threeMonomialsAreEnough}, when some, but not all, admissible monomials at a fixed node of $\Gamma$ have trivial $\tau$-truncations. It will simplify the computation of  $\ptrop(\sG{\Gamma},\tau)$ for each proper face $\tau$ of $\sigma$.

\begin{lemma}\label{lm:twoMonomialsAreEnoughforTruncation} Fix a positive-dimensional proper face $\tau$ of $\sigma$, a node $v$  of $\Gamma$ and some  $\wu{} \in \projT{\tau}(\sigma)$. Assume that some  $\tau$-truncated series in $\{\fvi{v}{i}^{\tau}\}_{i=1}^{\valv{v}-2}$ is not identically zero and that one of two conditions holds:
  \begin{enumerate}
  \item \label{item:ifdropmoreThanTwo} the system involves at most $\valv{v}-2$ admissible monomials and there exists an edge $e$ adjacent to $v$ with $\wtNve{v}{e} \notin \N\langle z_{\lambda}: \lambda \in L_\tau\rangle$ satisfying
$ \wu{} \cdot \wtNve{v}{e} <       \wu{}\cdot m$ 
for each $m$ in the support of some  $(\gvi{v}{i})^{\tau}$; or
\item \label{item:ifdropOne} the system involves exactly $\valv{v}-1$ monomials, and we have   two distinct edges $e$, $e'$  of $\starT{\Gamma}{v}$ with  $\wtNve{v}{e}, \wtNve{v}{e'} \notin \N\langle z_{\lambda}: \lambda \in L_\tau\rangle$ such that    $\wu{}\cdot \wtNve{v}{e} < \wu{}\cdot \wtNve{v}{e'}$  and $\wu{}\cdot \wtNve{v}{e} < \wu{}\cdot m$ 
  for each $m$ in the support of some $(\gvi{v}{i})^{\tau}$.
  \end{enumerate}
 Then, $\zexp{\wtNve{v}{e}}$ is the $\wu{}$-initial form of a series in the linear span of $\{(\fgvi{v}{i})^{\tau}\}_{i=1}^{\valv{v}-2}$ and $\wu{}\notin \ptrop (\sG{\Gamma},\tau)$.
\end{lemma}

\begin{proof}  
   Recall from~\eqref{eq:surfaceSeries} that $\fgvi{v}{i} = \fvi{v}{i}+\gvi{v}{i}$ for each node $v$ of $\Gamma$ and each $i\in \{1,\ldots, \valv{v}-2\}$. Our assumptions on $(\fvi{v}{i})^{\tau}$ and the weight vector $\wu{}$ ensure both that the $\tau$-truncation $(\fgvi{v}{i})^{\tau}$ of $\fgvi{v}{i}$ is non-zero and that the $\wu{}$-initial form of $(\fgvi{v}{i})^{\tau}$ agrees with that of $(\fvi{v}{i})^{\tau}$.

  The proof follows the same reasoning as that of~\autoref{lm:threeMonomialsAreEnough}, considering the truncations of a new basis $\{\fgvi{v}{i}'\}_{i=1}^{\valv{v}-2}$ obtained by suitable linear combinations of the original series $\{\fgvi{v}{i}\}_{i=1}^{\valv{v}-2}$ and ordering the edges adjacent to $v$ in a convenient way. Writing each element in the new basis as $\fgvi{v}{i}'=\fvi{v}{i}' + \gvi{v}{i}'$ in the spirit of~\eqref{eq:surfaceSeries}, the Hamm determinant conditions on the original splice system allow us to choose a special form for the  polynomials $\{\fvi{v}{i}'\}_i$. Indeed, their coefficient matrix has the block form $(\operatorname{Id}_{\valv{v}-2}|*|*)$, where $*$ represents a column vector in $(\CC^*)^{\valv{v}-2}$.     By construction, the new tails $\gvi{v}{i}'$ will satisfy conditions (\ref{item:ifdropmoreThanTwo}), respectively (\ref{item:ifdropOne}), if and only if the old tails $\gvi{v}{i}$ did.
  
  If condition~(\ref{item:ifdropmoreThanTwo}) holds, we pick any two edges $e',e''$ with $\wtNve{v}{e'}, \wtNve{v}{e''} \in \N\langle z_{\lambda}: \lambda \in L_\tau\rangle$ and order the edges adjacent to $v$ so that $e_{\valv{v}-2} := e$, $e_{\valv{v}-1}:=e'$ and $e_{\valv{v}}:=e''$. In this situation, the statement follows by our assumptions on the basis $\{\fgvi{v}{i}'\}_i$, combined with~\autoref{lem:keyid}. Indeed, we have 
  \[\initwf{\wu{}}{(\fgvi{v}{\valv{v}-2}')^{\tau}} = \initwf{\wu{}}{(\fvi{v}{\valv{v}-2}')^{\tau}}= \zexp{\wtNve{v}{e}}.
  \]
  Similarly, if condition~(\ref{item:ifdropOne}) holds, we pick $e''$ to be the unique edge with $\wtNve{v}{e''} \in \N\langle z_{\lambda}: \lambda \in L_\tau\rangle$. Indeed, $\wtNve{v}{e''}$ is the single monomial of the polynomials $\{\fvi{v}{i}\}_{i}$ that drops when taking their $\tau$-truncations. We set $e_{\valv{v}}:=e''$ and fix $e:=e_{\valv{v}-2}$ and $e':=e_{\valv{v}-1}$. As in the previous case, we have $\initwf{\wu{}}{(\fgvi{v}{\valv{v}-2}')^{\tau}} =  \zexp{\wtNve{v}{e}}$.
  \end{proof}

Here is the main result of this section:

\begin{theorem}\label{thm:boundaryTrop} 
   Let $\tau$ be a positive-dimensional proper face of $\sigma$ with 
   $\lambda\in L_{\tau}$. Let $v$ be the unique node of $\Gamma$ adjacent to $\lambda$. 
   Then:
  \begin{enumerate}
           \item \label{item:boundaryDim1} If $\dim \tau =1$, then $\ptrop (\sG{\Gamma},\tau)\subseteq 
                 \R_{>0}\langle \projT{\tau}(\wu{v})\rangle$.
           \item \label{item:boundaryDimatleast2} If $\dim \tau \geq 2$, then 
                 $\ptrop (\sG{\Gamma},\tau) = \emptyset$.
  \end{enumerate}
\end{theorem}

\begin{proof}   
 We fix $k=\dim \tau$. Since $\ptrop (\sG{\Gamma},\tau)$ is included in $(\R_{>0})^{n-k}$, we can determine the positive tropicalization by its intersection with  the standard simplex $\simplex{n-k-1}$. We do this by following the  proof strategy of~\autoref{thm:hardInclusion2D}.
  Since the vector $\projT{\tau}(\wu{v})$ lies in $(\R_{>0})^{n-k}$, we set
  \begin{equation}\label{eq:wvunderTau}
    \boxedo{\wu{v}'} := \frac{\projT{\tau}(\wu{v})}{|\projT{\tau}(\wu{v})|} \in \relo{(\simplex{n-k-1})}
  \end{equation}
  and perform a stellar subdivision of this standard simplex using $\wu{v}'$.
  
  If $k = 1$,   \autoref{pr:RelevantSimplicesForTau1} below ensures that every positive-dimensional simplex $\rho$ of the stellar subdivision that meets $\relo{(\simplex{n-2})}$ satisfies  $\relo{\rho} \cap \ptrop(\sG{\Gamma},\tau) = \emptyset$. The subdivision has only one zero-dimensional simplex, namely $\{\projT{\tau}(\wu{v})\}$, so  item~(\ref{item:boundaryDim1}) holds.
  In turn, when $\dim \tau = 2$, the same proposition ensures the intersection is empty even if $\rho$ is a point.  This forces $\ptrop(\sG{\Gamma},\tau) = \emptyset$, as stated in item~(\ref{item:boundaryDimatleast2}).
\end{proof}

Our next result is an adaptation of~\autoref{pr:noMixing}, needed to rule out points of the extended tropicalization $\ptrop(\sG{\Gamma},\tau)$ in a given simplex $\rho \subseteq \simplex{n-\dim\tau-1}$. It is central to proving~\autoref{thm:boundaryTrop}.

\begin{proposition}\label{pr:RelevantSimplicesForTau1}  
       Let $\tau$, $\lambda$ and $v$ be as in~\autoref{thm:boundaryTrop} and $L_{\tau}$ be as 
       in~\eqref{eq:Ltau}. Fix a proper subset $L$ of $\leavesT{\Gamma}\smallsetminus L_{\tau}$. 
       Let $\wu{v}'$ be as in~\eqref{eq:wvunderTau} and set 
       $\boxedo{\rho} := \conv{(\{\wu{v}'\}\cup\{ \wu{\mu}: \mu \in L\})}\subseteq \simplex{n-k-1}$, 
       where $k=\dim \tau$. If $|L\cup L_{\tau}|>1$, then
    \begin{equation}\label{eq:reloRhoAvoidsTrop}
            \relo{\rho}\cap \ptrop(\sG{\Gamma},\tau) = \emptyset.
    \end{equation}
\end{proposition}

 \begin{proof} We follow the proof strategy of~\autoref{pr:noMixing} for $T=\Gamma$ 
      but now incorporating the set $L_{\tau}$ into the arguments. We let 
      $T_1,\ldots, T_{\valv{v}}$ be the branches of $\Gamma$ adjacent to  $v$, with 
      $T_{1}=\{\lambda\}$. Given $j\in \{2,\ldots, \valv{v}\}$, we set
   \begin{equation}\label{eq:LjBoundary}
     \boxedo{L_{\tau,j}}:= L_{\tau} \cap \leavesT{T_j}, \qquad \boxedo{L_j} := L \cap \leavesT{T_j} \quad \text{ and }\quad \boxedo{\rho_j} :=\Rp\langle \projT{\tau}(\wu{\mu}) : \mu \in L_j\rangle\subset \R^{n-k}.
   \end{equation}
   If $L_j=\emptyset$, we declare $\rho_j=\{0\}$.
   Since $1< |L\cup L_{\tau}|<n$, upon relabeling the branches $T_2, \ldots, T_{\valv{v}}$ if necessary,  there exists a unique $q\in \{2,\ldots, \valv{v}\}$ with $L_j \cup L_{\tau,j}\neq \emptyset$ for all $j\in\{1,\ldots, q\}$, and $L_j\cup L_{\tau,j}=\emptyset$ for $j>q$.

  We argue by contradiction and pick $\wu{} \in \relo{\rho}\cap  \ptrop (\sG{\Gamma},\tau)$.   Since $\rho$ is a simplex and $T_1\subset L_{\tau}$, we have
  \begin{equation}\label{eq:wwInMixedBatchTau}
           \wu{} = \alpha_v\wu{v}' + \sum_{j=2}^{q} \wu{j} \qquad \text{ with } \alpha_v>0 \; 
           \text{ and } \; \wu{j} \in \relo{\rho_j} \text{ for all }\;j \in \{2,\ldots, q\}.
  \end{equation}
In particular, we know that $\wu{j}=0$ if and only if $\rho_j=\{0\}$.

\autoref{lm:properIsEmptyBoundary} below ensures that $L_j\cup L_{\tau,j} = \leavesT{T_j}$ for all $j\leq q$. Since $1<|L\cup L_{\tau}|<n$, we conclude that:  
    \begin{equation} \label{eq:ineqq} 
         1<q<\valv{v}.
    \end{equation} 
    In particular, we see that  $\leavesT{T_{\valv{v}}} \cap (L\cup L_{\tau}) = \emptyset$, so the $\tau$-truncation of $\zexp{\wtNve{v}{T_{\valv{v}}}}$ is non-zero.

By the Hamm conditions, the admissible monomial $\zexp{\wtNve{v}{T_{\valv{v}}}}$ must feature in some $\fvi{v}{i}$. Therefore,  the  $\tau$-truncated series $\{(\fvi{v}{i})^{\tau}\}_{i=1}^{\valv{v}-2}$ are not all identically zero. In turn, since $\lambda\in \leavesT{T_1}\cap L_{\tau}$ and $\zexp{\wtNve{v}{T_1}}= z_{\lambda}^{\du{v,\lambda}}$, we conclude that its $\tau$-truncation is zero. Thus, the system of $\tau$-truncations $\{(\fvi{v}{i})^{\tau}\}_{i=1}^{\valv{v}-2}$  involves at most $\valv{v}-1$ admissible monomials.

The proposition follows from our next claim, which contradicts our assumption that  $\wu{}\in \ptrop(\sG{\Gamma},\tau)$.

\begin{claimNoNum}\label{cl:initialFormExtendedTrop}
     The monomial  $\zexp{\wtNve{v}{T_{\valv{v}}}}$ is the $\wu{}$-initial form of a series 
     in the  linear span of $\{(\fgvi{v}{i})^{\tau}\}_{i=1}^{\valv{v}-2}$. 
\end{claimNoNum}

\noindent The statement of this claim matches the conclusion of~\autoref{lm:twoMonomialsAreEnoughforTruncation}. Recall that our earlier discussion confirmed that the $\tau$-truncations of $\fvi{v}{i}$ are not all zero, and there are two options for the number of monomials featured in these $\tau$-truncations. We prove the claim via a case-by-case analysis, each one matching one of the two possible scenarios portrayed in the lemma.

  First, we assume that  the system of $\tau$-truncations  $\{(\fvi{v}{i})^{\tau}\}_{i=1}^{\valv{v}-2}$ involves exactly $\valv{v}-1$ monomials.  In particular, $(\zexp{\wtNve{v}{T_2}})^{\tau}\neq 0$, so no variable appearing in $\zexp{\wtNve{v}{T_2}}$ can be indexed by an element of $L_{\tau,2}$. Since $L_2\cup L_{\tau,2} = \leavesT{T_2}$, we conclude that $\rho_2\neq \{0\}$  and $\wu{2}\cdot \wtNve{v}{T_2} > 0$.

  Next, we confirm that the hypotheses of~\autoref{lm:twoMonomialsAreEnoughforTruncation}~(\ref{item:ifdropOne}) hold for the weight vector $\wu{}$ and the edges $e=e_{\valv{v}}$ and $e'=e_2$ giving rise to the admissible exponents $\wtNve{v}{T_{\valv{v}}}$ and  $\wtNve{v}{T_{2}}$, respectively. Notice that our earlier discussion ensures that these two vectors lie outside $\N\langle z_{\lambda}: \lambda \in L_\tau\rangle$. To check the required inequalities between the $\wu{}$-weights of these two admissible monomials and of any monomial $\zexp{m}$ in the support of some $(\gvi{v}{i})^{\tau}$, we compare the contribution of each summand  of $\wu{}$ featured in \eqref{eq:wwInMixedBatchTau} to the total $\wu{}$-weight of each monomial. This was the same method used in the proofs of~\textcolor{blue}{Lemmas}~\ref{lm:directCalculationsIForOldProp6_7} and~\ref{lm:directCalculationsIIForOldProp6_7}.

  First, we focus on $\wu{v}'$. \autoref{lem:keyid} and the conditions (\ref{eq:gviConditions}) on $\gvi{v}{i}$ imply that 
  \begin{equation}\label{eq:wuv'ComparisonCase2Boundary}
         \wu{v}' \cdot \wtNve{v}{T_{\valv{v}}} =  \wu{v}' \cdot \wtNve{v}{T_{2}} =  
         \frac{\du{v}}{|\projT{\tau}(\wu{v})|} <  \wu{v}' \cdot m.
  \end{equation}
     In turn, our earlier discussion and the inequality $q<\valv{v}$ seen in \eqref{eq:ineqq} imply the following relations among the $\wu{j}$-weights:
  \begin{equation}\label{eq:wujComparisonCase2Boundary}
            \wu{2}\cdot \wtNve{v}{T_{\valv{v}}} = 0<\wu{2}\cdot \wtNve{v}{T_2} \quad \text { and } 
            \quad   \wu{j}\cdot \wtNve{v}{T_{\valv{v}}} = \wu{j}\cdot \wtNve{v}{T_2} = 0 \;\leq \wu{j}\cdot m \qquad 
            \text{ for } j \in \{3,\ldots, q\}\,,
  \end{equation}
  where $\zexp{m}$ is any monomial in the support of $\gvi{v}{i}$ surviving the $\tau$-truncation.  Combining the positivity of $\alpha_v$ with expressions~\eqref{eq:wuv'ComparisonCase2Boundary} and~\eqref{eq:wujComparisonCase2Boundary} confirms the validity of the conditions required in~\autoref{lm:twoMonomialsAreEnoughforTruncation}~(\ref{item:ifdropOne}).

 To finish proving the claim, we must analyze the case when the $\tau$-truncated system $\{(\fvi{v}{i})^{\tau}\}_{i=1}^{\valv{v}-2}$ has at most $\valv{v}-2$ admissible monomials. In this case, we consider the edge $e=e_{\valv{v}}$ and certify that the hypotheses imposed in~\autoref{lm:twoMonomialsAreEnoughforTruncation}~(\ref{item:ifdropmoreThanTwo}) hold. This is done by checking that the inequalities in~\eqref{eq:wuv'ComparisonCase2Boundary} and~\eqref{eq:wujComparisonCase2Boundary} involving $\wtNve{v}{T_{\valv{v}}}$ and $m$ remain valid in this new setting.
  \end{proof}

  \begin{figure}[tb]
    \includegraphics[scale=0.5]{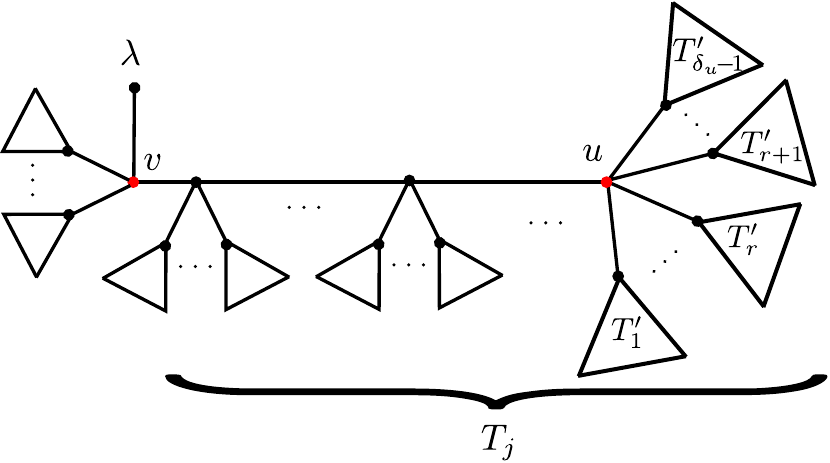}
      \caption{Given a branch $T_j$ of $\Gamma$ adjacent to $v$ and a fixed proper subset of leaves $L$ of $\leavesT{\Gamma}\smallsetminus \{\lambda\}$ meeting $\leavesT{T_j}$ properly, we find a node $u$ in $T_j$ with $\distGuv{\Gamma}{u}{v}$ maximal so that the branches $T'_1,\ldots, T'_r$ have a given property, but $T_{r+1}', \ldots, T_{\valv{u}-1}'$ do not. This construction is the main ingredient in the proof of~\autoref{lm:properIsEmptyBoundary}.
\label{fig:keyBoundaryLemma}}
    \end{figure}

  Our next result is analogous to~\autoref{lm:properIsEmpty}. We prove it by using the same techniques.

  \begin{lemma}\label{lm:properIsEmptyBoundary} 
        Fix $L_{\tau,j}$, $L_j$ and $\rho_j$ as in~\eqref{eq:LjBoundary}. 
        If $\relo{\rho} \cap \ptrop (\sG{\Gamma},\tau) \neq \emptyset$, then the set 
        $L_j\cup L_{\tau,j}$ is either empty or it equals $\leavesT{T_j}$.
\end{lemma}

\begin{proof}
  We use the notation established in the proof of~\autoref{pr:RelevantSimplicesForTau1} and pick $\wu{}\in \relo{\rho} \cap \ptrop (\sG{\Gamma},\tau)$. We decompose $\wu{}$ as in~\eqref{eq:wwInMixedBatchTau}.  We  argue by contradiction, assuming that $0< |L_j\cup L_{\tau,j}|< |\leavesT{T_j}|$.

  First, note that we can find a branch $T'$ of $\Gamma$ contained in $T_j$ with $\leavesT{T'}\cap (L_j\cup L_{\tau,j})=\emptyset$. For example, any leaf $\mu$ of $\leavesT{T_j}$ not in $L\cup L_{\tau}$ will produce such a branch.
  Since $T_j$ is finite, we can choose the branch $T'$ to be maximal with respect to the condition $\leavesT{T'}\cap (L_j\cup L_{\tau,j})=\emptyset$ and furthest away from $v$ in the geodesic metric on $\Gamma$. Let $u$ be the unique node of $\Gamma$ adjacent to $T'$. Since $L_j\cup L_{\tau,j} \neq \emptyset$, we know that $u\neq v$.
  
Assume that there are $r$ branches adjacent to $u$ with this maximality property, and denote them by $T_1'=T',\ldots, T_r'$ as in~\autoref{fig:keyBoundaryLemma}.
  We let $T'_{r+1},\ldots, T'_{\valv{u}-1}$ be the remaining branches of $\Gamma$ adjacent to $u$ and not containing $\lambda$. The maximality of both $T'$ and $\distGuv{\Gamma}{u}{v}$ combined with the  condition $u\neq v$  implies that $r<\valv{u}-1$ and  $\leavesT{T_i'}\subseteq L_j \cup L_{\tau,j}$  for each $i\in \{r+1,\ldots, \valv{u}-1\}$. Indeed, if the latter were not the case, we could find a branch $T''$ of $\Gamma$ inside $T'_i$ with the same properties as $T'$, but further away from $v$.

  The construction of $T_1'$ ensures that $(\zexp{\wtNve{u}{T_1'}})^{\tau}\neq 0$. Therefore,  the collection  $\{(\fvi{u}{i})^{\tau}\}_{i=1}^{\valv{u}-2}$ is not identically zero.
Our original assumption that  $\wu{}\in \ptrop (\sG{\Gamma},\tau)$ will be contradicted by the following assertion:
  
  \begin{claimNoNum}\label{cl:T1'initialMix} 
        The admissible monomial $\zexp{\wtNve{u}{T_1'}}$ is the $\wu{}$-initial form of a series 
        in the  linear span of $\{(\fgvi{u}{i})^{\tau}\}_{i=1}^{\valv{u}-2}$.
  \end{claimNoNum}
  
  We prove our claim by  analyzing  three cases (in decreasing order of difficulty), depending on whether the number of admissible monomials  in the $\tau$-truncation of the collection $\{\fvi{u}{i}\}_{i=1}^{\valv{u}-2}$ is exactly $\valv{u}$, at most $\valv{u}-2$, or exactly $\valv{u}-1$. The first case will make use of~\autoref{lm:threeMonomialsAreEnough}, whereas we will invoke~\autoref{lm:twoMonomialsAreEnoughforTruncation} to establish the other two.

    \smallskip
        \noindent\textbf{Case 1:}  We assume that the number of admissible monomials at $u$ remaining after $\tau$-truncation is $\valv{u}$. By the Hamm conditions, this forces $\fvi{u}{i} = (\fvi{u}{i})^{\tau}$ for each $i \in \{ 1,\ldots, \valv{u}-2\}$. After replacing each $\gvi{u}{i}$ with its $\tau$-truncation, we are in the setting of~\autoref{lm:threeMonomialsAreEnough}, 
  where we view ${(\Rp)}^{n-k}\subset \Rp^{n}$ by adding zeros as complementary coordinates. 
  In this situation, the condition $\leavesT{T'_1} \cap L=\emptyset$ and the defining properties of $\gvi{u}{i}$ ensure that
  \begin{equation}\label{eq:case1NoDropping}
           \wu{} \cdot \wtNve{u}{T'_1} < \wu{} \cdot \wtNve{u}{v}, \qquad \wu{} \cdot \wtNve{u}{T'_1} 
           < \wu{} \cdot \wtNve{u}{T'_{\valv{u}-1}} \quad \text{ and } \quad \wu{} \cdot \wtNve{u}{T'_1} 
           < \wu{}\cdot m
  \end{equation}
    for each $\zexp{m}$ appearing in $(\gvi{u}{i})^{\tau}$. The aforementioned lemma then implies the validity of our claim. 

    To prove the inequalities in~\eqref{eq:case1NoDropping}, we proceed by comparing the contributions of each summand of $\wu{}$ in the decomposition \eqref{eq:wwInMixedBatchTau} to the weight of each monomial.~\autoref{lem:keyid} and~\eqref{eq:gviConditions} ensure that 
  \begin{equation}\label{eq:noDropswuv'weights}
        \wu{v}'\cdot \wtNve{u}{T'_1} = \wu{v}'\cdot \wtNve{u}{T'_{\valv{u}-1}} = 
        \frac{\wtuv{u}{v}}{|\projT{\tau}(\wu{v})|} < \min\{\wu{v}'\cdot \wtNve{u}{v}, 
        \,\wu{v}'\cdot m\}\,. 
  \end{equation}
        In turn, for each $p\in \{2,\ldots,\valv{v}\}$ with $p\neq j$, the fact that $L$ contains 
        no node of $\Gamma$ gives 
  \begin{equation}\label{eq:noDropswupweights}
         \wu{p}\cdot \wtNve{u}{T'_1}= \wu{p}\cdot \wtNve{u}{T'_{\valv{u}-1}}
         =0\leq \min\{\wu{p}\cdot \wtNve{u}{v}, \,\wu{p}\cdot m\}\,.
  \end{equation}
        Finally, since $L_j\cap \leavesT{T_1'}=\emptyset$, $\leavesT{T_{\valv{v}-1}'} 
        \subseteq L_j\cap L_{\tau,j}$ and the admissible monomial $\zexp{\wtNve{u}{T'_{\valv{v}-1}}}$ 
        has no variable indexed by an element in $L_{\tau,j}$, we see that 
    \begin{equation}\label{eq:noDropswujweights}
             \wu{j} \cdot \wtNve{u}{T'_1}= 0 \leq \min \{\wu{j}\cdot \wtNve{u}{v}, \,\wu{j}\cdot m\} 
             \quad \text{ whereas } \quad 0< \wu{j} \cdot \wtNve{u}{T'_{\valv{u}-1}}\,.
    \end{equation}
    Expressions~\eqref{eq:noDropswuv'weights}, \eqref{eq:noDropswupweights} and~\eqref{eq:noDropswujweights}, combined with the positivity of $\alpha_v$ confirm~\eqref{eq:case1NoDropping}.

  \smallskip
\noindent\textbf{Case 2:} Assume that the $\tau$-truncated series $\{(\fvi{u}{i})^{\tau}\}_{i=1}^{\valv{u}-2}$ involve at most $\valv{u}-2$ admissible monomials. In this situation, the claim follows by~\autoref{lm:twoMonomialsAreEnoughforTruncation}~(\ref{item:ifdropmoreThanTwo}) since the inequalities involving $\wtNve{u}{T_1'}$ and $m$ in~\eqref{eq:case1NoDropping} remain valid in this scenario.

\noindent\textbf{Case 3:} Suppose that  the $\tau$-truncations of the collection $\{\fvi{u}{i}\}_{i=1}^{\valv{u}-2}$  involve exactly $\valv{u}-1$ admissible monomials. If so, up to relabeling of $T'_{r+1}, \ldots, T'_{\valv{u}-1}$ we may assume that  exactly one of the admissible monomials with exponent  $\wtNve{u}{v}$ or $\wtNve{u}{T'_{\valv{u}-1}}$ vanishes under $\tau$-truncation. The inequalities in~\eqref{eq:case1NoDropping} involving $\wtNve{u}{T'_1}$, $m$ and the aforementioned surviving admissible exponent remain valid. Thus,  the hypotheses required by~\autoref{lm:twoMonomialsAreEnoughforTruncation}~(\ref{item:ifdropOne}) hold, hence confirming the claim.
\end{proof}

 \autoref{thm:boundaryTrop} has the following two important consequences:

    \begin{corollary}\label{cor:codim2boundaryStataGerm}
    The intersection  of the germ $(X,0)$ defined by the system $\sG{\Gamma}$ with any coordinate subspace $H$ of $\CC^n$ of codimension at least two is just the origin.
  \end{corollary}

    \begin{proof} Assume the contrary and let $H$ be a coordinate subspace of $\CC^n$ of maximal codimension such that $X\cap H\neq \{0\}$.  Consider the face $\tau$ of $(\Rp)^n$ associated to $H$ and let $X_{\tau}$ be the intersection of $X$ with the dense torus of $H$. The maximality condition on  $H$ implies that $X_\tau\neq \emptyset$. Since    $\dim X_{\tau} = \dim \ptrop (\sG{\Gamma},\tau)$  by~\autoref{prop:puredimltrop}~(\ref{condsamedim}), we conclude that  $\ptrop (\sG{\Gamma},\tau) \neq\emptyset$.

Since the dimension of $\tau$ agrees with the codimension of $H$ in $\CC^n$, our hypothesis implies that $\dim \tau \geq 2$. Therefore, we have $\ptrop (\sG{\Gamma},\tau)=\emptyset$ by~\autoref{thm:boundaryTrop}~(\ref{item:boundaryDimatleast2}).  This contradicts our earlier observation.
    \end{proof}
    
  \begin{corollary}\label{cor:codim1boundaryStataGerm}
    The germ $(X,0)$  defined by the system $\sG{\Gamma}$ intersects each coordinate hyperplane  along a germ of a curve. All its irreducible components meet the dense torus of the corresponding hyperplane.
  \end{corollary}

  \begin{proof} We fix a coordinate hyperplane $\{z_{\lambda} = 0\}$ and let $\tau$ be the cone 
  generated by $\wu{\lambda}$ in $\R^n$. Counting the number of equations defining 
  $Y:=X\cap \{z_{\lambda}=0\}$ in $\CC^n$, we see that $\dim Y \geq 1$ by Krull's principal ideal 
  theorem. By construction, $X_{\tau} = Y \cap (\CC^*)^{n-1}$.  Since any face $\tau'$  of 
  $(\Rp)^n$ properly containing $\tau$ satisfies  $X_{\tau'} =\emptyset$ 
  by~\autoref{cor:codim2boundaryStataGerm}, it follows that no component of $Y$ lies in the toric 
  boundary of $\CC^{n-1}$. Furthermore,~\autoref{prop:puredimltrop}~(\ref{condsamedim}) 
  implies that  $\dim X_{\tau} = \dim \ptrop (\sG{\Gamma}, \tau)$. 
  Since ~\autoref{thm:boundaryTrop}~(\ref{item:boundaryDim1}) ensures that $ \dim \ptrop (\sG{\Gamma}, \tau)\leq 1$, we conclude that $\dim X_{\tau} =1$. 
  \end{proof}

  In turn, the last corollary has two consequences. First, it confirms that the  germ defined by $\sG{\Gamma}$ is a complete intersection, and second, it shows that equality must hold in~\autoref{thm:boundaryTrop}~(\ref{item:boundaryDim1}).

  \begin{corollary}\label{cor:expDimensionsG}
    The germ $(X,0)$ defined by the system $\sG{\Gamma}$ is a two-dimensional complete intersection. Each of its irreducible components meets the dense torus $(\CC^*)^n$ non-trivially in dimension two.
  \end{corollary}

Note that this result allows a priori for the germ $(X,0)$ to have several irreducible components. We will see in~\autoref{cor:isolatedSing} that $(X,0)$ is in fact irreducible.
    \begin{corollary}\label{cor:codim1TropBoundary}
      Let $\lambda$ be a leaf of $\Gamma$ and let $v$ be the unique node of $\Gamma$ adjacent to it. If $\tau = \Rp \langle\wu{\lambda}\rangle$, then $\ptrop (\sG{\Gamma},\tau) = \R_{>0}  \langle\projT{\tau}(\wu{v})\rangle$.      
        \end{corollary}
    \begin{proof} The result follows by combining~\autoref{thm:boundaryTrop}~(\ref{item:boundaryDim1}), \autoref{cor:codim1boundaryStataGerm}, and the equality between the dimensions of $X_{\tau}$ and $\ptrop (\sG{\Gamma},\tau)$ stated in~\autoref{prop:puredimltrop}~(\ref{condsamedim}).
      \end{proof}

   Our last result in this subsection is a special case of~\autoref{prop:compartrop}. It will be used in the next subsection.
    
  \begin{corollary}\label{cor:closureOfPositiveTrop}
      The local tropicalization of the germ defined by the splice type system $\sG{\Gamma}$ is the Euclidean closure of 
      $\ptrop \langle\sG{\Gamma}\rangle$ in $\R^n$.
   \end{corollary}

\smallskip

    \subsection{The support of the splice fan is contained in the local tropicalization.}
\label{ssec:proof-supseteq}
$\:$ 
\smallskip

In this subsection, we prove the remaining inclusion in~\autoref{thm:tropsG}.  Our arguments are purely combinatorial, and rely on the  balancing condition for pure-dimensional local tropicalizations (see~\autoref{thm:balancingCondition} and \autoref{rem:balancedLoc}). \autoref{cor:expDimensionsG} confirms that such condition holds for  the positive tropicalization of the germ defined by $\sG{\Gamma}$. Furthermore, the proofs in this section imply that no proper two-dimensional subset of $\ptrop \langle\sG{\Gamma}\rangle$ is balanced.
\smallskip

We start by stating the main result in this section. Its proof will be broken into several lemmas and propositions for clarity of exposition.

\begin{theorem}\label{thm:EasyInclusionBalancing}
  For every splice diagram $\Gamma$, we have $\emb(\Gamma)\subseteq \Trop \langle\sG{\Gamma}\rangle$ 
  in $(\Rp)^n$. 
\end{theorem}

\begin{proof}  By~\autoref{lm:nodeInTrop} below we know that $\simplex{n-1}\cap \Trop \langle\sG{\Gamma}\rangle$ is a 1-dimensional polyhedral complex and $\emb(v)\in \ptrop \langle\sG{\Gamma}\rangle$ for some node $v$ of $\Gamma$. We claim that, in fact, $\emb(u)\in \ptrop \langle\sG{\Gamma}\rangle$ for each node  $u$.

  We prove this claim  by induction on the distance between $u$ and $v$ (recall that $\Gamma$ is  connected). If $u=v$ there is nothing to show. For the inductive step, pick a node $u$ with $\distGuv{\Gamma}{u}{v} = k \geq 1$ and assume that the claim holds for each node $u'$ of $\Gamma$ with $\distGuv{\Gamma}{u'}{v} = k -1$. Let $u'$ be the unique node of $[u,v]$ adjacent to $u$. Then, $\emb(u')\in \ptrop \langle\sG{\Gamma}\rangle$ by our inductive hypothesis. \autoref{pr:StarIsIn} below and~\autoref{cor:closureOfPositiveTrop} yield:
  \[    \emb(u)\in \emb(\starT{\Gamma}{u'})\cap (\R_{>0})^n \subseteq   \Trop \langle\sG{\Gamma}\rangle.
  \]

  The desired inclusion $\emb(\Gamma) \subseteq \overline{\ptrop \langle\sG{\Gamma}\rangle}$ follows by combining~\autoref{pr:StarIsIn} with the identity
  \[\emb(\Gamma) = \bigcup_{v \text{ node of } \Gamma} \emb(\starT{\Gamma}{v}).\qedhere\]
\end{proof}

Our first lemma ensures that the image of some node of $\Gamma$ lies in the positive tropicalization of $\sG{\Gamma}$: 

\begin{lemma}\label{lm:nodeInTrop}     The intersection $ \simplex{n-1}  \cap \Trop \langle\sG{\Gamma}\rangle$ is a 1-dimensional polyhedral complex. Furthermore, there exists a node $v$ of $\Gamma$ with $\emb(v) \in \simplex{n-1} \cap \ptrop \langle\sG{\Gamma}\rangle$.
  \end{lemma}

\begin{proof}
 By~\autoref{cor:expDimensionsG} and~\autoref{prop:puredimltrop}~(\ref{condsamedim}) we  know that $\Trop\sG{\Gamma}$ is a fan of pure dimension two, and so  $\simplex{n-1} \cap \Trop \langle\sG{\Gamma}\rangle$ is a pure 1-dimensional  polyhedral complex.  To conclude, we must find a node $v$ with $\emb(v) \in \ptrop \langle\sG{\Gamma}\rangle$.
    
Since $\ptrop \langle\sG{\Gamma}\rangle \subseteq \emb(\Gamma)$ by~\autoref{cor:hardInclusion2D}, we have two possibilities for any  $\wu{} \in \simplex{n-1} \cap \ptrop \langle\sG{\Gamma}\rangle$: either $\wu{}= \emb(v)$ for some node $v$  of $\Gamma$ or $\wu{}\in \emb(\relo{[u,u']})$ for two adjacent vertices $u,u'$ of $\Gamma$. In the second situation,~\autoref{lm:EdgesInTrop} below ensures that $\emb([u,u'])\subset \simplex{n-1} \cap \Trop \langle\sG{\Gamma}\rangle  $. Since one of  $u$ or $u'$ must be a node  of $\Gamma$ and each node maps to $\relo{(\simplex{n-1})}$ under $\emb$, the claim follows.
  \end{proof}

  Our next lemma is central to the proof of both~\autoref{thm:EasyInclusionBalancing} and~\autoref{lm:nodeInTrop}. It describes the possible intersections between the local tropicalization of the ideal $\langle\sG{\Gamma}\rangle$ and the  edges of $\Gamma$ embedded in $\simplex{n-1}$ via the map $\emb$. The balancing condition for positive local tropicalizations plays a prominent role. 
  
\begin{lemma}\label{lm:EdgesInTrop} 
    Let $u, u'$ be two adjacent vertices of $\Gamma$. If $\emb(\relo{[u,u']})$ intersects 
    $\ptrop \langle\sG{\Gamma}\rangle$ non-trivially, then   $\emb([u,u'])\subseteq \Trop \langle\sG{\Gamma}\rangle$.
\end{lemma}

\begin{proof} It suffices to show that $\emb(\relo{[u,u']}) \subseteq \Trop \langle\sG{\Gamma}\rangle$. We argue by contradiction and fix a point $\wu{} \in \emb([u,u']) \smallsetminus \ptrop \langle\sG{\Gamma}\rangle$.
  Let  $\varphi\colon [0,1] \to [u,u']$ be the affine map such that $\varphi(0)=u$ and $\varphi(1)=u'$. There exists a unique $t_0 \in [0,1]$ such that  $\wu{} =\emb(\varphi(t_0))$.
  Since $\Trop \langle\sG{\Gamma}\rangle$ is closed in $(\Rp)^n$, we can find a pair $a,b\in [0,1]$ satisfying   $a < t_0 < b$  and such that the open segment $(\emb(\varphi(a)), \emb(\varphi(b)))$ avoids $\ptrop \langle\sG{\Gamma}\rangle$. Furthermore, we pick $(a,b)$ to be the maximal open interval in $[0,1]$ containing $t_0$ with this property. A contradiction  will arise naturally if we prove that $a=0$ and $b=1$.
  
  By symmetry, it suffices to show that $a=0$. We argue by contradiction, and assume $a>0$, so  $\emb(\varphi(a))\in \relo{(\simplex{n-1})}$ by construction. The maximality of $(a,b)$ combined with~\autoref{cor:closureOfPositiveTrop} ensures that $\emb(\varphi(a))\in \ptrop \langle\sG{\Gamma}\rangle$. Recall from~\autoref{cor:hardInclusion2D} that  $ \simplex{n-1} \cap \ptrop \langle\sG{\Gamma}\rangle \subseteq \emb(\Gamma)$.
  Since $\emb$ is an injection by~\autoref{thm:injectivityrho}, any neighborhood of $\emb(\varphi(a))$ in $\ptrop (\langle\sG{\Gamma}\rangle)$ lies in the relative interior of the two-dimensional cone 
 $\Rp \langle\emb([u, \varphi(a)])\rangle$. In turn, since $\ptrop(\langle\sG{\Gamma}\rangle)$ 
  is two-dimensional, and $\relo{\emb([\varphi(a),\varphi(b)])}$ avoids $\ptrop \langle\sG{\Gamma}\rangle$, it follows that $\ptrop \langle\sG{\Gamma}\rangle$ is not balanced at $\emb(\varphi(a))$ 
  when $a>0$. Indeed, the balancing condition (see~\autoref{thm:balancingCondition}) at $\emb(\varphi(a))$ involves a single non-zero vector, namely, the image of $\emb(u)-\emb(\varphi(a))$ 
  in the lattice $(\Z^n\cap \langle \emb(u), \emb(\varphi(a)) \rangle) /
 (\Z^n \cap   \langle \emb(\varphi(a)) \rangle)$,
  and a non-zero scalar. Therefore, we must have $a=0$.
\end{proof}

Our next result plays a prominent role in the induction arguments used to prove~\autoref{thm:EasyInclusionBalancing}. Once again, as in the proof of \autoref{lm:EdgesInTrop}, the balancing condition for local tropicalizations is crucial in our reasoning.

\begin{proposition}\label{pr:StarIsIn}
Let $v$ be a node of $\Gamma$ with $\emb(v)\in \ptrop \langle\sG{\Gamma}\rangle$. Then, $\emb(\starT{\Gamma}{v}) \subseteq \Trop \langle\sG{\Gamma}\rangle$.  
\end{proposition}

\begin{proof} After refinement if necessary, we may assume that $\boxedo{\tau}:=\Rp \emb(v)$ is a ray of $\ptrop \langle\sG{\Gamma}\rangle$. In order to exploit the balancing condition, we use analogous notation to that of~\autoref{def:balancingCondition}, simplified slightly by the fact that $\ptrop \langle\sG{\Gamma}\rangle$ is 2-dimensional. 
  
  Given a vertex $u$ of $\Gamma$ adjacent to $v$, we consider the 1-dimensional saturated lattice
  \[
        \boxedo{\Lambda}  :=\frac{\Z^n \cap \R\langle \wu{u}, \wu{v}\rangle}{\Z^n\cap \R\langle \wu{v} \rangle}.\]
    Fix a vector ${\wu{u}}_{|\tau} \in \Z^n\cap \Rp\emb([u,v]) $ whose natural projection to 
    $\Lambda$  
  generates this lattice. In particular, we can write ${\wu{u}}_{|\tau}$ uniquely as
  \begin{equation}\label{eq:balancingVectors}
    {\wu{u}}_{|\tau} = \alpha_{u} \emb(u) + \beta_{u} \emb(v)
  \end{equation}
  for some $\boxedo{\alpha_{u}}, \boxedo{\beta_{u}}  \in \Q$ and with $\alpha_{u}>0$.

We let $u_1,\ldots, u_{\valv{v}}$ be the vertices of $\Gamma$ adjacent to $v$.   The balancing condition for $\ptrop \langle\sG{\Gamma}\rangle$ at $\tau$ combined with~\autoref{thm:hardInclusion2D} ensures the existence of non-negative integers $\{k_{1}, \ldots, k_{\valv{v}}\}$ (i.e., the tropical multiplicities 
of \autoref{def:tropMult}) satisfying:
  \begin{equation}\label{eq:balancingv}
\sum_{i=1}^{\valv{v}} k_i\, {\wu{u_i}}_{|\tau}   \in  \Z^n\cap \R\langle \wu{v} \rangle.
  \end{equation}
  Moreover, by~\autoref{def:tropMult}, $\relo{(\emb([u_j,v]))}$ intersects $\ptrop \langle\sG{\Gamma}\rangle$ non-trivially in a neighborhood of $\emb(v)$ if and only if $k_j\neq 0$.

  Since $\ptrop \langle\sG{\Gamma}\rangle$ is pure of dimension two and $\emb(v)\in \ptrop \langle\sG{\Gamma}\rangle$, we know that $k_{j_0}>0$ for some $j_0 \in \{1,\ldots, \valv{v}\}$.
  We claim that, furthermore, all $k_1,\ldots,k_{\valv{v}}$ are positive
  integers.  The inclusion $\emb(\starT{\Gamma}{v})\subset {\Trop \langle\sG{\Gamma}\rangle}$
  follows by combining this statement with~\autoref{lm:EdgesInTrop}.
   
To prove our claim, we consider the following linear equation in $k_1,\ldots, k_{\valv{v}}$ that is equivalent to~\eqref{eq:balancingv}:
      \begin{equation}\label{eq:equivBalancingv}
        \sum_{j=1}^{\valv{v}} (k_j\,\alpha_{u_j}) \emb(u_j) = \beta\, \emb(v) \qquad \text{ with } \beta \in \Q.
      \end{equation}
      First, we argue that this system admits a unique solution  $(k_1,\ldots, k_{\valv{v}})\in \Q^{\valv{v}}$ for all $\beta$.  Uniqueness follows directly  since $\alpha_{u_j}>0$ for all $j$ and the set $\{\emb(u_1), \ldots, \emb(u_{\valv{v}})\}$ is linearly independent by~\autoref{pr:convexitySimplices}~(\ref{linIndep}).

      Second, we claim that for $\beta\neq 0$, any solution to~\eqref{eq:equivBalancingv} has $k_j\neq 0$ for all $j$. By homogeneity, we may assume  $\beta = 1$. Then,~\autoref{pr:convexityBar}~(\ref{inRelInt}) and  the linear independence of $\{\emb(u_j)\}_{j=1}^{\valv{v}}$   force the coefficients $k_j\alpha_{u_j}$ in~\eqref{eq:equivBalancingv} to be the ones used to write $\emb(v)$ as an element of $\relo{(\simplex{\starT{\Gamma}{v}})}$. Since $\alpha_{u_j}>0$ for all $j$, we have $k_j> 0$ for all $j$, as we wanted.
        
      To finish, we argue that~\eqref{eq:balancingv} has a solution with $k_{j_0}>0$ by the balancing condition. This forces $\beta\neq 0$ in~\eqref{eq:equivBalancingv}, and so   $k_j\neq 0$ for all $j$ by the previous discussion. This concludes our proof.
\end{proof}

 \section{Splice type singularities are Newton non-degenerate}\label{sec:NewtonND}

In this section we discuss the Newton non-degeneracy  of splice type systems 
(see \autoref{thm:NewtonNonDeg}), 
following the original framework introduced by Khovanskii (see~\autoref{def:Nnondegci}). 
  This property only involves the initial forms of the generators of the system $\sG{\Gamma}$, as opposed to the initial ideals of the ideal  generated by the elements of the system $\sG{\Gamma}$. 
In turn,~\autoref{thm:newL3.3} ensures the $\wu{}$-initial forms of $(\fgvi{v}{i})_{v,i}$ generate the $\wu{}$-initial ideal of $\langle\sG{\Gamma}\rangle$, for each rational weight vector $\wu{}$ in the positive local  tropicalization of $\sG{\Gamma}$, thus  giving an alternative proof of~\autoref{thm:EasyInclusionBalancing}. Furthermore, Newton non-degeneracy implies that condition~(\ref{condinit}) of~\autoref{prop:puredimltrop} holds for every cell in $\emb(\Gamma)$, thus showing that the splice fan of $\Gamma$ is a standard tropicalizing fan for the germ defined by the system $\sG{\Gamma}$, 
in the sense of~\autoref{def:tropicalizingFan}. This is the content of~\autoref{cor:tropsG}. In turn, \autoref{cor:onetoricplanecurve} gives an  alternative proof of a recent theorem of de Felipe, Gonz\'alez P\'erez and Mourtada~\cite{FGM 21}, which resolves any germ of a reduced plane curve  by one toric morphism after a reembedding in a higher dimensional complex affine space. The section concludes with an open question regarding embedding dimension of splice type surface singularities preserving the Newton non-degeneracy property (see~\autoref{ques:minNNDemb}).
\medskip

Our first result is a direct consequence of \autoref{lem:keyid} and \eqref{eq:gviConditions}. It allows us to determine the $\wu{u}$-initial form of each series $\fgvi{v}{i}$ in the system $\sG{\Gamma}$  for each  node $u$ of $\Gamma$. This computation is central to all arguments in this section:

\begin{proposition}\label{pr:InitWufvi} 
   For each pair of nodes $u,v$ in $\Gamma$, and each $i \in \{ 1, \dots, \valv{v}-2\}$, we have:
  \[
  \initwf{\wu{u}}{\fgvi{v}{i}} =   \initwf{\wu{u}}{\fvi{v}{i}} =
  \begin{cases}
    \fvi{u}{i} & \text{ if } u=v,\\
    \fvi{v}{i} - \cvi{v}{[u,v]} \,\zexp{\wtNve{v}{e}} & \text{ otherwise},
  \end{cases}
  \]
  where $e$ is the unique edge adjacent to $v$ with $e\subseteq [v,u]$. 
 \end{proposition}
 
\noindent  The second formula, for the case where $u \neq v$, means that we remove from 
   $\fvi{v}{i}$ the term associated to the edge starting from $v$  in the direction of $u$.

\begin{remark}\label{rm:EndCurvesInitialForms} 
   Recall from~\autoref{def:endcurverSR} that rooted splice diagrams $\Gamma_{\roottree}$ with $n+1$ leaves (including the root $r$) give rise to end-curves in $\CC^n$.  Consider a fixed strict splice type system $\sG{\Gamma}$ associated to $\Gamma$ giving rise to this curve and let $u$ be the unique node of $\Gamma_{\roottree}$ adjacent to $r$. In  view of \autoref{pr:InitWufvi}, the curve $\Ccurve_{\roottree}$ is defined by the vanishing set of $z_r$ and the $\wu{u}$-initial forms of the $n-1$ polynomials in the system $\sG{\Gamma}$, viewed in $\{0\}\times \CC^n$.
\end{remark}

     Next, we set up notation that we will use throughout this section. For each weight vector $\wu{} \in (\R_{>0})^n$ we define:
  \begin{equation}\label{eq:JinitFormGens}
    \boxedo{J_{\wu{}}}:=\langle \initwf{\wu{}}{\fgvi{v}{i}}\colon v \text{ node of } \Gamma, i=1,\ldots, \valv{v}-2\rangle \subset \CC[z_{\lambda_1}, \ldots, z_{\lambda_n}],
  \end{equation}
  and let $\boxedo{Z_{\wu{}}}$ be the subscheme of  $\CC^n$ defined by $J_{\wu{}}$.

\smallskip

Next, we state the main result of this section. The rest of the section is devoted to its proof.
\begin{theorem}\label{thm:NewtonNonDeg} 
      The splice type system $\sG{\Gamma}$ is a Newton non-degenerate complete intersection system.
\end{theorem}

\begin{proof}
  We certify the two conditions from \autoref{def:Nnondegci} starting with the regularity of the sequence defined by  $\sG{\Gamma}$. This property follows by~\autoref{lm:wu} and~\autoref{thm:newL3.3}. Indeed, since the weight vector $\wu{u}$ associated to any node $u$ of $\Gamma$ lies in $(\Z_{>0})^n$, and $Z_{\wu{u}}$ is a complete intersection of dimension two with defining ideal $J_w$ by the lemma, we conclude that the $n-2$ generators of $J_w$ form a regular sequence in $\cO$. In turn, the theorem confirms that the series $\{\fgvi{v}{i}\}_{v,i}$ form a regular sequence in $\cO$.

Recall from~\autoref{def:posloctrop} that $Z_{\wu{}}\cap (\CC^*)^n$ is the empty-set whenever $\wu{}\notin \ptrop \langle\sG{\Gamma}\rangle$. Thus, it suffices to certify the linear independence condition of the differentials of $\initwf{\wu{}}{\fgvi{v}{i}}$ when $\wu{} \in (\R_{>0})^n\cap \ptrop \langle\sG{\Gamma}\rangle$. In turn, by \autoref{thm:tropsG} it is enough to restrict our analysis to vectors $\wu{}\in (\R_{>0})^n \cap \emb(\Gamma)$.

Since the generators of $J_{\wu{}}$ form a regular sequence in $\cO$ by~\autoref{cor:initialForms},~\autoref{pr:regSmoothIsNND} simplifies our task:  we need only verify that $Z_{\wu{}}\cap (\CC^*)^n$ is smooth for each such $\wu{}$. We analyze three cases, depending on the nature of the unique cell of $\emb(\Gamma)$ containing $\wu{}$ in its relative interior. Each case matches the settings of \textcolor{blue}{Lemmas}~\ref{lm:wu},~\ref{lm:wulambda} and~\ref{lm:wuv} below.

First, we pick a node $u$ of $\Gamma$ and consider the weight vector $\wu{}={\wu{u}}/{|\wu{u}|}$. By~\autoref{lm:wu}, the components of the scheme $Z_{\wu{u}} \subseteq \CC^n$ are Pham-Brieskorn-Hamm singularities, and they can only meet at coordinate subspaces of $\CC^n$. Furthermore, the Hamm conditions ensure that these germs are smooth away from the origin. This fact confirms the smoothness of $Z_{\wu{u}}\cap (\CC^*)^n$.

Next, we consider a weight vector $\wu{}\in \relo{\emb([\lambda, u])}$, were $\lambda$ is a leaf of $\Gamma$ and $u$ is a node of $\Gamma$ adjacent to $\lambda$. Then, by \autoref{lm:wulambda}, $Z_{\wu{}}$ is a cylinder over an end-curve $\Ccurve_{\lambda}$ associated to the rooted diagram $\Gamma_{\lambda}$. By~\autoref{thm:end-curvesNW}, such curves are smooth away from the origin. This  confirms that $Z_{\wu{}}\cap (\CC^*)^n$ is smooth in this case as well.

Finally, we fix a weight vector $\wu{}\in \relo{\emb([u,u'])}$ where $u,u'$ are adjacent nodes of $\Gamma$. Then,~\autoref{lm:wuv} confirms that the scheme $Z_{\wu{}}$ is the product of two end-curves $\Ccurve_{\Gamma'_u}$ and $\Ccurve_{\Gamma''_v}$ associated to two splice subdiagrams of $\Gamma$ rooted at $u$ and $v$, respectively. More precisely, we have $\Gamma':=[u, \nodesTevRoot{u,[u,v]}]$ and $\Gamma'':=[v,\nodesTevRoot{v,[u,v]}]$ in the notation of~\autoref{def:semgpcond}. By~\autoref{thm:end-curvesNW}, these curves are smooth in the corresponding ambient tori, so $Z_{\wu{}}\cap (\CC^*)^n$ is smooth. This concludes our proof.
\end{proof}

  \begin{figure}[tb]
        \includegraphics[scale=0.5]{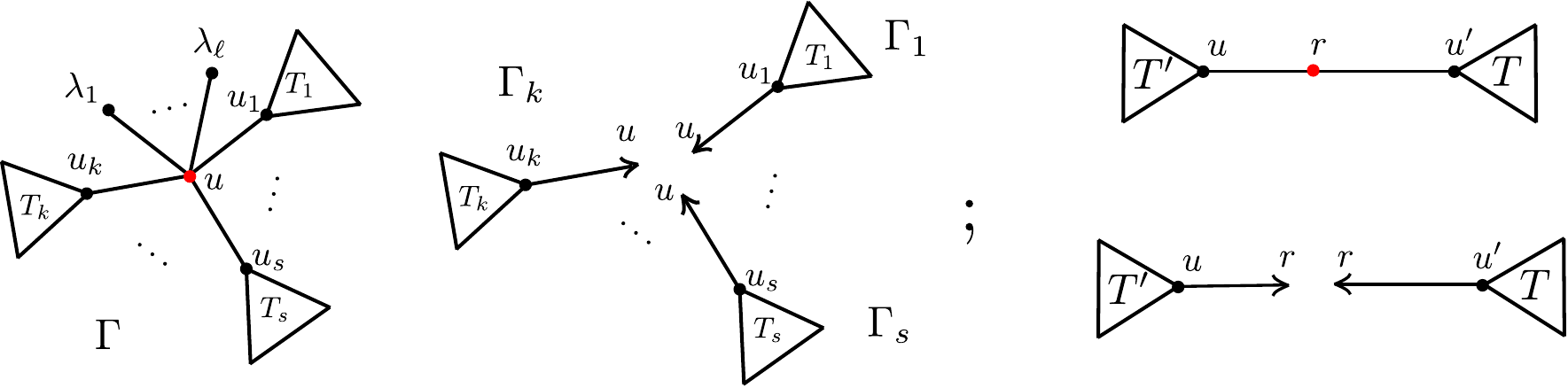}
    \caption{Rooted splice diagrams used in the proofs of~\textcolor{blue}{Lemmas}~\ref{lm:wu},~\ref{lm:wuv} and~\ref{lem:intersect-with-coord-subsp}. The vertex $u'$ on the right is allowed to be a leaf of $\Gamma$, whereas $u$ is always a node.}\label{fig:figSplicing}  \end{figure}

The next three lemmas characterize the ideal $J_{\wu{}}$ associated to points in the relative interior of each cell in the  piecewise-linearly embedded tree $\emb(\Gamma) \hookrightarrow \simplex{n-1}$. 

\begin{lemma}\label{lm:wu} 
    Let $u$ be a node of $\Gamma$ and let   $\wu{} = \wu{u}$.  Let $\Gamma_0$ be the star 
    of $\Gamma$ at $u$, viewed as a splice diagram with inherited weights around $u$. Then:
    \begin{enumerate}
      \item \label{Jusesf} $J_{\wu{}} = \langle \initwf{\wu{}}{\fvi{v}{i}}\colon v \text{ node of } 
            \Gamma, i=1,\ldots, \valv{v}-2\rangle$;
      \item\label{ZisCI} $Z_{\wu{}}$ is  a complete intersection of dimension two;
      \item \label{noBoundaryComp} 
         $Z_{\wu{}}$ has no irreducible component  contained in a coordinate hyperplane of $\CC^n$, so $J_{\wu{}}$ is monomial-free;
    \item \label{PBHForNodes} 
        all components of $Z_{\wu{}}$ are obtained as images of torus-translated monomial maps $\CC^{\valv{u}}\to \CC^{n}$ and their preimages lie in Pham-Brieskorn-Hamm complete intersections in $\CC^{\valv{u}}$ determined by $\Gamma_0$;
    \item \label{smoothness}
        the intersection of any two distinct components of $Z_{\wu{}}$ lies outside of $(\CC^*)^n$. 
    \end{enumerate}
\end{lemma}

  \begin{proof} Item~(\ref{Jusesf}) follows from the identity 
      $\initwf{\wu{u}}{\fgvi{v}{i}}=\initwf{\wu{u}}{\fvi{v}{i}}$  for all $v$ and $i$, seen in \autoref{pr:InitWufvi}.
    To prove the remaining items, we let $\boxedo{k}$ be the number of nodes 
    of $\Gamma$ adjacent to $u$ and $\boxedo{T_1},\ldots, \boxedo{T_k}$ be the  branches of $\Gamma$ adjacent to $u$ containing these nodes, as seen in~\autoref{fig:figSplicing}. Set $\boxedo{\Gamma_i} := [T_i,u]$ and  express  
$J_{\wu{}}$ as the sum of the following $k+1$ ideals:
\begin{equation*}
\boxedo{J_0}: =\langle \initwf{\wu{}}{\fvi{u}{j}}\colon j=1, \ldots, \valv{u}-2\rangle, \quad
\boxedo{J_i}: =\langle \initwf{\wu{}}{\fvi{v}{j}}\colon v \text{ node of } \Gamma_i, j=1,\ldots, \valv{v}-2\rangle,
\; \text{ for } i=1,\dots,k.
\end{equation*}
Note that the generators of each $J_i$ with $i\neq 0$ lie in $\CC[z_{\lambda}: \lambda \in \leavesT{T_i}]$. In particular, no variable $z_{\lambda}$ with $\lambda$ adjacent to $u$ appears in them.  We call these leaves $\boxedo{\lambda_1},\ldots,\boxedo{\lambda_l}$,  with $k+l=\valv{u}$.

To simplify notation, we write $\boxedo{Z} :=Z_{\wu{}}$.
We start by characterizing the components of $Z$.
Each ideal $J_i$ with $i\in \{1,\ldots,k\}$ defines an 
end-curve $\Ccurve_i$ determined by the splice diagram $\Gamma_i$ rooted at $u$, as seen in the center of \autoref{fig:figSplicing}.  In turn, \autoref{thm:end-curvesNW} ensures that each of the ($g_i$-many) components of $\Ccurve_i$ admits a (torus-translated) monomial parameterization of the form
\begin{equation}\label{eq:monMap}
  \CC[z_{\lambda}\colon \lambda \in \rDGi{u}{\Gamma_i}] \to \CC[t_{i}] \quad \text{ where } z_{\lambda} \mapsto \,c_{\lambda, i}^{(u)}\; t_{i}^{\wtuv{u}{\lambda}/g_i},
\end{equation}
with $\boxedo{c_{\lambda, i}^{(u)}}  \neq 0$ for all $\lambda, i$.
 Thus, each  component of $\Ccurve_1\times \cdots \times \Ccurve_k$ can be parameterized using  $\CC^{k}$ by combining these (torus-translated) monomial maps.  

  Substituting the expressions from~\eqref{eq:monMap} for the chosen component of each $\Ccurve_i$ into the generators of $J_0$ shows that the closure of each component of $Z$ can be parameterized using a component of a splice type singularity defined by the diagram $\Gamma_0$. Note that the Hamm determinant conditions at $u$ arising from  $\sG{\Gamma_0}$  agree with those determined by the generators of $J_0$ up to multiplying the columns corresponding to the branches $\Gamma_i$,
$i=1,\ldots,k$ by non-zero constants. Therefore, each component of $Z$ is the image of a (torus-translated) monomial map from $\CC^{\valv{u}}=\CC^k\times \CC^{\ell}$ restricted to a Pham-Brieskorn-Hamm complete intersection defined by  $\Gamma_0$. In particular, no component of $Z$ lies in a coordinate subspace of  $\CC^n$, so $J_{\wu{}}$ is monomial-free. This proves both~(\ref{noBoundaryComp}) and~(\ref{PBHForNodes}).

Since $Z$ is equidimensional of dimension two and it is  defined by $n-2$ polynomial  equations, it is a complete intersection. This proves~(\ref{ZisCI}).
It remains to address~(\ref{smoothness}) when $Z$ is not irreducible.

To determine the intersection of two distinct components (say, $Z_1$ and $Z_2$) of $Z$ we exploit the parameterization described earlier.
Let $I\subseteq \{1,\ldots, k\}$ be the collection of indices $i$ for which the projections of $Z_1$ and $Z_2$ to $\spec(\CC[t_i])$ disagree. If $I=\emptyset$, then the two components are parameterized using the same Pham-Brieskorn-Hamm system of equations associated to  $\Gamma_0$ and  the same (torus-translated) monomial map. This cannot happen since such germs are irreducible and $Z_1\neq Z_2$.

Next, assume $|I|\geq 1$. In this setting, the projection of $Z_1\cap Z_2$ to $\CC^{\rDGi{u}{\Gamma_i}}$ for each $i\in I$ lies in the intersection of two components of the end-curve $\Ccurve_i$, which can only be the origin of $\CC^{\rDGi{u}{\Gamma_i}}$ by~\autoref{thm:end-curvesNW}. This means that $Z_1\cap Z_2$ lies in the coordinate subspace of $\CC^n$ defined by the vanishing of all $z_{\lambda}$ with $\lambda \in \bigcup_{i\in I} \leavesT{T_i}$. This concludes our proof.
  \end{proof}

  \begin{lemma}\label{lm:wulambda} Let $u$ be a node of $\Gamma$ adjacent to a leaf $\lambda$ and pick $\wu{} \in \relo{[\emb({u}),\emb({\lambda})]}$. Then, $Z_{\wu{}}$ is a cylinder over a  monomial curve in  $\CC^{n-1}$ with $\gcd(\wtuv{\lambda}{\mu} :\mu\in \leavesT{\Gamma}\smallsetminus \{\lambda\})$ many components and $J_{\wu{}}$ is monomial-free.
  \end{lemma}
  
  \begin{proof}  We prove that $J_{\wu{}}$   defines the cylinder over the end-curve $\Ccurve_{\lambda}$ associated to the rooted splice diagram $\Gamma_{\lambda}$. 
       A simple inspection shows that for all nodes $v$ of $\Gamma$ and each $i\in \{1,\ldots, \valv{v}-2\}$, the initial form $\initwf{\wu{}}{\fgvi{v}{i}}$  agrees with the corresponding polynomial $\hvi{v}{i}$ defining $\Ccurve_{\lambda}$ (see~\autoref{def:endcurverSR}). Indeed, if we write $\wu{}$ as 
  \begin{equation}\label{eq:wwInTropTest}
\wu{} = \alpha\,\emb(\lambda) + (1-\alpha)\,\emb(u) \quad \text{ with }\quad 0<\alpha < 1,
  \end{equation}
then the defining properties of $\gvi{v}{i}$ ensure that $\initwf{\wu{}}{\fgvi{v}{i}} = \initwf{\wu{}}{\fvi{v}{i}}$ for all $v,i$. In turn, combining this fact with~\autoref{lem:keyid} and \autoref{pr:InitWufvi} yields
   \[\wu{} \cdot \wtNve{v}{e} \geq  (1-\alpha)\,\frac{\wtuv{u}{v}}{|\wu{u}|}   \qquad \text{    for each edge }e \text{ in }\starT{\Gamma}{v}.
   \]
   Furthermore, equality holds if and only if $e \nsubseteq [v,\lambda]$. Thus, $\initwf{\wu{}}{\fvi{v}{i}}$ is obtained from $\fvi{v}{i}$ by dropping the admissible monomial at $v$ pointing towards  $\lambda$.    In particular, these initial forms do not involve $z_{\lambda}$.
   
  The above discussion shows that 
    $\boxedo{J_{\wu{}}'} :=J_{\wu{}}\cap \CC[z_{\mu}\colon \mu \in \leavesT{\Gamma}\smallsetminus\{\lambda\}]$
      defines  $\Ccurve_{\lambda}$. The general point of this curve lies in  $(\CC^*)^{n-1}$ by~\autoref{thm:end-curvesNW}, so $J_{\wu{}}'$ is monomial-free.
   In turn, the same result confirms that $J_{\wu{}}'$ defines a reduced complete intersection, smooth outside the origin, and with $\gcd(\wtuv{\lambda}{\mu} :\mu\in \leavesT{\Gamma}\smallsetminus \{\lambda\})$ many irreducible components.
    Since $J_{\wu{}}$ is obtained from $J_{\wu{}}'$ by base change to $\CC[z_{\lambda}\colon \lambda \in \leavesT{\Gamma}]$, and both ideals admit a common generating set, the result follows.
  \end{proof}

  Given two adjacent nodes $u$ and $u'$ of $\Gamma$, we let $T'$  be the branch of $\Gamma$ adjacent to $u'$ containing $u$. Similarly, we let $T$ be the branch of $\Gamma$ adjacent to $u$ and containing $u'$.  Our final lemma is analogous to~\autoref{lm:wulambda}: 

  \begin{lemma}\label{lm:wuv} Let $u$ and $u'$ be two adjacent nodes of $\Gamma$ and pick    $\wu{} \in \relo{[\emb(u), \emb(u')]}$. Then, $Z_{\wu{}}$ is isomorphic to a product of two monomial curves in $\CC^{|\leavesT{T'}|} \times \CC^{|\leavesT{T}|}$. Furthermore, the number of irreducible components of $Z_{\wu{}}$ equals $\gcd(\wtuv{u'}{\lambda}':\lambda \in \leavesT{T'})\gcd(\wtuv{u}{\mu}': \mu \in \leavesT{T})$ and $J_{\wu{}}$ is monomial-free.
  \end{lemma}
  \begin{proof} We write $\wu{} = \alpha \wu{u} + (1-\alpha) \wu{u'}$ with $0<\alpha<1$ and follow the same proof-strategy as in~\autoref{lm:wulambda}. By~\autoref{thm:injectivityrho}, we can find a unique  $\roottree\in \relo{[u,u']}$ with $\emb(\roottree) = \wu{}/|\wu{}|$ (see the rightmost picture in~\autoref{fig:figSplicing}). 
    
  The conditions on $\alpha$ together with~\autoref{pr:InitWufvi} guarantee that for each node $v$ in $T'$, $\initwf{\wu{}}{\fgvi{v}{i}}=\initwf{\wu{}}{\fvi{v}{i}}$. Furthermore, each form is obtained from $\fvi{v}{i}$ by dropping the admissible  monomial at $v$  pointing towards $u'$. This observation and the symmetry between $u$ and $u'$ determine a partition of the generating set of $J_{\wu{}}$, where each initial form $\initwf{\wu{}}{\fgvi{v}{i}}$ for any node $v$ of $T'$ (respectively, in $T$) only involves the leaves of $T'$ (respectively, of $T$). Thus, each set determines an end-curve for the diagrams $[T',\roottree]$ and $[T, \roottree]$, respectively, both rooted at $\roottree$.

    By construction, $J_{\wu{}}$ defines the product of these two monomial curves. Since each of them is a complete intersection in their respective ambient spaces, the same is true for $J_{\wu{}}$. The number of components is determined  by~\autoref{thm:end-curvesNW}. Since $Z_{\wu{}}$ meets $(\CC^*)^n$, we conclude that $J_{\wu{}}$ is monomial-free.
  \end{proof}

  The next result is a direct consequence of~\autoref{thm:newL3.3} and the previous three lemmas. It gives an alternative proof of~\autoref{thm:EasyInclusionBalancing} since for any $\wu{}$ in $\relo{\emb(\Gamma)}$, these lemmas confirm that  $J_{\wu{}}$ is monomial-free.

  \begin{corollary}\label{cor:initialForms} 
         For any strictly positive vector  $\wu{}$ in the splice fan of $\Gamma$, we have $\initwf{\wu{}}{\langle \sG{\Gamma} \rangle}=J_{\wu{}}$, where $\langle \sG{\Gamma} \rangle$ is the ideal generated by the splice type system $ \sG{\Gamma}$. 
  \end{corollary}

  \begin{proof} By construction, we have $J_{\wu{}} \subseteq\initwf{\wu{}}{\sG{\Gamma}}$ for each $\wu{} \in \relo{(\simplex{n-1})}\cap \emb(\Gamma)$. In addition, \textcolor{blue}{Lemmas}~\ref{lm:wu},~\ref{lm:wulambda} and~\ref{lm:wuv}  imply that the generators of $J_{\wu{}}$ form a regular sequence in the ring of convergent power series  $\CC\{z_{\lambda}: \lambda \in \leavesT{\Gamma}\}$ near  the origin.  Thus,~\autoref{thm:newL3.3} implies that $J_{\wu{}}=\initwf{\wu{}}{\sG{\Gamma}}$  as ideals of $\CC\{z_{\lambda}: \lambda \in \leavesT{\Gamma}\}$ for each $\wu{}\in \Q^n\cap \Rp\langle \emb(\Gamma)\rangle$. 
    Finally, since these ideals are constant when we consider all rational points in the relative interior of any fixed edge of $\emb(\Gamma)$, \autoref{lem:fromRationalToIrrationa} below 
    confirms that the same must be true for all points in these open segments. 
  \end{proof}

  \begin{lemma}\label{lem:fromRationalToIrrationa}
    Fix an edge $[u,v]$ of $\Gamma$. If $J_{\wu{}}$ is constant for all $\wu{} \in \emb([u,v])\cap (\Q_{>0})^n$, the same is true for all $\wu{}\in \emb(\relo{[u,v]})$. 
  \end{lemma}

\begin{proof}
We set $\boxedo{J} :=  \langle \sG{\Gamma} \rangle$, $\boxedo{\sigma} := \Rp\langle \emb([u,v])\rangle$, and fix a weight vector $\wu{}'\in \relo{\sigma}\smallsetminus \Q^n$. We wish to show that $\initwf{\wu{}'}{J} = \initwf{\wu{}}{J}$ for all $\wu{}\in \relo{\sigma}\cap \Q^n$. We prove this equality by double-inclusion. 
  
  The inclusion $\initwf{\wu{}'}{J} \subseteq \initwf{\wu{}}{J}$ is a direct consequence   of the following claim:
\begin{claimNoNum} For each $f\in J$, we have $\initwf{\wu{}'}{f} \in \initwf{\wu{}}{J}$.
\end{claimNoNum}
\begin{subproof}

     Throughout our reasoning, we need the following auxiliary fact. By~\cite[Lemma 2.4.6]{MS 15}, given $\wu{1}$ and $\wu{2}$ in $(\R_{>0})^n$, there exists $\varepsilon >0$ with
    \begin{equation}\label{eq:iteratedInitialForms}
          \initwf{\wu{1}}{\initwf{\wu{2}}{f}} = \initwf{\wu{2} + \varepsilon' \wu{1}}{f}
    \end{equation}
for all $0<\varepsilon'<\varepsilon$ and for each $f\in \CC[z_{\lambda}: \lambda \in \leavesT{\Gamma}]$. Restricting further the value of $\varepsilon>0$ to require that
$\wu{2}+\varepsilon'\wu{1}\in (\R_{>0})^n$ whenever $0<\varepsilon'<\varepsilon$, extends the validity of~\eqref{eq:iteratedInitialForms} to convergent power series $f\in \CC\{z_{\lambda}:\lambda \in \leavesT{\Gamma}\}$.

  Since $\wu{}'\in (\R_{>0})^n$ by construction, the initial form $\initwf{\wu{}'}{f}$ is a polynomial. Its Newton fan is rational and complete, and we get an induced rational fan (denoted by $\Sigma$) on the cone $\sigma$ by common refinement. Thus,  $\wu{}'$ cannot be a ray of $\Sigma$ (the vector is not rational), so it lies in a 2-dimensional cone $\tau$ of $\Sigma$. By construction, $\tau\subseteq \sigma$. Picking any $\wu{}''\in \tau \cap (\Q_{>0})^n$, our earlier discussion ensures the existence of $\varepsilon>0$ with 
  \[ \initwf{\wu{}'}{f} = \initwf{\wu{}''}{\initwf{\wu{}'}{f}} = \initwf{\wu{}'+\varepsilon'\wu{}''}{f} \qquad \text{ for all }0<\varepsilon'<\varepsilon.
\]
Since the set $\{\wu{}'+\varepsilon'\wu{}''\colon 0<\varepsilon'<\varepsilon\}$ contains a point $\wu{}$ in $\relo{\sigma} \cap \Q^n$, we conclude that $\initwf{\wu{}'}{f}\in \initwf{\wu{}}{J}$.
\end{subproof}

For the reverse inclusion, we argue as follows. Direct calculations used in the proofs of \textcolor{blue}{Lemmas}~\ref{lm:wulambda} and~\ref{lm:wuv} confirm the identities
\[
\initwf{\wu{}}{\fgvi{v}{i}}=\initwf{\wu{}'}{\fgvi{v}{i}} \qquad \text{ for all nodes } v \in \Gamma \;\text{ and each}\; i\in\{1,\ldots, \valv{v} - 2\}.
  \]
  In turn,~\autoref{thm:newL3.3} ensures that this collection generates  $\initwf{\wu{}}{J}$. Thus, $\initwf{\wu{}}{J}\subseteq \initwf{\wu{}'}{J}$, as desired.
  \end{proof}

  \begin{remark}\label{rm:extendedSpliceFan} 
       The geometric information  collected in~\textcolor{blue}{Lemmas}~\ref{lm:wulambda} 
       and~\ref{lm:wuv} combined with~\autoref{cor:initialForms} determines the tropical multiplicities 
       of $\ptrop \langle\sG{\Gamma}\rangle$. Using~\autoref{rem:linkvsRedLink} 
       and~\autoref{def:tropMult}, we have:
\begin{enumerate}
       \item if $\tau = \Rp\emb([\lambda,u])$ for a node $u$ of $\Gamma$, then  
             $\tau$ has multiplicity 
                \[ \frac{1}{\du{u,\lambda}}\gcd(\wtuv{u}{\mu}: \mu \in   \leavesT{\Gamma}\smallsetminus 
                     \{\lambda\});  \]
       \item if $\tau = \Rp\emb([u,v])$ for two adjacent nodes $u$ and $v$ of $\Gamma$, 
            then $\tau$ has multiplicity 
         \[\frac{1}{\du{u,v}\du{v,u}}\gcd(\wtuv{u}{\lambda}: \lambda \in 
          \leavesT{\Gamma}, u\in [\lambda, v])\,     \gcd(\wtuv{v}{\lambda}: \lambda \in 
          \leavesT{\Gamma}, v\in [\lambda, u]).\]
\end{enumerate}
      This information completes the characterization of $\ptrop \langle\sG{\Gamma}\rangle$ as a tropical object, 
      i.e., as a weighted balanced polyhedral fan. We use this data
      in~\autoref{sec:recov-splice-diagr} when discussing how to recover $\Gamma$ from its splice fan. 
 \end{remark}

Our next statement follows naturally from~\autoref{prop:puredimltrop} and~\autoref{cor:initialForms}:

 \begin{corollary}\label{cor:tropsG}
   The  splice fan of $\Gamma$ is a standard tropicalizing fan 
   for the germ $(X,0)$ defined by $\sG{\Gamma}$, in the sense of~\autoref{def:tropicalizingFan}.
 \end{corollary}
 
Next, we show how we can use their local tropicalizations 
to recover some known facts about splice type systems and their associated end-curves from~\cite{NW 05bis, NW 05}.  Our first statement discusses the intersections of the initial degenerations of the germ defined by $\sG{\Gamma}$ will suitable codimension-2 coordinate subspaces of $\CC^n$:

\begin{lemma}\label{lem:intersect-with-coord-subsp}
Let $\wu{}$ be a vector in $(\R_{>0})^n$ which lies in the embedded splice diagram $\emb(\Gamma)$, and let $\lambda$ and $\mu$ be two leaves of $\Gamma$ in different branches of $\emb(\Gamma)$ adjacent to $\wu{}$. 
Then, the system
\begin{equation}\label{eq:initialSystem2Leaves}
  \begin{cases}
\initwf{\wu{}}{\fgvi{v}{i}}=0, \; v \text{ is a node of } \Gamma, \;\; i\in \{1,\dots, \valv{v} - 2\}, \\
z_\lambda=z_\mu=0
\end{cases}
\end{equation}
has $0\in \CC^n$ as its only solution.
\end{lemma}

\begin{proof}
    If $\wu{}=\emb(u)$ is the (normalized) weight vector corresponding to a
node $u$ of $\Gamma$, this statement is proved as part of~\cite[Theorem~2.6]{NW 05bis}, more precisely, on ~\cite[page 710]{NW 05bis}.
For any point $\wu{}$ in $\emb(\Gamma)\cap \relo{(\simplex{n})}$, the claim follows by similar arguments, but for the reader's convenience, we give a complete proof that includes the node case. Our reasoning is guided by~\autoref{fig:figSplicing}, and it is based on the techniques discussed in the proofs of \textcolor{blue}{Lemmas}~\ref{lm:wu},~\ref{lm:wulambda} and~\ref{lm:wuv}.

By~\autoref{thm:injectivityrho}, we can find a unique  point $r$ in $\Gamma$ with $\wu{}=\emb(\roottree)$. As in the figure, we write $\roottree =u$ if $u$ is a node of $\Gamma$. Otherwise, $\roottree$ is in the relative interior of a unique edge $[u,u']$, where $u$ is a node of $\Gamma$ but $u'$ is allowed to be a leaf of $\Gamma$. The proofs of the three aforementioned lemmas ensure that $\initwf{\wu{}}{\fgvi{v}{j}} = \initwf{\wu{}}{\fvi{v}{j}}$ for all nodes $v$ in $\Gamma$ and $j\in \{1,\ldots, \valv{v}-2\}$. Thus, to prove the statement, we may assume that $\gvi{v}{i}=0$ for all $v, i$.

By abuse of notation, we think of the point $\roottree$ as a node of $\Gamma$ (potentially, of valency two). We assume $\roottree$ is adjacent to $k$ nodes and $\ell$ leaves of $\Gamma$, with $k+\ell=\valv{\roottree}$. We let $T_1, \ldots, T_{\valv{\roottree}}$ be the branches of $\Gamma$ adjacent to $T$. We assume $T_1,\ldots, T_k$ contain nodes of $\Gamma$, whereas the remaining branches are singletons. We view the single vertex of a singleton branch as a leaf of this branch.

Using $\roottree$, we build a collection of  rooted splice diagrams $\Gamma_1,\ldots, \Gamma_k$ rooted at $\roottree$. Each  $\Gamma_i$ is obtained as the convex hull $[T_i,r]$.  In particular, if $\roottree\in \relo{[u,u']}$ (as in the right of the figure), then $k$ is the number of nodes of $\Gamma$ among $\{u,u'\}$. We assume that $u\in T_1$ whenever $\valv{\roottree}=2$.

   Throughout, we consider the ideals 
 \begin{equation}\label{eq:initialwJi}
\boxedo{J_i}  := \langle \initwf{\wu{}}{\fvi{v}{j}}\colon v \text{ node of } 
        \Gamma_i, j=1,\ldots, \valv{v}-2\rangle \; \;\text{ for } i\in \{1,\dots,k\}.
 \end{equation}
 The generators of $J_i$ are among the series of the system~\eqref{eq:initialSystem2Leaves} per our assumption on the tails $\gvi{v}{i}$. Since $\roottree$ is not a node of $\Gamma_i$, \autoref{lem:keyid} confirms that $J_i$ is generated by polynomials in  $\CC[z_{\lambda'}: \lambda'\in \leavesT{T_i}]$. Viewed in this ring, the ideal $J_i$ defines an end-curve $\Ccurve_i$ in $\CC^{|\leavesT{T_i}|}$ associated to the rooted diagram $\Gamma_i$ (see~\autoref{rm:EndCurvesInitialForms}).

Next, we consider a solution $\underline{p} \in \CC^n$ of the system~\eqref{eq:initialSystem2Leaves}. Our assertion that $\underline{p}$ is the origin will follow by combining the next two claims:

\begin{claimA}\label{cl:oneLeafThenAll}
      Fix $i\in \{1,\ldots, \valv{\roottree}\}$ and assume that $p_{\lambda'}=0$ for some leaf 
        $\lambda'$ of $T_i$. Then, $p_{\lambda}=0$ for all leaves $\lambda$ of $T_i$. 
 \end{claimA}
 
 \begin{subproof} If $i>k$, then $T_i=\{\lambda'\}$ and the statement is tautological. On the contrary, if $i\leq k$, then the projection of $\underline{p}$ to $\CC^{|\leavesT{T_i}|}$ lies in the end-curve $\Ccurve_i$. Since $p_{\lambda'}=0$,~\autoref{thm:end-curvesNW} guarantees that the projection of $\underline{p}$ is the origin of $\CC^{|\leavesT{T_i}|}$, as we wanted.
 \end{subproof}

 \begin{claimB}\label{cl:ifTwoBranchesThenAll}
     Assume that all coordinates of $\underline{p}$ indexed by two different branches (say, $T_{i_1}$ and $T_{i_2}$) adjacent to $\roottree$ vanish. Then, the same is true for each remaining branch.\end{claimB}
 \begin{subproof} If $\valv{\roottree}=2$, there is nothing to show. On the contrary, if $\valv{\roottree} \geq 3$, we consider the splice type polynomials 
 $\fvi{\roottree}{j}$ for $j\in \{1,\ldots, \valv{\roottree}-2\}$. By~\autoref{pr:InitWufvi}, $\initwf{\wu{}}{\fvi{\roottree}{j}} = \fvi{\roottree}{j}$ for each $j$. 

   Using the Hamm conditions, we may assume that each $\fvi{\roottree}{j}$ involves exactly three admissible monomials, two of which are $\zexp{\wtNve{\roottree}{T_{i_1}}}$ and $\zexp{\wtNve{\roottree}{T_{i_2}}}$. The third one equals $\zexp{\wtNve{\roottree}{T_{i_j}}}$. Since the first two monomials vanish along $\underline{p}$ by construction, we deduce that the same is true for the remaining monomial in $\fvi{\roottree}{j}$. Thus, for each of the remaining branches $T_{i_j}$ adjacent to $\roottree$ we have $p_{\lambda'_j}=0$ for some leaf $\lambda'_j$ in $T_{i_j}$. \textcolor{blue}{Claim A} 
   ensures that the same is true for all leaves of $T_{i_j}$. This concludes our proof.
 \end{subproof}

 To finish our proof, it is enough to notice that the hypothesis of \textcolor{blue}{Claim B} 
 holds by combining \textcolor{blue}{Claim A} 
 with the fact that the leaves $\lambda$ and $\mu$ belong to distinct branches of $\Gamma$ 
 adjacent to $\roottree$.
\end{proof}

Our second statement gives a stronger version of~\autoref{cor:dominant-curve-map}:

\begin{corollary}\label{cor:dom-map}
  Assume that the conditions of~\autoref{lem:intersect-with-coord-subsp} hold and consider the map $F_{\wu{}}\colon \CC^n\to \CC^{n-2}$ obtained from the collection 
$\{\initwf{\wu{}}{\fgvi{v}{i}}\}_{v,i}$, ordered appropriately. Then, the restriction
of this map to the codimension-$2$ subspace $L=\{z_\lambda=z_\mu=0\}$ of $\CC^n$ is surjective.
\end{corollary}

\begin{proof}
By \autoref{lem:intersect-with-coord-subsp}, we see that the fiber of the restriction of $F_{\wu{}}$ 
to $L$ over the origin of $\CC^{n-2}$ is $0$-dimensional. By upper semicontinuity of fiber dimensions,
the generic fiber of $F_{\wu{}|L}$ is also $0$-dimensional. Since $\dim L=n-2$, the map 
$F_{\wu{}|L}$ is dominant.

Since $F_{\wu{}|L}$ is defined by weighted homogeneous functions, it admits a projectivization as a map between weighted projective spaces. But a dominant
projective map must be surjective. Thus, as an affine map, $F_{\wu{}|L}$ is surjective
as well.
\end{proof}

Next, we  recover \autoref{thm:spliceicis} (originally due to Neumann and Wahl)  by combining~\autoref{cor:codim2boundaryStataGerm} with the following result:

\begin{corollary}\label{cor:isolatedSing}
The singularity defined by the splice type system $\sG{\Gamma}$ is isolated. In particular, it is also irreducible.
\end{corollary}

\begin{proof}
Let $(X,0)\hookrightarrow \CC^n$ be the germ defined by $\sG{\Gamma}$. By~\autoref{cor:expDimensionsG}, we know that $X \cap (\CC^*)^n$ is dense in $X$. 
Since $\sG{\Gamma}$ is a Newton non-degenerate complete intersection system by~\autoref{thm:NewtonNonDeg}, we conclude that $X \subset \CC^n$ admits an embedded toric resolution  (see, e.g., Khovanskii~\cite[Section~2.7]{K 84} and Oka~\cite[Chapter III, Theorem (3.4)]{O 97}). 
In particular, for a suitable subdivision $\Sigma$ of the splice fan of $X$, the corresponding toric morphism $\pi\from \tv_\Sigma \to \CC^n$ induces an embedded resolution of the pair $(\CC^n,X)$. But as we saw in~\autoref{thm:tropsG}, the local tropicalization of $\sG{\Gamma}$ intersects the boundary of the non-negative orthant only along the canonical basis elements. It follows from this that the morphism $\pi$ is an isomorphism outside the origin, i.e., $\pi$ is a resolution of $(X,0)$. We conclude that the singularity at the origin is isolated. 
As $(X,0)$ is moreover a complete intersection of dimension two by~\autoref{cor:expDimensionsG}, it is automatically irreducible by Hartshorne's connectedness theorem \cite[Theorem 2.2]{H 62} (see also \cite[Theorem 18.12]{E 95}), which states that a complete intersection singularity cannot be disconnected by removing a closed subgerm of codimension at least two.
\end{proof}

We end this section by showing how to use our results to get embedded resolutions of complex plane curve singularities by composing re-embeddings of 
$\CC^2$ into higher-dimensional smooth spaces $\CC^n$
with toric modifications of $\CC^n$.
The fact that such resolutions are possible was proven by
Goldin and Teissier~\cite{GT 00} in the irreducible case
and recently by de Felipe, Gonz\'alez P\'erez and Mourtada~\cite[Theorem 2.27]{FGM 21}
in full generality.

\begin{corollary}  \label{cor:onetoricplanecurve}
   Let $(X, 0) \hookrightarrow \CC^2$ be the germ of a reduced complex analytic plane curve. Then, the ambient germ $(\CC^2,0)$ can be holomorphically re-embedded into a suitable higher-dimensional germ $(\CC^n,0)$ in such a way that the induced germ $(X, 0) \hookrightarrow \CC^n$ can be resolved by a single toric modification of $\CC^n$.
 \end{corollary}

\begin{proof}
     This result is a consequence of~\autoref{thm:NewtonNonDeg}, as we now explain. 
     Indeed, consider a
      \emph{completion}  $(\hat{X}, 0)\hookrightarrow \CC^2$ of the input germ
      $(X,0)\hookrightarrow \CC^2$ in the sense of~\cite[Definition 1.4.15]{GGP 20}.
      By construction, $(\hat{X},0)$ is also the germ of a reduced plane curve,
      it contains $(X, 0)$ as a subgerm
    and it admits an embedded resolution (that is, a modification $\pi \colon S \to \CC^2$ where $S$ is smooth and the total transform of $\hat{X}$ on $S$ has normal crossings) such that the strict transform of $\hat{X}$ intersects all the leaf components of the exceptional divisor. Here, leaf components correspond to leaves 
    of the dual graph of the exceptional divisor.
    
    Moreover, possibly after extra blowups at points,
    the modification $\pi \colon S \to \CC^2$ can be chosen to ensure that  each leaf component
    is intersected by exactly one irreducible component of the strict transform of $\hat{X}$.
    Since all the irreducible components of $(\hat{X}, 0)$ are principal divisors on $(\CC^2, 0)$,
    and $(\CC^2,0)$ has an integral homology sphere link, Neumann and Wahl's
    \emph{end-curve theorem} \cite[Theorem 4.1 (3)]{NW 05} guarantees the existence of   
    a holomorphic embedding $\phi \from (\CC^2, 0) \to (\CC^n, 0)$
    such that $(\phi(\CC^2), 0)$ is a splice type singularity. Furthermore, 
    the irreducible components of $(\phi(\hat{X}), 0)$ 
    are exactly the intersections of $(\phi(\CC^2), 0)$ with the 
    coordinate hyperplanes of $(\CC^n, 0)$.

    Let $\Gamma$ be the splice diagram of the splice type singularity
    $(\phi(\CC^2), 0) \hookrightarrow (\CC^n, 0)$. Consider the splice fan $\cF$ of $\Gamma$,
    in the sense of \autoref{def:spliceFan}. By \autoref{cor:tropsG}, $\cF$
    is a standard tropicalizing fan of $(\phi(\CC^2), 0)$, in the sense of
    \autoref{def:tropicalizingFan}. Consider a regular subdivision
    $\cF'$ of $\cF$. It is also a standard tropicalizing fan of $(\phi(\CC^2), 0)$.
    Let $ \tv_{\cF'}$ be the toric variety defined by the fan $\cF'$ and let
    $\pi_{\cF'} \colon \tv_{\cF'} \to \CC^n$ be the birational toric morphism defined
    by the inclusion of the support of the fan $\cF'$ in the non-negative orthant
    $(\R_{\geq 0})^n$ of the weight lattice $\Z^n$ of $\CC^n$.

    Consider the strict transform $S'$ of $(\phi(\CC^2), 0)$ by $\pi_{\cF'}$ and
    denote by
       \[   \pi'\colon S' \to (\phi(\CC^2), 0)  \]
    the restriction of $\pi_{\cF'}$ to $S'$. By construction, the support of $\cF'$  equals the
    local tropicalization of $(\phi(\CC^2), 0)$ in $\CC^n$. Thus, 
    \autoref{prop:purecodimlift}~(\ref{condsamedim})
      implies that the bimeromorphic morphism $\pi'$ is proper.

      By \autoref{thm:NewtonNonDeg},
    $(\phi(\CC^2), 0)$ is defined by a Newton non-degenerate complete intersection system
    inside $\CC^n$. Therefore, the scheme-theoretic intersections of $S'$ with the
    orbits of the toric variety $\tv_{\cF'}$ are either empty or smooth 
    (see~\cite[Remark 6.4.18]{MS 15}). As a consequence,
    $S'$ intersects transversally the toric boundary of $\tv_{\cF'}$ 
    (see \cite[Proposition 3.9]{CPPS 22}). As $\cF'$ is a regular fan, the toric
    variety $\tv_{\cF'}$ is smooth. Therefore, $S'$ is also smooth
    and the total transform of $(\phi(\hat{X}), 0)$ in $S'$ is a normal crossings curve.
    
This fact and the properness of $\pi'$ imply that $\pi'$ is an embedded resolution of $(\phi(\hat{X}), 0) \hookrightarrow (\phi(\CC^2), 0)$.
    As $(X, 0) \hookrightarrow (\hat{X}, 0)$, the morphism $\pi'$ is also an embedded resolution
    of  $(\phi(X), 0) \hookrightarrow (\phi(\CC^2), 0)$. Therefore,
    $\phi^{-1} \circ \pi' \colon S' \to (\CC^2, 0)$ is an embedded resolution of 
    $(X,0) \hookrightarrow (\CC^2, 0)$, as we wanted to show.
 \end{proof}

\begin{remark}
Results from~\cite[Section 5]{GGP 19} allow to describe the splice diagram  associated to a plane curve singularity  in terms of the Newton-Puiseux series of its branches.  Applying \autoref{cor:tropsG} to this splice diagram yields a concrete description of the local tropicalization 
 of  the embedding $\phi$ from the proof of~\autoref{cor:onetoricplanecurve}
 in terms of standard combinatorial invariants of the given plane curve singularity.
 This method is similar to the one used by de Felipe, Gonz\'alez P\'erez and Mourtada 
 in~\cite[Section 3]{FGM 21} to characterize such local tropicalizations.
\end{remark}

The construction of splice type systems by Neumann and Wahl implies that the embedding dimension of a splice type singularity is bounded  from above by the number of leaves of the associated splice diagram. If an edge ending in a leaf has weight one, then it can be removed by simple elimination to produce a germ in lower dimension. As the following example illustrates, this operation need not necessarily produce a splice type singularity.

\begin{figure}[tb]  
         \includegraphics[scale=0.65]{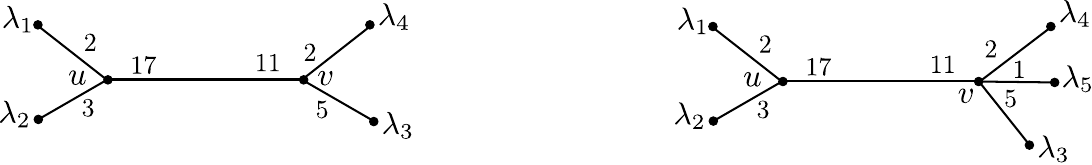}
         \caption{Two splice diagrams used in~\autoref{ex:weight1Leaf}: the right one is obtained from the left one by attaching an extra leaf through a weight one edge.\label{fig:weight1Example}
}               
  \end{figure}

\begin{example}\label{ex:weight1Leaf} Consider the splice diagram $\Gamma$ on the right of~\autoref{fig:weight1Example} satisfying the edge determinant and semigroup conditions. We consider two strict splice type systems for $\Gamma$:
  
  \begin{minipage}{0.5\textwidth}
    \begin{equation}
    \label{eq:leaft1weight}
    \begin{cases}
z_1^2 + z_2^3+ z_3z_4^3 = 0,\\
\;\;z_1z_2^4 + \;\;z_3^5+\;\; z_4^2 + z_5 = 0,\\
2\,z_1z_2^4 + 3 \,z_3^5 - 2\,z_4^2 + z_5 = 0,
  \end{cases}
    \end{equation}
  \end{minipage}
  \begin{minipage}{0.45\textwidth}
    \begin{equation*}
    \label{eq:leaft1weightmixed}
    \begin{cases}
      z_1^2 + z_2^3+ z_3z_4z_5 = 0,\\
\;\;z_1z_2^4 + \;\;z_3^5+\;\; z_4^2 + z_5 = 0,\\
2\,z_1z_2^4 + 3 \,z_3^5 - 2\,z_4^2 + z_5 = 0.
  \end{cases}
    \end{equation*}
  \end{minipage}

Eliminating the $z_5$ variable from the left system yields a splice system associated to the splice diagram on the left of the figure, namely $z_1^2 + z_2^3+ z_3z_4^3 = -z_1z_2^4 -2 z_3^5+ 3 z_4^2  = 0$. On the other hand, elimination on the rightmost system in~\eqref{eq:leaft1weight} produces a system of two equations in four unknowns that is not of splice type:
  \[
  \begin{cases}
    z_1^2+z_2^3 + z_3^6z_4 - 4z_3z_4^3 = 0,\\
    -z_1z_2^4 -2 z_3^5+ 3 z_4^2=0.
  \end{cases}\qedhere
  \]
\end{example}

But even if no such terminal edge with weigh one exists, the embedding dimension of a splice type singularity can still be smaller than the number of leaves in the splice diagram. In particular, \cite[Example 3]{NW 05} exhibits a hypersurface singularity $Z(f)$ in $\CC^3$ realizing a splice type surface singularity $(X,0)$ whose associate splice diagram  $\Gamma$ has six leaves.

A simple calculation reveals that the standard local tropicalization of this hypersurface is the cone over a star-shaped tree with three leaves (corresponding to the standard coordinate basis vectors) and a single node $\wu{}$. The $\wu{}$-initial form of $f$ is non-reduced so the Newton non-degeneracy condition fails for this hypersurface presentation, despite the fact that the splice type system $\sG{\Gamma}$ defining $(X,0)$ is a Newton non-degenerate complete intersection system by  \autoref{thm:NewtonNonDeg}.

Motivated by this example, we define the \emph{Newton non-degenerate embedding dimension} $\mathrm{edim}_{NND}(X,0)$
of an abstract germ 
$(X,0)$ as the smallest $n \geq 0$ such that $(X,0)$  may be embedded in $\CC^n$ as a 
Newton non-degenerate subgerm. In this context, the following question arises naturally:

\begin{question}   \label{ques:minNNDemb}
       Fix a splice diagram with $n$ leaves satisfying the edge determinant and semigroup conditions. 
       Assume that no edge of  $\Gamma$   ending in a leaf has  weight one. 
       Let $(X,0)$ be a splice type singularity associated to $\Gamma$. 
Is it true that  $\mathrm{edim}_{NND}(X,0) = |\leavesT{\Gamma}|$?
\end{question}

\section{Recovering splice diagrams from  splice fans}
\label{sec:recov-splice-diagr}

Throughout this section we assume that the splice type diagrams satisfy the 
edge determinant condition of \autoref{def:edgedet}  and the semigroup condition of \autoref{def:semgpcond}.
The construction of splice fans from splice diagrams introduced in \autoref{def:spliceFan} raises a natural question: how much data about $\Gamma$ can be recovered from its  splice fan, decorated with the tropical multiplicities? 
The main result of this section  answers this question 
under the following coprimality restrictions, which are central to~\cite{NW 05}:

\begin{definition}\label{def:coprimeSplice}
       We say that a splice diagram $\Gamma$ satisfying the 
edge determinant condition and the semigroup condition is \emph{coprime} if the weights around  each node of $\Gamma$ are pairwise coprime.
\end{definition}

 As mentioned in~\autoref{thm:charZHSlinks}, if $\Gamma$ is coprime, then there exists a unique integral homology 
sphere 3-manifold $\Sigma(\Gamma)$ associated to $\Gamma$. However, there may be also non-integral rational 
homology sphere links with the same splice diagram, since 
$\Gamma$ only determines the topological types of their universal abelian 
covers (see~\cite[page 2]{NW 05bis}).

Our next result highlights the restrictions on the  tropically weighted splice fan imposed by a coprime splice diagram. Its proof will be postponed until the end of this section. Precise formulas for the tropical multiplicities are given in~\autoref{rm:extendedSpliceFan}.
\begin{theorem}\label{thm:coprimePrimitiveAndWeightOne} Let $\Gamma$ be a coprime splice diagram. Then:
  \begin{enumerate}
            \item \label{item:wuPrim} for each node $v$ of $\Gamma$ the vector 
                   $\wu{v}\in \Nw{\Gamma} \simeq \Z^n$ is primitive;
             \item \label{item:tropMultOne} all tropical multiplicities of  $\ptrop \langle\sG{\Gamma}\rangle$ equal one.
  \end{enumerate}
\end{theorem}

The following is the main result in this section. It ensures that coprime splice diagrams 
can be recovered from their  tropically weighted splice fans:
\begin{theorem}\label{thm:recoverGammaFromTrop} Assume that all tropical multiplicities on the splice fan equal one. Then, there is a unique coprime  splice diagram $\Gamma$ yielding  the given weighted splice fan.
  
\end{theorem}
The rest of the section is devoted to the proof of these  two results. A series of lemmas and propositions will simplify the exposition.

\begin{remark}\label{rem:recoverFailsIfNonCoprime} Notice that the analog of~\autoref{thm:recoverGammaFromTrop} may fail if we drop the tropical multiplicity one restrictions.  This can be seen by looking at the example in~\autoref{fig:recoverDiagram}. For each choice of edge weight $d_1=1,2$ or $4$, the diagram $\Gamma$ satisfies the semigroup and edge determinant conditions. Furthermore, all tropical multiplicities on the 2-dimensional cones of the splice fan  equal  four. For each value of $d_1$ we can choose systems $\sG{\Gamma}$ defining a germ in $\CC^4$, whose local tropicalization is supported on the input splice fan.
\end{remark}
\begin{figure}
  \includegraphics[scale=0.55]{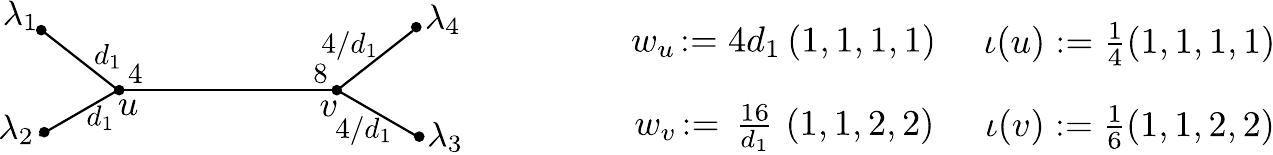}
  \caption{A splice diagram $\Gamma$ which cannot be recovered uniquely from  its splice fan, as in~\autoref{rem:recoverFailsIfNonCoprime}. Here, $d_1$ may take the values $1,2$ or $4$ and all tropical multiplicities of the splice fan equal $4$.\label{fig:recoverDiagram}}
  \end{figure}

Our first technical result will allow us to employ a pruning argument to prove~\autoref{thm:recoverGammaFromTrop}. To this end, we use superscripts to indicate the underlying splice diagram considered for the  computation of each linking number. The absence of a superscript refers to $\Gamma$. The same notation will be used for weight vectors. 

\begin{proposition}\label{pr:subDiagram} 
       Let $[u,v]$ be an internal edge of a splice diagram $\Gamma$. Let $T$ be the branch 
       of $\Gamma$ adjacent to  $u$ and containing $v$. Consider $\Gamma' = [u,T]$ 
       with weights around its nodes inherited from $\Gamma$. Then, the weighted tree 
       $\Gamma'$ also satisfies the semigroup and the edge determinant conditions.
\end{proposition}

\begin{proof} 
   We only need to check that the semigroup conditions hold for $\Gamma'$. 
   The linking numbers for $\Gamma$ involving a vertex $v'$ of $\Gamma'$ with $v'\neq u$ 
   and a leaf $\lambda$ of $\Gamma$  can be obtained from those in $\Gamma'$ via:
  \begin{equation}\label{eq:linkingNumberRelns}
      \wtuv{v'}{\lambda} =
  \begin{cases}
    \wtuvG{v'}{\lambda}{\Gamma'} & \text{ if }\lambda  \in \leavesT{\Gamma'}\cap \leavesT{\Gamma},\\
    \wtuvG{v'}{u}{\Gamma'}\,\frac{\wtuv{u}{\lambda}}{\du{u,v}} & \text{ otherwise }.
  \end{cases}
      \end{equation}
Since the semigroup condition at each $v'$ holds for $\Gamma$, expression~\eqref{eq:linkingNumberRelns} implies the same is true for $\Gamma'$.
\end{proof}
Assume that $u$ and $v$ are adjacent nodes of $\Gamma$ and let $\Gamma'$ be the associated splice diagram introduced in~\autoref{pr:subDiagram}. Up to relabeling, we write $\leavesT{\Gamma}\smallsetminus \leavesT{\Gamma'}=\{\lambda_1,\ldots,\lambda_s\}$ for some $s$. Consider the following $n\times(n-s+1)$ matrix with integer entries in block form obtained from~\eqref{eq:linkingNumberRelns}:
\begin{equation}\label{eq:Auv}
A:=  \left (
    \begin{array}{c|c}
      \wtuv{u}{\lambda_1}/\du{u,v} & \\
      \vdots & \mathbf{0} \\
      \wtuv{u}{\lambda_s}/\du{u,v} & \\
      \hline
      \mathbf{0} & \operatorname{Id}_{n-s} 
    \end{array}
    \right ).
\end{equation}
A direct computation yields the following identity for the tree $\Gamma'$:
\begin{lemma}\label{lm:ImageOfA}  For each vertex $v'$ of $\Gamma'$ with $v'\neq u$ we have $A \,\wu{v'}^{\Gamma'} = \wu{v'}$. 
\end{lemma}

\begin{lemma}\label{lm:gcdCoprimeSetting} 
    Fix a coprime splice diagram  $\Gamma$.  Let $u,v$ be two adjacent nodes of $\Gamma$, 
    and let $\Gamma'$ be the diagram from~\autoref{pr:subDiagram}. Then, we have
    \[ \du{v,u} = \gcd(\wtuv{v}{\lambda}: \lambda \in \leavesT{\Gamma'}\cap \leavesT{\Gamma}). \]
\end{lemma}

\begin{proof} 
   The result follows by an easy induction on the number of nodes of $\Gamma'$. If $\Gamma'$ has a single node, the identity holds by the coprimality of the weights around $v$.    For the inductive step, we assume that  $v$ is adjacent to $q$ leaves  and $k$ nodes other than $u$, denoted by $\{\mu_1,\ldots, \mu_q\}$ and  $\{v_1,\ldots, v_k\}$, respectively. Then,
  \begin{equation}\label{eq:leafBranchesAtv}
    \gcd(\wtuv{v}{\mu_1}, \ldots, \wtuv{v}{\mu_q}) = \du{v,u} \prod_{j=1}^k \du{v,v_j}.
  \end{equation}

  For each $j \in \{1,\ldots,k\}$ we let $T_j$ be the branch of  $\Gamma'$ adjacent to  $v$ 
  and containing $v_j$. Let $\Gamma_j' = [v,T_j]$ be the corresponding splice diagram with 
  inherited weights. The inductive hypothesis on each $\Gamma_j'$ yields
    $ \du{v_j,v} = \gcd(\wtuv{v_j}{\mu}: \mu \in \leavesT{\Gamma_j'}\smallsetminus \{v\})$.
   The identity $\wtuv{v}{\mu} = (\wtuv{v_j}{\mu}/\du{v_j,v})\, (\du{v}/\du{v,v_j})$ where 
   $\mu \in \leavesT{\Gamma_j'}\smallsetminus \{v\}$ gives
    \begin{equation}\label{eq:RelbranchesAtV}
          \gcd(\wtuv{v}{\mu}: \mu \in \leavesT{\Gamma_j'}\smallsetminus \{v\}) = \frac{\du{v}}{\du{v,v_j}}.
    \end{equation}
    The result follows by combining~\eqref{eq:leafBranchesAtv} and~\eqref{eq:RelbranchesAtV} 
    with the coprimality of the weights at $v$.
 \end{proof}

\begin{proof}[Proof of~\autoref{thm:coprimePrimitiveAndWeightOne}]
  We prove the statement by induction on the number of nodes of $\Gamma$, which we denote by $p$. If $p=1$, we let $u$ be the unique node of $\Gamma$. Then, the coprimality condition
  \[
  \gcd(\du{u,\lambda_i}, \du{u,\lambda_j})=1 \quad \text{ for }i\neq j
  \]
  implies that $\wu{u}$ is a primitive vector. Furthermore, the formula in~\autoref{rm:extendedSpliceFan}  confirms that the tropical multiplicity associated to the edge $[u,\lambda]$ equals one since
  \[
  \frac{1}{\du{u,\lambda}}\gcd(\du{u,\lambda}\prod_{\gamma\neq \lambda, \mu} \du{u,\gamma}: \mu \in \leavesT{\Gamma}\smallsetminus \{\lambda\})=1.
  \]

  Next, assume $p>1$ and let $u$ be an end-node of $\Gamma$. Let $\{\lambda_1,\ldots, \lambda_s\}$ be the leaves adjacent to $u$ and let $v$ be the unique node of $\Gamma$ adjacent to $u$. We let $\Gamma'$ be the splice sub-diagram of $\Gamma$ obtained from $u$ and $v$, as in~\autoref{pr:subDiagram}. The coprimality condition for $\Gamma'$ and our inductive hypothesis confirm that for each vertex $v'$ of $\Gamma'$,  $\wu{v'}^{\Gamma'}$  is a primitive vector in $\Nw{\Gamma'}$, and all tropical multiplicities of $\ptrop\sG{\Gamma'}$ are one.

  Since the $\gcd$ of all maximal minors of the matrix $A$ equals 1 by the coprimality condition around $u$, it follows that $A$ maps primitive vectors in $\Nw{\Gamma'}$ to primitive vectors in $\Nw{\Gamma}$. This fact together with~\autoref{lm:ImageOfA} ensures that the vector $\wu{v'}$ is primitive whenever $v'\neq u$ is a node of $\Gamma'$.
If  $v'=u$ we have
\[  \wu{u}= \sum_{i=1}^s \frac{\du{u}}{\du{u,\lambda_i}} \wu{\lambda_i} + \frac{\du{u}}{\du{u,v}} \sum_{\mu\in \pag{v}} \frac{\wtuv{v}{\mu}}{\du{v,u}}\wu{\mu},
\]
where $\pag{v}$ is the set $\nodesTevRoot{u,[u,v]}$ from~\autoref{def:semgpcond}. Since $\gcd(\wtuv{v}{\mu}: \mu \in \pag{v})=\du{v,u}$ by~\autoref{lm:gcdCoprimeSetting}, and $\gcd({\du{u}}/{\du{u,\lambda_i}}: i=1,\ldots, s) = \du{u,v}$, the pairwise coprimality of weights around $u$ ensures that $\wu{u}$ is a primitive vector in $\Nw{\Gamma}$.

  To finish, we compute the  tropical multiplicities. \autoref{rm:extendedSpliceFan},~\autoref{lm:gcdCoprimeSetting} and the coprimality of weights around $u$ implies that the multiplicity corresponding to  the edge $[u,\lambda_i]$ of $\Gamma$ is one. Indeed, we have
\begin{equation*}
  \begin{aligned}
    \gcd(\frac{\wtuv{u}{\mu}}{\du{u,\lambda_i}} : \mu \in \leavesT{\Gamma}\smallsetminus \{\lambda_i\}) & = \gcd(\gcd(\frac{\du{u}}{\du{u,\lambda_i}\du{u,\lambda_j}} j=1,\ldots, s, j\neq i), \gcd(\frac{\wtuv{v}{\mu}}{\du{v,u}} \frac{\du{u}}{\du{u,\lambda_i}\du{u,v}} : \mu \in \pag{v}))\\
    & = \gcd(\gcd(\frac{\du{u}}{\du{u,\lambda_i}\du{u,\lambda_j}} j=1,\ldots, s, j\neq i), \frac{\du{u}}{\du{u,\lambda_i}\du{u,v}}) =  1.
  \end{aligned}
\end{equation*}
If we pick an edge $[v',\lambda_j]$ with $j>s$ we get multiplicity one by the inductive hypothesis applied to $\Gamma'$ combined with~\eqref{eq:linkingNumberRelns} and the coprimality of the weights around $u$. More precisely,
\begin{equation*}
  \begin{aligned}
    \gcd(\gcd(\frac{\wtuv{v'}{\lambda_k}}{\du{v',\lambda_j}}: k=s+1,\ldots, n, k\neq j),\gcd(\frac{\wtuv{v'}{\lambda_i}}{\du{v',\lambda_j}} : i=1,\ldots, s)) & = \\
\gcd(\gcd(\frac{\wtuvG{v'}{\lambda_k}{\Gamma'}}{\du{v',\lambda_j}}: k=s+1,\ldots, n, k\neq j),\frac{\wtuvG{v'}{u}{\Gamma'}}{\du{v',\lambda_j}}\underbrace{\gcd(\frac{\wtuv{u}{\lambda_i}}{\du{u,v}} : i=1,\ldots, s)}_{=1}) & = 1.
  \end{aligned}
\end{equation*}
Finally, the cone associated to an edge between two adjacent nodes $u',v'$ of $\Gamma$ will have tropical multiplicity one by~\autoref{lm:gcdCoprimeSetting} since
$\gcd({\wtuv{u'}{\lambda}}/{\du{u',v'}} : \lambda \in 
\nodesTevRoot{v',[u',v']})  =  \gcd({\wtuv{v'}{\mu}}/{\du{v',u'}}: \mu \in 
\nodesTevRoot{u',[u',v']}) = 1$.
\end{proof}
  
\begin{proof}[Proof of~\autoref{thm:recoverGammaFromTrop}] The combinatorial type of $\Gamma$ is completely determined by intersecting $\simplex{n-1}$ and the  splice fan. In turn,  \autoref{thm:coprimePrimitiveAndWeightOne}~(\ref{item:wuPrim})  allows us to characterize the vector $\wu{u}$ as the primitive vector associated to the corresponding ray of the fan $\Rp\emb(\Gamma)$.
All that remains is to determine the weights around each node of $\Gamma$ from this data.  We do so by induction on the number of nodes of $\Gamma$, which we denote by $p$.

  If $p=1$, then the coprimality of the weights around the single node $u$ of $\Gamma$ determines each $\du{u,\lambda_i}$ uniquely as follows. By construction, the entries of $\wu{u}$ are coprime and we have
  \[\du{u,\lambda} = \gcd((\wu{u})_{\mu} : \mu \in \leavesT{\Gamma}\smallsetminus \{\lambda\}).
\]

Next, assume $p>1$ and fix an end-node $u$ of $\Gamma$. Let $\lambda_1,\ldots, \lambda_s$ be the leaves of $\Gamma$ adjacent to $u$, and $v$ be  the unique node of $\Gamma$ adjacent to $u$. Let $\Gamma'$ be the tree obtained by pruning $\Gamma$ from $u$, as in~\autoref{pr:subDiagram}.
The weights around $u$ can be recovered uniquely from  the splice fan of $\Gamma$.  Indeed, write 
\begin{equation}
  \wu{u} = 
  \sum_{i=1}^s \frac{\du{u}}{\du{u,\lambda_i}}\,\wu{\lambda_i} + \frac{\du{u}}{\du{u,v}}  \sum_{j=s+1}^n  \frac{\wtuvG{v}{\lambda_{j}}{\Gamma'}}{\du{v,u}} \,\wu{\lambda_j}.
\end{equation}
Notice that the coprimality condition gives $\du{u,v} = \gcd(\du{u}/\du{u,\lambda_i}: i =1,\ldots, s)$ and $\du{u} = \lcm(\du{u}/\du{u,\lambda_i}: i =1,\ldots, s)$. From this we recover all remaining $s$ weights at $u$ since $\du{u,\lambda_i} = \du{u}/(\du{u}/\du{u,\lambda_i})$ for every $i\in \{1,\ldots, s\}$.

Next, for each node $v'$ of $\Gamma$ with $v'\neq u$, we use the full-rank matrix $A$ from~\eqref{eq:Auv} to recover $\wu{v'}^{\Gamma'}\in \Z^{n-s+1}$ uniquely from $\wu{v'}$. Since $\wu{u}^{\Gamma'}$ is a prescribed canonical basis element of $\Nw{\Gamma'}\simeq \Z^{n-s+1}$, the set of vectors $\{\wu{v'}^{\Gamma'}: v' \text{ is a node of } \Gamma'\}$ allows us to determine the  splice fan of $\Gamma'$. The inductive hypothesis then uniquely recovers the splice diagram $\Gamma'$, and hence $\Gamma$ has been fully determined.
\end{proof}

\section{Intrinsic nature of splice type singularities}
\label{sec:natur-splice-type}

\autoref{thm:tropsG} shows that the local tropicalization of the germ defined by a given splice type system $\sG{\Gamma}$ is independent of the choice of admissible monomials and higher order terms used to define it.  In fact, Neumann and Wahl proved that the set of splice type singularities defined by splice type systems with fixed admissible monomials associated to a given splice diagram is independent of these choices, both in the coprime setting and in the general case under a suitable equivariant hypothesis on the series $\gvi{v}{i}$ collecting the higher order terms of each series $\fgvi{v}{i}$. In this section we give a variant of their proof in the coprime case and we show by an example that without the equivariance hypothesis, the result no longer holds.
\medskip

Here is the precise statement for coprime diagrams, which can be deduced from the equivariant case~\cite[Theorem 10.1]{NW 05bis}. For completeness, we include a direct proof:

        \begin{theorem}\label{thm:naturality} 
           Let $\Gamma$ be a coprime splice diagram with $n$ leaves, and let 
           $$\mathcal{M}:=\{\zexp{\wtNve{v}{e}}: v \text{ is a node in }\Gamma, e \in \starT{\Gamma}{v}\}$$ 
           be  a complete set of admissible monomials for $\Gamma$. 
          Let $\sG{\Gamma}_{\mathcal{M}}$ be the set of all splice type systems 
           that can be constructed using the set $\mathcal{M}$. 
          Then, the set  $X_{\mathcal{M}}$ of all germs $(X,0)\hookrightarrow \CC^n$ 
          defined by the vanishing of a system in $\sG{\Gamma}_{\mathcal{M}}$ 
          is independent of $\mathcal{M}$.
        \end{theorem}
        
        \begin{proof} Let $\mathcal{M'}$ be a second complete set of admissible monomials. We proceed by induction on the size $p$ of the set $\mathcal{ M}\smallsetminus \mathcal{M'}$. 
        If $p=0$, there is nothing to show. For the inductive step, it suffices to analyze the case 
        when $p=1$, i.e. when $\mathcal{M}$ and $\mathcal{M'}$ differ by exactly 
        one admissible monomial.
          Fix a pair $(v_0,e_0)$ for which the corresponding monomials in $\mathcal{M}$ and 
          $\mathcal{M'}$ differ, and let $\wtNve{v_0}{e_0}$ and $\wtNve{v_0}{e_0}'$ be the associated admissible exponents.
          To prove the statement, we use the  $\wu{v_0}$-filtration 
            $\mathcal{I}_{\bullet}\colon I_0\supseteq I_1 \supseteq I_2 \supseteq \ldots$
          of ideals of the local ring $\cO :=\CC\{z_1,\ldots, z_n\}$, where
          \begin{equation}\label{eq:vFiltrationIdeal}
            I_d:=\{f\in \cO: \wu{v_0}(f) \geq d\}.
          \end{equation}
          
           Let $\{\fgvi{v}{i}: v, i\}$ be a splice type system in $\sG{\Gamma}_{\mathcal{M}}$. 
           It determines a germ 
          $(X,0) \in X_{\mathcal{M}}$ together with an embedding 
          $\varphi\colon X \hookrightarrow \CC^n$.           We write $\fgvi{v}{i}=\fvi{v}{i} + \gvi{v}{i}$ 
          for each node $v$ of $\Gamma$ and $i\in \{0, \ldots, \valv{v}-2\}$ as in~\eqref{eq:surface} 
          and~\eqref{eq:surfaceSeries}.  The $\wu{v_0}$-filtration $\mathcal{I}_{\bullet}$ 
          restricted to $X$ yields a 
          filtration $\mathcal{J}_{\bullet}$ under the corresponding surjective map of 
          local rings $\varphi^{\sharp}\colon \cO \twoheadrightarrow \cO_{(X,0)}$, i.e.,
          \begin{equation}\label{eq:JIdealFiltration}
                 J_{k}:= \varphi^{\sharp}(I_k) \quad \text{ for all } k\geq 0.
          \end{equation}

          By~\autoref{lm:PropSection10NW} below, there exist $a\in \CC^*$, $g\in I_{\du{v_0}+1}$ 
          and $h$ in the ideal of $\cO$ spanned by $\{\fgvi{v}{i}: v, i\}$, satisfying the equality
          \begin{equation}\label{eq:swapAdmMonomials}
            \zexp{\wtNve{v_0}{e_0}} - a\, \zexp{\wtNve{v_0}{e_0}'} = h + g.
          \end{equation}
          Furthermore, up to moving  to $g$ all higher-order contributions of terms in $h$  
          coming from each $\fgvi{v_0}{i}$,  
          we may assume that $h = \sum_{v,i} a_{v,i} \fgvi{v}{i}$ with $a_{v_0,i}\in \CC$ 
          for each $i\in \{1,\ldots, \valv{v_0}\}$. 
          This can be achieved thanks to~\eqref{eq:admMonwuvalue} and~\eqref{eq:gviConditions}.
          
          We prove the inclusion $X_{\mathcal{M}} \subseteq X_{\mathcal{M}'}$ 
          by constructing an explicit 
          system $\{\fgvi{v}{i}': v, i\}$ in $\sG{\Gamma}_{\mathcal{M}'}$ whose vanishing set equals $X$. 
          The reverse inclusion then follows by exploiting the symmetry between $\mathcal{M}$ 
          and $\mathcal{M'}$.
          We consider the set $\{\fgvi{v}{i}': v,i\}$ with $\fgvi{v}{i}' := \fgvi{v}{i}$  for each $v\neq v_0$ and  $i\in \{1,\ldots, \valv{v}-2\}$, whereas   
for $v=v_0$ and $i\in \{1,\ldots, \valv{v_0}-2\}$ we pick  $\fgvi{v_0}{i}':= \fvi{v_0}{i}' + \gvi{v_0}{i}'$ with
          \begin{equation}\label{eq:newSystemMp}
                 \fvi{v_0}{i}' :=
                  (\fvi{v_0}{i} - \cvei{v_0}{e_0}{i} \zexp{\wtNve{v_0}{e_0}}) +  
                    (\cvei{v_0}{e_0}{i}\,a) \zexp{\wtNve{v_0}{e_0}'} \quad 
                \text{ and } \quad \gvi{v_0}{i}':= \gvi{v_0}{i} +  \cvei{v_0}{e_0}{i}\, g.
           \end{equation}
The scalar $a$ and the series $g \in I_{\du{v_0}+1}$ are those from~\eqref{eq:swapAdmMonomials}.

          We claim that $\{\fgvi{v}{i}'\}\in \sG{\Gamma}_{\mathcal{M}'}$. Indeed, by construction, each series $\gvi{v_0}{i}'$ lies in $I_{\du{v_0}+1}$, as required by~\eqref{eq:gviConditions} (see~\autoref{rm:gviConditions}). In addition, the matrix of coefficients for the polynomials $\{\fgvi{v_0}{i}'\}_{i}$ is obtained from the matrix $(\cvei{v_0}{e}{i})_{i,e}$ after rescaling by $a$ the column labeled with $e_0$. Thus, the Hamm determinant conditions  of~\autoref{def:splicesystem} are satisfied.
          
          Combining~\eqref{eq:swapAdmMonomials} and \eqref{eq:newSystemMp} yields
          \[
              \fgvi{v_0}{i}' = \fgvi{v_0}{i} -  \cvei{v_0}{e_0}{i}\, h \in \langle \fgvi{v}{j}: v 
              \text{ node of } \Gamma, j = 1,\ldots, \valv{v}-2\rangle.\]
We use the expression of $h$ given above to replace  $\{\fgvi{v}{i}' : v,i\}$ by a set  generating the same ideal, i.e.,
          \begin{equation}\label{eq:newGens}
              \{\fgvi{v}{i}: v \neq v_0, i=1,\ldots, \valv{v}-2\} \cup \; \{
                \fgvi{v_0}{i} - \cvei{v_0}{e_0}{i}\sum_{j=1}^{\valv{v_0}-2} a_{v_0,j} \,\fgvi{v_0}{j} 
                : i=1,\ldots, \valv{v_0}-2\}.
          \end{equation}

          Since both $\{\fgvi{v}{i}: v, i\}\in \sG{\Gamma}_{\mathcal{M}}$ and 
          $\{\fgvi{v}{i}': v, i\}\in \sG{\Gamma}_{\mathcal{M'}}$ determine complete intersection 
          systems of equations by~\autoref{thm:spliceicis}, the 
          $(\valv{v_0}-2)\times (\valv{v_0}-2)$-matrix of scalars 
          $\operatorname{Id} - (\cvei{v_0}{e_0}{i}\,a_{v_0,j})_{i,j}$ associated to the second set 
          in~\eqref{eq:newGens} must be invertible. From here it follows that the vanishing 
          sets of both collections $\{\fgvi{v}{i}': v, i\}$ and $\{\fgvi{v}{i}: v, i\}$ agree. Thus,  
          the germ $X$ lies in $X_{\mathcal{M}'}$, as we wanted to show.
        \end{proof}

        The following technical lemma gives a more precise version of the first half of the statement of~\cite[Theorem 10.1]{NW 05bis}. Its proof follows the same reasoning, so we omit it here:
        \begin{lemma}\label{lm:PropSection10NW} Fix two collections of admissible monomials $\mathcal{M}, \mathcal{M'}$ with $|\mathcal{M}\smallsetminus \mathcal{M'}|=1$. Assume $X\in X_{\mathcal{M}}$, and let $\zexp{\wtNve{v_0}{e_0}}\in \mathcal{M}\smallsetminus \mathcal{M'}$ and $\zexp{\wtNve{v_0}{e_0}'}\in\mathcal{M'}\smallsetminus \mathcal{M}$. Then, there exists $a\in \CC^*$ such that the restriction $(\zexp{\wtNve{v_0}{e_0}} - a \zexp{\wtNve{v_0}{e_0}'}\big)_{|_{X}}$ belongs to $J_{\du{v_0}+1}$, where $J_{\du{v_0}+1}$ is the ideal from~\eqref{eq:JIdealFiltration}.
        \end{lemma}

        As we mentioned earlier, if $\Gamma$ is not coprime,  analogous results 
        to~\autoref{thm:naturality} and~\autoref{lm:PropSection10NW} can be proved 
        under the condition that the higher order terms of each system 
        (i.e., the terms in each $g_{v,i}$) satisfy an equivariant condition under the action 
        of a suitable finite abelian group, namely, the discriminant group of a given plumbing graph with associated splice diagram $\Gamma$.
        For a precise statement, we refer to~\cite[Theorem 10.1]{NW 05}.

        It is natural to ask whether this equivariant condition can be weakened. 
        The next example shows that \emph{this is not the case}.

                \begin{figure}
          \includegraphics[scale=0.6]{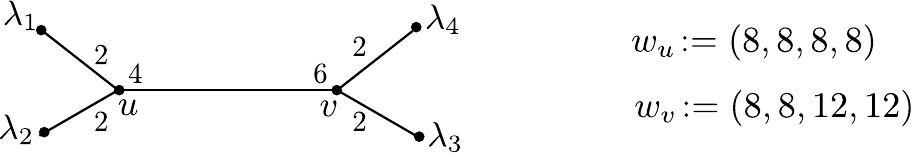}
        \caption{An example confirming the dependency of splice-type systems 
        on the choice of admissible monomials.\label{fig:MonomialDependency}}\end{figure}

        \begin{example}\label{ex:notNatural} 
            Consider the (non-coprime) splice diagram $\Gamma$ from
            ~\autoref{fig:MonomialDependency} and pick two sets of admissible monomials 
            for $\Gamma$ that differ only in the choice of exponent vectors for 
            the pair $(v,[v,u])$:
          \[
          \mathcal{M}:=\{z_1^2, z_2^2, z_3z_4, z_3^2, z_4^2\}\cup \{z_1^3\}, \quad          
           \mathcal{M}':=\{z_1^2, z_2^2, z_3z_4, z_3^2, z_4^2\}\cup \{z_1^2z_2\}.
          \]
          We claim that $X_{\mathcal{M}}\neq X_{\mathcal{M}'}$, i.e., the elements of 
          $\sG{\Gamma}_{\mathcal{M}}$ and $\sG{\Gamma}_{\mathcal{M}}$ 
          determine different sets of subgerms of $(\CC^n,0)$. 
          More precisely, we show that the germ in $X_{\mathcal{M}}$ defined by the system
  \begin{equation}\label{eq:noNaturalM}
    \begin{cases}  \fvi{u}{1}:= \;\;\;\, z_1^{2}\;\; \;+\;\;\; z_2^2\; +\;\;\; z_3\,z_4 , \\
  \fvi{v}{1}:=\;\;\;\; z_1^3 \;\; \;+\;\;\;  z_3^2 \;\; \;+\;\;\;z_4^2  ,
    \end{cases}
  \end{equation}
  in $\sG{\Gamma}_{\mathcal{M}}$  cannot be a member of $X_{\mathcal{M}'}$. 
  To do so, it suffices to show that no power series  associated to the node $v$ of $\Gamma$, that is,  no power series of the form
  \[
       \fgvi{v}{1}'  :=  b_1\,z_1^2z_2 + b_2\,  z_3^2  + b_3\,z_4^2 + \gvi{v}{1}, 
  \]
with $b_1,b_2, b_3\in \CC^*$ and $\gvi{v}{1}$ satisfying~\eqref{eq:gviConditions} 
can be an element of the ideal $(\fvi{u}{1}, \fvi{v}{1})$ of the power series ring 
$\CC\{z_1,\ldots, z_4\}$.
  
  We argue by contradiction and pick elements $A_1,A_2\in \CC\{z_1,\ldots, z_4\}$ with
\begin{equation}\label{eq:badSeries}
b_1\,z_1^2z_2 + b_2\,  z_3^2  + b_3\,z_4^2 + \gvi{v}{1} = A_1 \,\fvi{u}{1} + A_2\, \fvi{v}{1}.
\end{equation}
By construction, the $\wu{v}$-initial form on the left-hand side is $b_1\,z_1^2z_2 + b_2\,  z_3^2  + b_3\,z_4^2$ and its $\wu{v}$-weight is 24.
We claim that the $\wu{v}$-initial form on the right-hand side 
of~\eqref{eq:badSeries} may be written as
\begin{equation*}\label{eq:rightHandSideInitial}
  \alpha_1(\zu) (z_1^2+z_2^2) + \alpha_2(\zu) (z_1^3 + z_3^2 + z_4^2),
\end{equation*}
for two $\wu{v}$-homogeneous polynomials  $\alpha_1(\zu), \alpha_2(\zu)$. 

We prove this claim by explicit computation, comparing the $\wu{v}$-weights of both summands and noticing that
$\initwf{\wu{v}}{\fvi{u}{1}} = {z_1^2+z_2^2}$, $\initwf{\wu{v}}{\fvi{v}{1}} = \fvi{v}{1}$.
Three situations can occur.
First, if $\wu{v}(A_1 \fvi{u}{1}) < \wu{v}(A_2 \fvi{v}{1})$, 
then the $\wu{v}$-initial form on  the right-hand side of~\eqref{eq:badSeries} comes from the first summand, i.e., $\alpha_1(\zu) = \initwf{\wu{v}}{A_1}$ and $\alpha_2(\zu)=0$. Similarly, if $\wu{v}(A_1\fvi{u}{1}) > \wu{v}(A_2\fvi{v}{1})$, then the second summand determines the $\wu{v}$-initial form on the right-hand side of~\eqref{eq:badSeries}, so $\alpha_1(\zu)=0$ and $\alpha_2(\zu)=\initwf{\wu{v}}{A_2}$. Finally, if $\wu{v}(A_1\fvi{u}{1}) = \wu{v}(A_2\fvi{v}{1})$, the condition that the total $\wu{v}$-weight of $A_1\fvi{u}{1}+A_2\fvi{v}{1}$ and $\fvi{v}{1}$ agrees with the $\wu{v}$-weight of  $\fgvi{v}{1}'$ confirms that both terms contribute to the $\wu{v}$-initial form, with $\alpha_1(\zu) = \initwf{\wu{v}}{A_1}$ and $\alpha_2(\zu)=\initwf{\wu{v}}{A_2}$.

Comparing the $\wu{v}$-initial forms on both sides of~\eqref{eq:badSeries} yields an identity of $\wu{v}$-homogeneous polynomials:
\begin{equation*}\label{eq:initialIdentity}
 b_1\,z_1^2z_2 + b_2\,  z_3^2  + b_3\,z_4^2 =  \alpha_1(\zu) (z_1^2+z_2^2) + \alpha_2(\zu) (z_1^3 + z_3^2 + z_4^2).
\end{equation*}
Since the $\wu{v}$-weight on both sides equals $24$, we conclude that $\alpha_2(\zu)$ must be a constant. Evaluating both sides at  $z_1=z_2=0$ forces $b_2=b_3=\alpha_2(\zu)$, so in particular $\alpha_2:=\alpha_2(\zu)\in \CC^*$. We conclude from this that
\[b_1\,z_1^2z_2 =   \alpha_1(\zu) (z_1^2+z_2^2) + \alpha_2 \,z_1^3.\]
This identity implies that the function $z_1^3/(z_1^2z_2) = b_1\alpha_2^{-1}$ is constant on $Z(z_1^2+z_2^2)\cap (\CC^*)^2$, which is false. This contradiction confirms that $X_{\mathcal{M}}\neq X_{\mathcal{M}'}$, as we wanted to show.
        \end{example}

    \section*{Acknowledgments}
Maria Angelica Cueto was supported by an NSF
postdoctoral fellowship DMS-1103857 and NSF Standard Grants DMS-1700194 and DMS-1954163 (USA). 
Patrick Popescu-Pampu was supported by French grants ANR-12-JS01-0002-01 SUSI, 
ANR-17-CE40-0023-02 LISA and Labex CEMPI (ANR-11-LABX-0007-01). 
Labex CEMPI also financed a one month research stay of the first author in Lille during Summer 2018.

Part of this project was carried out during two Research in triples programs, 
one at the Centre International de Rencontres Math\'ematiques 
(Marseille, France, Award number: 1173/2014)
 and one at the Center International Bernoulli (Lausanne, Switzerland).  
 The authors would like to thank both institutes for their hospitality, 
 and for providing excellent working conditions. 
 
The authors are very grateful to Jonathan Wahl for contributing the proof presented 
in the appendix of this paper. Finally, we express our warm thanks to the referee 
for comments and suggestions that helped us improve the exposition.

 \medskip
 
\appendix
\section{Initial ideals and local regular sequences \\ \small{(by Jonathan Wahl)}}
  \label{sec:appendixL3.3}

  In~\cite{NW 05bis}, the authors invoke a folklore lemma in commutative algebra 
  in order to prove several of their main theorems. This result involves regular sequences 
  in a polynomial ring and their initial forms with respect to integer weight vectors.
    As originally stated,~\cite[Lemma 3.3]{NW 05bis} is not quite-correct: the global setting must be replaced by a local one. 
  This appendix provides a complete proof of this result in the local setting of convergent power series near the origin, a result we could not locate in the literature.  This local version  agrees with the general  framework of~\cite{NW 05bis}.
Throughout, we let $n$ be a positive integer and let $\boxedo{(\cO, \mfkm)}$ denote the local ring of convergent power series $\CC\{z_1,\ldots, z_n\}$ near the origin. 
\medskip

  We start by stating our main result, namely, a reformulation of~\cite[Lemma 3.3]{NW 05bis} in the local setting. Its proof will be given at the end of this appendix, after discussing a series of preliminary technical results. Note that the same statement and proof will hold if $\cO$ denotes the localization of the polynomial ring $\CC[z_1, \dots, z_n]$ at the maximal ideal of the origin of $\CC^n$.
  
  \begin{theorem}\label{thm:newL3.3} Let $(f_1,\ldots, f_s)$ be a finite sequence of elements in the maximal ideal $\mfkm$ of $\cO$, and let $J$ be the ideal generated by them. Fix a positive weight vector $\wu{}\in (\Z_{>0})^n$. Assume that $(\initwf{\wu{}}{f_1},\ldots, \initwf{\wu{}}{f_s})$ is a regular sequence in $\cO$. Then:
    \begin{enumerate}
    \item\label{item:fregSeq} the sequence $(f_1,\ldots, f_s)$ is also regular, and
    \item \label{item:initialFormsGenerate}
      the $\wu{}$-initial ideal $\initwf{\wu{}}{J}\cO$ is generated by $\{\initwf{\wu{}}{f_1},\ldots, \initwf{\wu{}}{f_s}\}$.
    \end{enumerate}
  \end{theorem}

  \begin{remark}\label{rm:PolyVsLocal} 
   As mentioned earlier,~\autoref{thm:newL3.3} does not hold in the polynomial setting. 
   For instance, $(z_1(1-z_1), z_2(1-z_1))$ is a regular sequence in the local ring 
   $\CC\{z_1,z_2\}$ but not in the polynomial ring $\CC[z_1,z_2]$. However, 
   the sequence $(z_1,z_2)$ of initial forms with respect to any weight vector $\wu{}\in (\Z_{>0})^2$ 
   is regular in both rings.
  \end{remark}

  \begin{remark}\label{rm:converseFails} The regularity of the sequence of $\wu{}$-initial forms is needed in~\autoref{thm:newL3.3}. As an example, fix $n=4$,  $\wu{}=(1,1,1,1)$, and consider the sequence $(f_1,f_2)$ with
    \[f_1 :=z_1^2+z_2^4-z_3^3 \qquad \text{ and } \qquad f_2:=z_1\,z_2-z_4^3.
    \]
    By construction, $(f_1,f_2)$ is a regular sequence in $\cO$ defining an isolated complete intersection surface singularity. The sequence of  initial forms $(\initwf{\wu{}}{f_1}, \initwf{\wu{}}{f_2}) = (z_1^2, z_1\,z_2)$ is not regular, and the $\wu{}$-initial ideal of $(f_1,f_2)\cO$ is generated by $\initwf{\wu{}}{f_1}$, $\initwf{\wu{}}{f_2}$ and $\initwf{\wu{}}{z_2\,f_1 -z_1\,f_2} =  - z_2\,z_3^3 + z_1\,z_4^3$.
  \end{remark}

Throughout, we fix $\wu{}\in (\Z_{>0})^n$ and an arbitrary sequence $(f_1,\ldots, f_s)$ of elements of the maximal ideal $\mfkm$. We let $J$ be the ideal generated by the $f_i$'s. Consider the first few steps in the Koszul complex of $\cO$-modules determined by it (see, e.g., \cite[Sect. IV.A]{S 65}):
     \begin{equation} \label{eq:Koszul} 
\xymatrix{
  \boxedo{F} :=\displaystyle{\bigoplus_{1 \leq i<j \leq s} \cO \cdot e_{ij}}\ar[r]^-{d_2} & \boxedo{E} :=\displaystyle{\bigoplus_{1 \leq i \leq s} \cO\cdot e_i}\ar[r]^-{d_1} & \cO \ar[r] & \cO/J.}
     \end{equation}
     The map $d_1\colon E\to \cO$ sends $e_i$ to $f_i$ for each $i=1,\ldots, s$ and the kernel $\boxedo{R}$ of $d_1$ is the \emph{module of relations} between the given generators of $J$. The morphism $d_2\colon F\to E$ sends $e_{ij}$ to $f_j\,e_i-f_i\,e_j$, and its image is the submodule of ``trivial relations'' between $(f_1,\ldots, f_s)$. By definition, the image of $d_2$ lies in $R$, so we view $d_2$ also as as a map $d_2\colon F\to R$.

     By a standard result in commutative algebra (see, e.g., \cite[Prop.3, Chapter IV.A.2]{S 65}) we have:
    \begin{proposition}  \label{prop:charreg1} 
        The sequence $(f_1,\cdots,f_s)$ of elements in $\mfkm$ is regular in $\cO$ if and only if the Koszul complex \eqref{eq:Koszul} is exact at $E$.
    \end{proposition}

    Since the definition of $E$ does not depend on the order of the sequence $(f_1,\ldots, f_s)$, the following consequence arises naturally:
    \begin{corollary}\label{cor:reorderRegSeq}
      If $(f_1,\ldots, f_s)$ is a regular sequence in $\cO$, any reordering of it is also a regular sequence.
    \end{corollary}

    The weight vector $\wu{}$ inducing the $\wu{}$-weight valuation~\eqref{eq:valwf} on $\cO$ endows this ring with a weight filtration by ideals $(I_p)_{p\geq 0}$, where
          $\boxedo{I_p} :=\{g\in \cO : \wu{}(g)\geq p\}$.
    Similarly, we can filter $\cO$ via the ideals $(\mfkm^p)_{p\geq 0}$. Both filtrations are cofinal since
    \begin{equation}\label{eq:cofinal}
      I_{d p} \subseteq \mfkm^p \subseteq I_p \quad \text{ for all }\; p\geq 0,
      \end{equation}
    where  $\boxedo{d}$ is the maximum among all coordinates of $\wu{}$. It follows from this that the completions of $\cO$ with respect to both filtrations are canonically isomorphic. The completion induced by the $\mfkm$-adic filtration $(\mfkm^p)_{p\geq 0}$ is the ring of formal power series in $n$ variables.

    In a similar fashion, we can filter the modules $E$ and $F$ appearing in~\eqref{eq:Koszul}  via $(E_p)_{p\geq 0}$ and $(F_p)_{p\geq 0}$, respectively, by assigning the weights $\wu{}(f_i)$ and  $\wu{}(f_i)+\wu{}(f_j)$ to $e_i$  and $e_{ij}$, respectively. More precisely,
    \begin{equation}\label{eq:EFwtFiltration}
      \boxedo{E_p}:=\{ \sum_{i=1}^s a_i e_i : \wu{}(a_i) \geq p-\wu{}(f_i) \;\forall i
      \} \:\; \text{ and } \:\; \boxedo{F_p}:=\{ \sum_{i<j} b_{ij} e_{ij} : \wu{}(b_{ij}) \geq p-\wu{}(f_i) - \wu{}(f_j)\; \forall i,j\}.
    \end{equation}
    These choices ensure that the maps $d_1$ and $d_2$ from the Koszul complex~\eqref{eq:Koszul}  preserve the filtration. In addition, the module $R$ of relations is filtered as well, via
    \begin{equation}\label{eq:RwtFiltration}
      \boxedo{R_p}:=E_p \cap R.      
    \end{equation}

    We use these filtrations to define the $\wu{}$-initial forms on $E$ and $F$.  We state the definition for $E$, since the one for $F$ is analogous. The definition for $R$ is given by restriction.

    \begin{definition}\label{def:initwE}   Given any $g\in E$ with $g \neq 0$, we let $p$ be the unique integer such that $g\in E_{p}\smallsetminus E_{p+1}$. An element $g:=\sum_{i=1}^s r_i \, e_i \in E_p\smallsetminus E_{p+1}$ satisfies $\wu{}(r_i) + \wu{}(f_i) \geq p$ for all $i\in \{1,\ldots, s\}$ and equality must hold for some index $i$. Let $I$ be the set of indices where equality is achieved.   The \emph{$\wu{}$-initial form} of $g$ is $\boxedo{\initwf{\wu{}}{g}}:=\sum_{i\in I} \initwf{\wu{}}{r_i} e_i$.
        We set $\initwf{\wu{}}{0} = 0$.
\end{definition}

    By~\autoref{prop:charreg1}, the regularity of the sequence $(f_1,\ldots, f_s)$ is equivalent to the surjectivity of the map $d_2\colon F\to R$ induced by~\eqref{eq:Koszul}. We prove the latter in~\autoref{lm:zerocompl}, assuming the regularity of the sequence of $\wu{}$-initial forms of all $f_i$'s.
    
    Our first  two lemmas use the regularity assumptions for the sequence of $\wu{}$-initial forms to prove the surjectivity of $d_2\colon F\to R$  by working with the filtrations of $F$ and $R$ described above.

    \begin{lemma}\label{lm:surjall}
         Assume that  the sequence 
    $(\init_w(f_1),\cdots, \init_w(f_s))$ is regular in $\cO$. 
      Then, the morphism of $\CC$-vector spaces $\varphi_p\colon F_p/F_{p+1}\rightarrow R_p/R_{p+1}$ induced by the 
      morphism of $\cO$-modules $d_2\colon F \rightarrow R$ 
      is surjective for all integers $p \geq 0$.
    \end{lemma}
    \begin{proof} We must show that modulo $R_{p+1}$, every element $g$ of $R_p$ is the image of an element of $F_p/F_{p+1}$ under the map $\varphi_p$. If $g=0$, there is nothing to show, so we assume $g\neq 0$. In particular, $g$ lifts to an element in $R_p\smallsetminus R_{p+1}$, which we denote by $g$ as well.  We write $g=\sum_{j=1}^s r_j\,e_j$.

      Assume that $\initwf{\wu{}}{g}$ has $k$ many terms, with $k\in \{1,\ldots, s\}$ (see~\autoref{def:initwE}). By~\autoref{cor:reorderRegSeq}, we can reorder the original sequence while preserving its regularity, and write  $\initwf{\wu{}}{g}$ as
      \[\initwf{\wu{}}{g}= \sum_{j=1}^k \initwf{\wu{}}{r_{j}} e_{j}\quad \text{  with   }\quad \wu{}(r_{j}) + \wu{}(f_{j}) = p\quad \text{ for all }\;j\in \{1,\ldots, k\}.
      \]
      We claim that  $g$ is congruent, modulo the image of $\varphi_p$, to an element of $R_p$ whose $\wu{}$-initial form lies in the ideal generated by $\{\initwf{\wu{}}{f_1},\ldots, \initwf{\wu{}}{f_{k-1}}\}$. The original statement will follow by  induction on $k\leq s$.

Since $\sum_{j=1}^s r_j f_j = 0$ by definition of $R$ and $\wu{}(r_j)+\wu{}(f_j)>p$ for $j\in \{k+1, \ldots, s\}$, we conclude that the expected $\wu{}$-initial form of $\sum_{j=1}^s r_jf_j$ must vanish, i.e.,
      \[\sum_{j=1}^k \initwf{\wu{}}{r_{j}} \initwf{\wu{}}{f_{j}} = 0.
      \]
      Therefore, $ \initwf{\wu{}}{r_{k}} \initwf{\wu{}}{f_{k}}$ is zero modulo  the ideal $I:=\langle \initwf{\wu{}}{f_1}, \ldots, \initwf{\wu{}}{f_{k-1}}\rangle \cO$. Since the sequence $(\initwf{\wu{}}{f_1}, \ldots, \initwf{\wu{}}{f_s})$ is regular, we conclude that $\initwf{\wu{}}{r_{k}}$ must lie in $I$.

      Taking the $\wu{}$-weight value of $r_{k}$ and each $f_{j}$ into account we write $\initwf{\wu{}}{r_{k}}$ as
      \[
           \initwf{\wu{}}{r_{k}}  = \sum_{j=1}^{k-1} a_j \initwf{\wu{}}{f_{j}},
      \]
      where $a_j$ is either $0$ or a non-zero $\wu{}$-weighted homogeneous polynomial with $\wu{}(a_j)= p-\wu{}(f_{k})-\wu{}(f_{j})\geq 0$ for all $j\in \{1,\ldots, k-1\}$. It follows from this that the element
      \begin{equation*}\label{eq:newr}
      r_{k}':=r_{k}-\sum_{j=1}^{k-1} a_j f_j
      \end{equation*}
      satisfies $\wu{}(r_{k}')> p-\wu{}(f_{k})$, so $r_{k}'\, e_{k} \in E_{p+1}$. Simple arithmetic manipulations give a new formula for $g$, i.e,
      \begin{equation}\label{eq:newg}
        g = \sum_{j=1}^s r_j\, e_j = \sum_{j=1}^{k-1} (r_j + a_j\,f_{k}) \, e_j + \underbrace{\sum_{j=1}^{k-1} a_j(f_j\,e_{k} - f_{k} \,e_j)}_{=: \, h} + r_{k}'\, e_{k} + \sum_{j=k+1}^s r_j \,e_j.
      \end{equation}
      By construction, it follows that $h=\varphi_p(\sum_{j=1}^{k-1} a_j e_{jk}) \in \varphi_p(F_p/F_{p+1})$. Furthermore, $\initwf{\wu{}}{g-h}$ only involves terms in the first of the four summands on the right-hand side of~\eqref{eq:newg} since the last two summands lie in $E_{p+1}$. This establishes the claim.
    \end{proof}

    \begin{lemma}\label{lm:geninit} Let $J$ be the ideal of $\cO$ generated by $\{f_1,\ldots, f_s\}$ and assume that  $(\init_w(f_1),\cdots, \init_w(f_s))$ is a regular sequence in $\cO$. If  $g\in J$ has $\wu{}$-weight equal to $p \in \N$, then $g$ admits an expression of the form  $g=\sum_{i=1}^s a_if_i$, where $\wu{}(a_i f_i)\geq p$ for all $i$.  In particular, $\initwf{\wu{}}{g}$ belongs to the ideal 
      of $\cO$ generated by $\{\initwf{\wu{}}{f_1}, \dots, \initwf{\wu{}}{f_s}\}$.  
    \end{lemma}

    \begin{proof}
  Since $g \in J$, we may write $g$ as $g=\sum_{i=1}^s b_i\,f_i$ with $b_i\in \cO$ for each $i$. Consider
  \[
  p'=\min\{\wu{}(b_i f_i): i=1,\ldots, s\}.
  \]
  Assume that this weight is achieved at $k$ many terms, which we can fix to be   $\{b_{1}\,f_{1}, \ldots, b_{k}\,f_{k}\}$
 upon reordering. If $p'\geq p$ we have $\wu{}(a_if_i)\geq p$ for all $i$ and equality must hold for some $i$ by definition of $p$. From here it follows that $p'=p$, so 
 $\initwf{\wu{}}{g} = \sum_{j=1}^k \initwf{\wu{}}{a_{j}}\initwf{\wu{}}{f_{j}}$, as we wanted to show.

  On the contrary assume that $p'<p$.  We claim that  we can find an alternative expression $g=\sum_{j=1}^{s} b_j'\, f_j$  where the corresponding minimum weight $p'':=\min\{\wu{}(b_j'\, f_j)\}$ satisfies $p''\geq p'$ and the number of summands realizing $p''$ is strictly smaller than $k$. An easy induction combined with the fact that $p',p\in \Z_{\geq 0}$ will then yield a new expression for $g$ with $p'\geq p$, as in our previous case.

  It remains to prove the claim. Since  $p'<p$, the terms in $g$ with $\wu{}$-weight $p'$ must cancel out, i.e.,  $\sum_{j=1}^k \initwf{\wu{}}{b_{j}}\,\initwf{\wu{}}{f_{j}} =0$.  As in the proof of~\autoref{lm:surjall}, the fact that $(\initwf{\wu{}}{f_1}, \ldots, \initwf{\wu{}}{f_s})$ is a regular sequence in $\cO$ ensures that
  \[\initwf{\wu{}}{b_{k}} = \sum_{j=1}^{k-1} c_j \initwf{\wu{}}{f_{j}},
  \]
  where $c_j$ is either zero or a $\wu{}$-homogeneous polynomial with $\wu{}(c_j) = p'-\wu{}(f_j)-\wu{}(f_k)\geq 0$.  It follows from here that the element $b_{k}' := b_{k} - \sum_{j=1}^{k-1} c_j \, f_j$ has weight $\wu{}(b_{k}')> \wu{}(b_{k})$, so $\wu{}(b_{k}'\, f_{k}) > p'$.

  An arithmetic manipulation allows us to rewrite $g$ as follows:
  \begin{equation}\label{eq:newgL2}
    g = \sum_{j=1}^{k-1} (b_j + c_j\, f_{k}) \,f_j + b_{k}'\,f_{k} + \sum_{j=k+1}^s b_j\, f_j
  \end{equation}
  By construction, the terms with minimum $\wu{}$-weight only appear in the first of the three summands on the right-hand side of~\eqref{eq:newgL2}. Furthermore, the corresponding minimum weight $p''$ satisfies $p''\geq p'$ since
  $\wu{}(b_jf_j)\geq p'$ for all $j$ and $\wu{}(c_j \, f_{k}f_j)\geq p'$ for $j<k$. This confirms the validity of our claim.\end{proof}
     
A standard commutative algebra result (see, e.g.,~\cite[Lemma 10.23]{AM 69}) combined with~\autoref{lm:geninit} yields:

\begin{lemma} \label{lm:surjcompl}
    Assume that  the sequence 
    $(\init_w(f_1),\cdots, \init_w(f_s))$ is regular in $\mathcal{O}$. 
    Then, the  map $d_2:F \to R$ of filtered modules induces a surjection between 
    their completions relative to the filtrations $(F_p)_{p \geq 0}$ and $(R_p)_{p \geq 0}$ respectively. More precisely, $\varprojlim F/F_p\twoheadrightarrow \varprojlim R/R_p$.
\end{lemma}

We let  $\boxedo{\hat{F}}$ and $\boxedo{\hat{R}}$ be the $\mfkm$-adic completions of $F$ and $R$ respectively,  which can be computed with standard methods. Indeed, by~\cite[Theorem 10.13]{AM 69}, we have
\begin{equation}\label{eq:hatFR}\hat{F} \simeq F\otimes_{\cO} \hat{\cO}\quad \text{ and }\quad \hat{R} \simeq R\otimes_{\cO} \hat{\cO}.
\end{equation}
The double inclusions in~\eqref{eq:cofinal} allow us to compare the completions in~\autoref{lm:surjcompl} induced by $(F_p)_{p \geq 0}$ and $(R_p)_{p\geq 0}$,  with $\hat{F}$ and $\hat{R}$, respectively. More precisely,
\begin{lemma} \label{lm:samecompl}
     Assume that  the sequence 
    $(\initwf{\wu{}}{f_1},\cdots, \initwf{\wu{}}{f_s})$ is regular in $\cO$.
    Then, the completions appearing in \autoref{lm:surjcompl} agree with  the $\mfkm$-adic ones, i.e.
    \begin{equation}\label{eq:ComplFR}\varprojlim F/F_p\simeq \varprojlim F/\mfkm^p F\simeq 
              F\otimes_{\mathcal{O}} \hat{\mathcal{O}} \quad \text{ and }\quad
\varprojlim R/R_p\simeq \varprojlim R/\mfkm^p R \simeq 
R\otimes_{\mathcal{O}} \hat{\mathcal{O}}.
    \end{equation}
\end{lemma}

\begin{proof} We let $\ell:=\max\{\wu{}(f_j): j=1,\ldots, s\}$. It suffices to prove the first isomorphism on each side of~\eqref{eq:ComplFR}, since the remaining ones appear in~\eqref{eq:hatFR}.
    By~\eqref{eq:EFwtFiltration}, we have
  \[
  E_p:=\bigoplus_{i} I_{p-\wu{}(f_i)} \, e_{i}\qquad \text{ and }\qquad  F_p:=\bigoplus_{i<j} I_{p-\wu{}(f_i)-\wu{}(f_j)} \, e_{ij}\qquad \text{ for each }p\geq 0.
    \]
    It follows from here that
    $I_p E \subseteq E_p \subseteq I_{p-\ell} E$  and $I_p F \subseteq F_p \subseteq I_{p-2\ell} F$  for each $p\geq 0$.  Combining these inclusions with~\eqref{eq:cofinal} yields:
     \begin{equation}\label{eq:cofinalsEFs}
       E_{dp+\ell} \subseteq I_{dp}E \subseteq \mfkm^p E\subseteq E_p\quad \text{ and } \quad F_{dp+2\ell} \subseteq I_{dp} F \subseteq \mfkm^p F  \subseteq F_p \quad  \text{ for each }p\geq 0.
     \end{equation}
The inclusions appearing on the right of~\eqref{eq:cofinalsEFs} ensure that the filtrations $(\mfkm^pF)_{p\geq 0}$ and $(F_p)_{p\geq 0}$ are cofinal in $F$. Thus, they yield isomorphic  completions.  This proves the  first isomorphism in~\eqref{eq:ComplFR}.

Next, consider the filtration $R_p$ from~\eqref{eq:RwtFiltration}. First, notice that $\mfkm^pR \subseteq R_p$ by~\eqref{eq:cofinalsEFs}. To finish, we claim the existence of some $k\geq 0$ for which  $R_{dp+(dk+\ell)} \subseteq \mfkm^pR$ for all $p\gg 0$. Indeed, by the Artin-Rees Lemma (see, e.g.,~\cite[Theorem 10.10]{AM 69}), there exists an integer $k\geq 0$ satisfying
\[
R\cap \mfkm^p E = \mfkm^{p-k}(R \cap \mfkm^k E) \qquad \text{ for all } p\geq k.
\]
Therefore, combining this fact with  property~\eqref{eq:cofinalsEFs} we obtained the desired inclusion:
\[
R_{dp+(dk+\ell)} = R \cap E_{d(p+k)+\ell} \subset R\cap \mfkm^{p+k} E = \mfkm^{p}(R\cap \mfkm^k E) \subset \mfkm^p R.
\]
We conclude that $(\mfkm^pR)_{p\geq 0}$ and $(R_p)_{p\geq 0}$ are cofinal filtrations in $R$, so they yield isomorphic completions. 
\end{proof}

We let $M$ be the cokernel of the map $d_2\colon F\to R$ given by~\eqref{eq:Koszul}, and we let $\boxedo{\hat{M}}$ be its $\mfkm$-adic completion. \autoref{lm:samecompl} yields the following result:

\begin{lemma}\label{lm:zerocompl}
     Assume that  the sequence 
    $(\initwf{\wu{}}{f_1},\cdots, \initwf{\wu{}}{f_s})$ is regular in $\cO$. Then, $\hat{M}=0$ and  $M=0$. In particular, the Koszul complex~\eqref{eq:Koszul} is exact at $E$.
\end{lemma}

\begin{proof} By standard commutative algebra~(see, e.g.,\cite[Corollaire 2, Chap.~II.A.5]{S 65}) we know that $\hat{\cO}$ is a flat $\cO$-module. Therefore, taking $\mfkm$-adic completion is an exact functor. Since $\hat{F}\to \hat{R}$ is surjective (by combining~\textcolor{blue}{Lemmas}~\ref{lm:surjcompl} and~\ref{lm:samecompl}) it follows that $\hat{M}=0$.

  By~\cite[Theorem 10.17]{AM 69}, the kernel of the canonical morphism $M\to \hat{M}$ is annihilated by an element of the form $(1+z)$ where $z\in \mfkm$. As $\cO$ is a local ring, the element $(1+z)$ must be a unit of $\cO$, thus $M=0$ as claim. The exactness of the Koszul complex at $E$ follows immediately, as it is equivalent to the surjectivity of the morphism $d_2\colon F\to R$.
\end{proof}

We end this appendix by proving its main result:
\begin{proof}[Proof of~\autoref{thm:newL3.3}]
  Since $(\initwf{\wu{}}{f_1},\cdots, \initwf{\wu{}}{f_s})$ is regular in $\cO$, \autoref{lm:zerocompl} ensures that the Koszul complex~\eqref{eq:Koszul} is exact at $E$. In turn,~\autoref{prop:charreg1} implies that $(f_1,\ldots, f_s)$ is a regular sequence in $\cO$. This proves item~(\ref{item:fregSeq}) of the statement.

  To finish, we must show that the $\wu{}$-initial forms $\{\initwf{\wu{}}{f_1},\cdots, \initwf{\wu{}}{f_s}\}$ generate the $\wu{}$-initial ideal $\initwf{\wu{}}{J}\cO$.   By definition, the ideal generated by these forms is contained in $\initwf{\wu{}}{J}\cO$.
  As $\initwf{\wu{}}{J}\cO$ is generated over $\cO$ by all elements   $\initwf{\wu{}}{g}$ with $g\in J$, the reverse inclusion will follow immediately if we show that
    $\initwf{\wu{}}{g} \in (\initwf{\wu{}}{f_1},\cdots, \initwf{\wu{}}{f_s}) \cO$.
This identity is a direct consequence of~\autoref{lm:geninit}.  Therefore, item~(\ref{item:initialFormsGenerate}) holds. This concludes our proof.
\end{proof}


 \bigskip
 
\noindent
\textbf{\small{Authors' addresses:}}
\smallskip
\

\noindent
\small{M.A.\ Cueto,  Mathematics Department, The Ohio State University, 231 W 18th Ave, Columbus, OH 43210, USA.
\\
\noindent \emph{Email address:} \url{cueto.5@osu.edu}}
\vspace{2ex}

\noindent
\small{P.\ Popescu-Pampu,
  Univ.~Lille, CNRS, UMR 8524 - Laboratoire Paul Painlev{\'e}, F-59000 Lille, France.
  \\
\noindent \emph{Email address:} \url{patrick.popescu-pampu@univ-lille.fr}}
\vspace{2ex}

\noindent
\small{D.\ Stepanov, Department of Higher Mathematics and Center of Fundamental Mathematics, Moscow Institute of Physics and Technology,
9 Institutskiy per., Dolgoprudny, Moscow, 141701, Russia.
  \\
\noindent \emph{Email address:} \url{stepanov.da@phystech.edu}}

\vspace{3ex}
\noindent
\small{J.\ Wahl, Department of Mathematics, The University of North Carolina at Chapel Hill, Chapel Hill, NC 27599, USA.
  \\
  \noindent \emph{Email address:} \url{jmwahl@email.unc.edu}}

\end{document}